\ifdefined\submit
  \RequirePackage{fix-cm}
  \documentclass[final]{svjour3}
  \smartqed
\else
  \documentclass[letter,11pt]{article}
  \usepackage{amsthm}
\fi

\usepackage[T1]{fontenc}
\ifdefined\submit
\else
  \usepackage{lmodern}
\fi

\usepackage{amsmath}
\usepackage{amssymb}
\usepackage{mathtools}
\usepackage{mathrsfs}
\usepackage{stmaryrd}
\usepackage{bbm}
\usepackage{bm}

\usepackage{graphicx}
\usepackage{float}
\usepackage{caption}
\usepackage{subcaption}
\usepackage{wrapfig}
\usepackage{booktabs}
\usepackage{tabularx}
\usepackage{bigstrut}
\usepackage{array,multirow,xcolor,colortbl}
\usepackage{threeparttable}

\newcommand{\Thick}{1.4pt} 
\newcommand{\HThick}{\noalign{\hrule height \Thick}}

\usepackage{epsfig,epstopdf}
\usepackage[shortlabels]{enumitem}
\usepackage{seqsplit}
\usepackage{soul}
\usepackage{color}
\usepackage{makecell}
\usepackage{algorithm}
\usepackage{algpseudocode}

\usepackage{tikz}
\usetikzlibrary{positioning,quotes}
\usepackage[framemethod=tikz]{mdframed}
\usepackage[breakable,most]{tcolorbox}

\newlength{\Rlen}
\ifdefined\submit
  \usepackage[sort&compress,numbers]{natbib}
  
\else
  \usepackage[numbers]{natbib}
\fi

\ifdefined\submit
  \usepackage[breaklinks=true,pdfstartview=FitH,hidelinks]{hyperref}
\else
  \definecolor{darkblue}{rgb}{0,0,.5}
  \usepackage[breaklinks=true,pdfstartview=FitH,pagebackref=true]{hyperref}
  \hypersetup{
    colorlinks=true,
    linkcolor=darkblue,
    citecolor=darkblue,
    urlcolor=darkblue
  }
  \usepackage{footnotebackref}
\fi

\usepackage{tablefootnote}
\ifdefined\submit
\else
  \usepackage[margin=1in]{geometry} 

  \usepackage{etoolbox}
  \makeatletter
    \patchcmd{\@maketitle}{\LARGE \@title}{\LARGE\bfseries\@title}{}{}
    \renewcommand{\@seccntformat}[1]{\csname the#1\endcsname.\quad}
  \makeatother

  \AtBeginDocument{
    \setlength{\abovedisplayskip}{4pt}
    \setlength{\belowdisplayskip}{4pt}
    \setlength{\abovedisplayshortskip}{2pt}
    \setlength{\belowdisplayshortskip}{2pt}
    \setlength{\jot}{1pt}
  }
  
  \allowdisplaybreaks
\fi

\ifdefined\submit
  \newcommand{\smartqedmark}{\qed}
\else
  \newcommand{\smartqedmark}{\qedhere}
  
\fi

\usepackage{aliascnt}
\usepackage[capitalize,nameinlink]{cleveref}

\crefname{appendix}{Appendix}{Appendices}
\Crefname{appendix}{Appendix}{Appendices} 
\newcommand{\newaliasedtheorem}[2]{%
  \newaliascnt{#1}{theorem}%
  \newtheorem{#1}[#1]{#2}%
  \aliascntresetthe{#1}%
}

\ifdefined\submit
\else
\fi

\ifdefined\submit
\newaliasedtheorem{fact}{Fact}
\newtheorem{assumption}{Assumption}
\renewcommand{\theassumption}{\Alph{assumption}}
\else 
\newtheorem{theorem}{Theorem}[section]
\newaliasedtheorem{lemma}{Lemma}
\newaliasedtheorem{corollary}{Corollary}
\newaliasedtheorem{proposition}{Proposition}
\newaliasedtheorem{fact}{Fact}
\newtheorem{assumption}{Assumption}
\renewcommand{\theassumption}{\Alph{assumption}}
\theoremstyle{definition}

\newaliasedtheorem{definition}{Definition}
\newaliasedtheorem{example}{Example}
\newaliasedtheorem{remark}{Remark}
\newaliasedtheorem{problem}{Problem}
\fi 

\crefname{theorem}{Theorem}{Theorems} 
\crefname{lemma}{Lemma}{Lemmas} 
\crefname{corollary}{Corollary}{Corollaries} 
\crefname{proposition}{Proposition}{Propositions} 
\crefname{assumption}{Assumption}{Assumptions}

\newlist{assumptionlist}{enumerate}{1}
\setlist[assumptionlist]{
	label=\textup{(\roman*)},
	ref=\theassumption(\roman*)
}

\crefname{assumptionlisti}{Assumption}{Assumptions}

\crefname{definition}{Definition}{Definitions}
\crefname{example}{Example}{Examples}
\crefname{remark}{Remark}{Remarks}
\crefname{problem}{Problem}{Problems}

\crefformat{equation}{\textup{#2(#1)#3}}
\crefrangeformat{equation}{\textup{#3(#1)#4--#5(#2)#6}}
\crefmultiformat{equation}{\textup{#2(#1)#3}}{ and \textup{#2(#1)#3}}
  {, \textup{#2(#1)#3}}{, and \textup{#2(#1)#3}}
\crefrangemultiformat{equation}{\textup{#3(#1)#4--#5(#2)#6}}%
  { and \textup{#3(#1)#4--#5(#2)#6}}{, \textup{#3(#1)#4--#5(#2)#6}}{, and \textup{#3(#1)#4--#5(#2)#6}}

\Crefformat{equation}{#2Equation~\textup{(#1)}#3}
\Crefrangeformat{equation}{Equations~\textup{#3(#1)#4--#5(#2)#6}}
\Crefmultiformat{equation}{Equations~\textup{#2(#1)#3}}{ and \textup{#2(#1)#3}}
  {, \textup{#2(#1)#3}}{, and \textup{#2(#1)#3}}
\Crefrangemultiformat{equation}{Equations~\textup{#3(#1)#4--#5(#2)#6}}%
  { and \textup{#3(#1)#4--#5(#2)#6}}{, \textup{#3(#1)#4--#5(#2)#6}}{, and \textup{#3(#1)#4--#5(#2)#6}}

\crefdefaultlabelformat{#2\textup{#1}#3}

\numberwithin{equation}{section}

\newtcolorbox{todolist}{
  enhanced,
  breakable,
  title=\textit{To-do List},
  fonttitle=\bfseries,
  colback=gray!5,
  colframe=gray!60!black,
  boxrule=0.8pt,
  arc=3pt,
  left=6pt,right=6pt,top=6pt,bottom=6pt
}

\DeclareMathOperator*{\argmin}{argmin}

\DeclareMathOperator*{\Span}{span}
\DeclareMathOperator*{\Fix}{Fix}

\DeclareMathOperator*{\interior}{int}

\DeclareMathOperator*{\ran}{ran}
\DeclareMathOperator*{\dom}{dom}
\DeclareMathOperator*{\gra}{gra}
\DeclareMathOperator*{\Id}{Id}
\DeclareMathOperator*{\zer}{zer}
\DeclareMathOperator{\diag}{diag}

\renewcommand{\Re}{\mathbb{R}}
\newcommand{\bR}{\mathbf{R}}

\def \rla{\right\rangle}
\def \lla{\left\langle}
\newcommand{\inner}[2]{\lla #1, #2 \rla}
\newcommand{\norm}[1]{{\left\|{#1}\right\|}}

\renewcommand{\H}{\mathcal{H}}
\newcommand{\K}{\mathcal{K}}
\newcommand{\bcalH}{\bm{\mathcal{H}}}
\newcommand{\bcalK}{\bm{\mathcal{K}}}

\newcommand{\bA}{\mathbf{A}}
\newcommand{\bB}{\mathbf{B}}

\newcommand{\bD}{\mathbf{D}}

\newcommand{\bL}{\mathbf{L}}
\newcommand{\bM}{\mathbf{M}}
\newcommand{\bN}{\mathbf{N}}
\newcommand{\bP}{\mathbf{P}}
\newcommand{\bQ}{\mathbf{Q}}
\renewcommand{\bR}{\mathbf{R}}
\newcommand{\bU}{\mathbf{U}}
\newcommand{\bV}{\mathbf{V}}
\newcommand{\bW}{\mathbf{W}}
\newcommand{\bX}{\mathbf{X}} 
\newcommand{\bY}{\mathbf{Y}} 
\newcommand{\bZ}{\mathbf{Z}}

\newcommand{\bOmega}{\mathbf{\Omega}}

\newcommand{\bXi}{\mathbf{\Xi}}
\newcommand{\bTheta}{\mathbf{\Theta}}
\newcommand{\hatOmega}{\mathbf{\Omega_b}}

\newcommand{\ba}{\mathbf{a}}
\newcommand{\bb}{\mathbf{b}}

\newcommand{\bq}{\mathbf{q}}

\newcommand{\bs}{\mathbf{s}}
\newcommand{\bu}{\mathbf{u}}
\newcommand{\bv}{\mathbf{v}}
\newcommand{\bw}{\mathbf{w}}
\newcommand{\bx}{\mathbf{x}}
\newcommand{\by}{\mathbf{y}}
\newcommand{\bz}{\mathbf{z}}

\newcommand{\toweak}{\rightharpoonup}
\newcommand{\toset}{\rightrightarrows}

\def\TheTitle{A Splitting Framework for Composite Semimonotone Inclusions}

\ifpdf
\hypersetup{
  pdftitle={\TheTitle},
  pdfauthor={Jan Harold Alcantara, Minh N. Dao, and Akiko Takeda}
}
\fi

\ifdefined\submit
  \title{\TheTitle\thanks{Version of \today.}}
  \titlerunning{Splitting Methods for Semimonotone Inclusions}
  \author{Jan Harold Alcantara \and Minh N. Dao \and Akiko Takeda}
  \institute{*Corresponding author: Jan Harold Alcantara \at
    Center for Advanced Intelligence Project, RIKEN, Tokyo, Japan.\\
    \email{\url{janharold.alcantara@riken.jp}}\\ \text{}\\
    Minh N. Dao \at
    School of Science, RMIT University, Melbourne, VIC 3000, Australia.\\
    \email{\url{minh.dao@rmit.edu.au}}\\ \text{}\\
    Akiko Takeda \at
    Department of Mathematical Informatics, Graduate School of Information Science and Technology, University of Tokyo,
    Tokyo, Japan, and Center for Advanced Intelligence Project, RIKEN, Tokyo, Japan.\\
    \email{\url{takeda@mist.i.u-tokyo.ac.jp}}
  }
  \date{}
\else
  \title{\TheTitle}
  \date{August 17, 2026}
  \author{
    Jan Harold Alcantara\thanks{\url{janharold.alcantara@riken.jp}. Center for Advanced Intelligence Project, RIKEN, Tokyo, Japan.}
    \qquad Minh N. Dao\thanks{\url{minh.dao@rmit.edu.au}. School of Science, RMIT University, Melbourne, VIC 3000, Australia.}
    \qquad Akiko Takeda\thanks{\url{takeda@mist.i.u-tokyo.ac.jp}. Department of Mathematical Informatics, Graduate School of Information Science and Technology, University of Tokyo,
    Tokyo, Japan, and Center for Advanced Intelligence Project, RIKEN, Tokyo, Japan.}
  }
\fi

\begin{document}

\maketitle

\begin{abstract}
We introduce a general framework for composite inclusion problems with affine constraints, covering both monotone and semimonotone regimes. The central idea is a new interpretation of the constrained inclusion through an operator-vector pair that separates the implicit inclusion from the affine constraint: the former is handled through possibly preconditioned resolvent evaluations, while the latter is handled through an explicit forward step in an auxiliary variable. This yields a single abstract iteration applicable to multioperator inclusions, linearly coupled inclusions, and block-separable inclusions with affine constraints. The freedom in choosing the operator-vector pair enables the systematic construction of problem-adapted splitting algorithms, including new schemes for several important problem classes. Exploiting the orthogonal decomposition induced by the constraint subspace, we develop a unified and streamlined convergence analysis and establish weak and strong convergence guarantees under semimonotonicity assumptions. When specialized to multioperator inclusions, the framework permits general bounded linear operator coefficients, rather than only scalar coefficients, and therefore accommodates preconditioned resolvents. The resulting schemes recover several existing methods while extending them to previously uncovered regimes, and in several important cases, require weaker assumptions and admit provably larger admissible parameter ranges.
\end{abstract}


\noindent{\bf Mathematics Subject Classification (MSC 2020):}
47H05,  
47J22,  
47J25,  
65K10.  

\section{Introduction}\label{sec:intro}

Let $\bcalH$, $\bcalK$ and $\bcalH'$ be real Hilbert spaces. We consider the inclusion problem
\begin{equation}
	\text{find } \bx \in \bcalH   \text{ such that } 0 \in \bA \bx + \bP \bB (\bR \bx) +N_{\hatOmega}(\bx) ,
	\label{eq:productspace}
\end{equation}
where $\bA : \bcalH \rightrightarrows \bcalH$ is a set-valued operator, $\bB: \bcalK \to \bcalK $  is a single-valued mapping, $\bP:\bcalK\to\bcalH$ and $\bR : \bcalH\to\bcalK$ are bounded linear operators, and $N_{\hatOmega}$ is the normal cone to the affine subspace
\begin{equation*}
	\hatOmega  = \{ \bx \in \bcalH: \bM^* \bx = \bf{b}\},
\end{equation*}
where $\bM^*$ is the adjoint of $\bM: \bcalH' \to \bcalH$, a bounded linear operator, and $\bf{b}\in \bcalH'$. Note that $\hatOmega$ can be expressed as 
	\begin{equation}
		\hatOmega = \bar{\bx} + \bOmega ,\qquad \bOmega \coloneqq \ker(\bM^*),
		\label{eq:hatOmega}
	\end{equation}
where $\bar{\bx}\in \hatOmega$ is arbitrary. We are particularly interested in the case where $\bA$ is accessed through its resolvent, while $\bB$ is treated by direct evaluation. In this paper, we study the setting in which $\bA$ is semimonotone, $\bB$ is cocoercive, and $\bP^*$ and $\bR$ agree on $\bOmega$; see \cref{assume:blanket} for the precise assumptions. Throughout this paper, we assume that \eqref{eq:productspace} has a solution.

\subsection{Problem classes motivating the framework}

The inclusion problem \eqref{eq:productspace} covers broad problem classes, each of which is itself quite general. The corresponding product space embeddings into \eqref{eq:productspace} are summarized in \cref{tab:productspace_embeddings}.

\subsubsection{Multioperator inclusion (without coupling)}

Our main motivation for the problem setting \eqref{eq:productspace} is the multioperator inclusion 
		\begin{equation}\label{eq:gen_prob_nocoupling}
		\text{find~} x\in \H \text{~such that~} 0 \in \sum_{i=1}^n A_i x + \sum_{j=1}^p B_j x, 
		\tag{P1}
	\end{equation}
where $A_i\colon \H \rightrightarrows \H$ is maximally semimonotone, $B_j\colon \H \to \H$ is cocoercive, and $\H$ is a real Hilbert space. 

Splitting methods for solving \eqref{eq:gen_prob_nocoupling} are well developed in the monotone setting. When $p=0$, the Douglas--Rachford method \cite{douglas_rachford_1956} plays a foundational role in the two-operator case, while for $n=3$, Ryu \cite{ryu_2020} introduced a three-operator resolvent splitting together with the notion of \emph{frugal resolvent splittings}. This line of work was extended to arbitrary $n\ge 2$ by Malitsky and Tam \cite{MT23}, and frugal \emph{parametrized} resolvent splittings, allowing resolvents of scaled $A_i$'s, were later studied in \cite{aragonartacho_bot_torregrosabelen_2023}. In the presence of cocoercive terms, representative examples include the forward-backward method for $(n,p)=(1,1)$ \cite{LM79}, Davis--Yin splitting for $(n,p)=(2,1)$ \cite{DY17}, generalized forward-backward splitting for arbitrary $n$ with $p=1$ \cite{RFP13}, and distributed forward-backward methods for general $n$ and $p$ \cite{AMTT23}. More general monotone frameworks were developed later: For $p=0$, Tam \cite{Tam24} introduced a unified framework for frugal and decentralized resolvent splittings for arbitrary $n$, while graph-based Douglas--Rachford constructions were studied in \cite{BCN24}; in the presence of forward terms, graph-based and distributed extensions were developed in \cite{ACL24,ACGN25,DTT26}. These works cover broad classes of monotone multioperator algorithms, but remain centered on splittings built from standard resolvents and linear combinations with scalar coefficients. 

Beyond the monotone setting, however, the available results are much more limited: under generalized monotonicity, comonotonicity, or, more generally, semimonotonicity, existing works are restricted to particular algorithms, including Douglas--Rachford for $(n,p)=(2,0)$ \cite{EPLP25,DP18,BDP22}, Davis--Yin splitting for $(n,p)=(2,1)$ \cite{DP21}, and multioperator Douglas--Rachford methods for arbitrary $n$ with $p=0$ \cite{AT25,ADT25}, based on reduced product space reformulations \cite{Cam22,Kruger1985}. To the best of our knowledge, no general framework is currently available for either the generalized monotone or the comonotone setting, both of which fall within the broader class of semimonotone operators.

\subsubsection{Multioperator inclusion with linear coupling}
An extended class of problems that includes \eqref{eq:gen_prob_nocoupling} is given by 
	\begin{equation}\label{eq:gen_prob}
		\text{find~} x\in \H \text{~such that~} 0 \in \sum_{i=1}^n A_i x + \sum_{\ell=1}^r D_\ell^* C_\ell(D_\ell x) + \sum_{j=1}^p B_j x, \tag{P2}
	\end{equation}
where $C_\ell\colon \K \rightrightarrows \K$ is maximally semimonotone,
	$D_\ell\colon \H \to \K$ is bounded and linear, and $\H,\K$ are real Hilbert spaces. This class arises naturally, for instance, from optimization problems of the form
	\[
	\min_{x\in\H}\;
	\sum_{i=1}^n f_i(x) + \sum_{\ell=1}^r g_\ell(D_\ell x) + \sum_{j=1}^p h_j(x),
	\]
	where the terms $f_i$ and $g_\ell$ are handled through resolvents, while the terms $h_j$ are treated explicitly through their cocoercive gradients.
	
	In the monotone setting, many primal-dual algorithms have been proposed, primarily for particular parameter regimes. For example, when $(n,r,p)=(1,1,0)$, one of the most prominent methods is the Chambolle--Pock algorithm \cite{Chambolle2011FirstOrderPrimalDual}. Other primal-dual schemes treating cases with $n\in\{0,1\}$ and $(r,p)=(1,1)$ were studied in \cite{ChenHuangZhang2013,DroriSabachTeboulle2015,CombettesPesquet2012,Condat2013,LatafatPatrinos2017}. When $n\geq 2$ and $r\geq 2$, \cite{aragonartacho_bot_torregrosabelen_2023} proposed an algorithm for $p=0$, with multiple linear compositions handled through a product space reformulation, while \cite{DPTT25} further developed a general framework that also accommodates $p\geq 1$, with the single-valued operators being either cocoercive or monotone and Lipschitz continuous. Beyond the monotone setting, the available results appear to be even more limited, as only the Chambolle--Pock algorithm has so far been studied under semimonotonicity \cite{ELP25}, to the best of our knowledge.
	
\subsubsection{Block-separable inclusion with affine constraints}
	Consider the variational inequality with semimonotone operators 
	\begin{align}
		& \text{find~} (x_1, \dots, x_n)\in \H_1\times \dots \times \H_n \text{~such that~} \notag \\
		& \qquad 0\in A_1x_1\times \dots \times A_nx_n +N_{\hatOmega}(x_1, \dots, x_n),    
		\label{eq:VI} \tag{P3}
	\end{align}
	where $A_i\colon \H_i\rightrightarrows \H_i$ is maximally semimonotone and $\hatOmega \coloneqq \{(x_1, \dots, x_n)\in \H_1\times \dots \times \H_n: \sum_{i=1}^nD_ix_i =b\}$ with $D_i\colon \H_i\to \H$ a bounded linear operator and $b\in \H$. This class arises, for example, from linearly constrained multi-block problems of the form
	\begin{equation}
		\min_{x_i \in \H_i} \sum_{i=1}^n f_i(x_i) \text{~subject to~} \sum_{i=1}^n D_ix_i = b,  
		\label{eq:multi-block_opt}  
	\end{equation}
	where $f_i\colon \H_i\to (-\infty, +\infty]$ is a proper lower semicontinuous function with semimonotone subdifferentials. 

For the particular linearly constrained multi-block optimization model \eqref{eq:multi-block_opt}, a standard route is to pass to the KKT system using inverses of subdifferential operators \cite{ADT25,MT23}. This recasts the problem into a multioperator inclusion \eqref{eq:gen_prob_nocoupling} with $p=0$. In the monotone case, convergence is known for the Malitsky--Tam algorithm \cite{MT23}. For the generalized monotone case, the multioperator Douglas--Rachford algorithm is studied in \cite{ADT25}.

\subsection{Contributions}

We develop a unified framework for splitting methods that provides a
common structural viewpoint for several classes of composite inclusion
problems in both the monotone and semimonotone settings. In addition to this unifying
role, the framework provides a systematic mechanism for deriving new
splitting algorithms. The multioperator inclusion
\eqref{eq:gen_prob_nocoupling} remains our main motivating problem. 
The main contributions are as follows.

\begin{enumerate}
	
	\item We introduce a new interpretation of splitting for affine-constrained
	inclusions through a preconditioner-like operator-vector pair; see
	\cref{assume:preconditioner}. The implicit inclusion is handled through
	possibly preconditioned resolvent evaluations, while the affine
	constraint is handled through an explicit forward step in an auxiliary
	variable. This leads to the abstract iteration
	\cref{algo:abstract_preconditioned_forward}. The freedom in choosing the operator-vector pair also leads to
		\textit{new splitting schemes}, as illustrated for
		\eqref{eq:gen_prob} and \eqref{eq:VI} in \cref{sec:realizations}; see \cref{ex:realization_P2,ex:realization_P3,ex:realization_P3_multiblock}.
	
	\item We establish convergence guarantees for the abstract scheme beyond the monotone setting. Our analysis covers two regimes: one in which a nonzero cocoercive forward operator $\bB$ is allowed, and another in which the forward term is absent but a broader semimonotone structure on $\bA$ can be incorporated. In both regimes, the sequence generated by the proposed algorithm converges weakly. Moreover, the corresponding shadow sequence converges weakly to a solution of the problem, while strong convergence is obtained under a restricted strong monotonicity assumption on the associated modulus of monotonicity.

	\item A central technical contribution is the proof strategy. The analysis is
	organized around the decomposition $\bOmega\oplus\bOmega^\perp$, reducing the
	main arguments to understanding how the relevant operators act on the subspaces. This yields a modular proof of the abstract
	theorem. The underlying viewpoint is
new even in the monotone setting, and the proof itself reveals how both monotone
and semimonotone classes can be treated within a single analysis.

	\item For our main motivational problem \eqref{eq:gen_prob_nocoupling}, we
	exploit the additional consensus structure of $\bOmega$ (see \cref{tab:productspace_embeddings}) to obtain general algorithms, including \cref{algo:full_implicit,algo:full}. In contrast to
	existing approaches, the resulting framework allows general, possibly
	non-scalar bounded linear operators in the splitting scheme%
	\footnote{In particular, $M_{ij}$, $N_{ij}$, $P_{ij}$, $R_{ij}$, and $D_i$
		in \cref{algo:full_implicit,algo:full} can be non-scalar operators.}
	and therefore accommodates preconditioned resolvents. In this specialization,
	we also obtain sharper guarantees in several important cases, including
	weaker assumptions and larger admissible parameter ranges than those
	available in closely related works; detailed comparisons are given in
	\cref{sec:comparisons}. To the best of our knowledge,
this is the first general splitting framework for multioperator inclusions
in the semimonotone setting; existing general frameworks are developed in
the monotone setting, whereas available semimonotone results concern
particular splitting methods. 
\end{enumerate}

\subsection{Organization of the paper}

The remainder of the paper is organized as follows. In \cref{sec:preliminaries}, we collect the basic material on bounded linear operators and semimonotone operators that will be used throughout the paper. In \cref{sec:gen_subspace}, we consider the abstract framework \eqref{eq:productspace}, formulate the associated splitting schemes, and state the main convergence theorem. In \cref{sec:multi_inclusion}, we show how the original multioperator inclusion problem \eqref{eq:gen_prob_nocoupling} fits into this framework and derive the corresponding convergence results as consequences of the abstract theory. In \cref{sec:comparisons}, we compare the present framework with existing works and highlight, in particular, the additional generality and sharper guarantees obtained in several important special cases. The proof of the main abstract convergence theorem is given in \cref{sec:proof_mainresult}. Finally, the appendices collect auxiliary results on linear operators and semimonotone operators, as well as several omitted proofs used in the main text.

\section{Preliminaries}\label{sec:preliminaries}

In this section, we briefly recall standard notation, definitions, and basic results concerning bounded linear and semimonotone operators on real Hilbert spaces. To distinguish these generic objects from the product space quantities used in the main development, we use non-bold symbols for the Hilbert spaces, operators, vectors, and sets appearing in this section.

Let $\H$ be a real Hilbert space with inner product $\langle \cdot,\cdot\rangle$ and induced norm $\|\cdot\|$. The set of nonnegative integers is denoted by $\mathbb{N}$ and the set of real numbers by $\mathbb{R}$. We denote $t_+ \coloneqq \max\{0,t\}$ for $t\in \Re$. We use the convention $\frac{a}{0}=+\infty$ for every $a>0$. We write $A\colon \H\rightrightarrows \H$ to indicate that $A$ is set-valued and $A\colon \H\to \H$ to indicate that $A$ is single-valued. We denote by $\Id_{\H}$ the identity operator on $\H$, and we omit the subscript if the underlying space is clear. For a sequence $\{x^k\}\subseteq \H$, we write $x^k\to x$ to denote strong convergence and $x^k\toweak x$ to denote weak convergence.

\subsection{Bounded linear operators}
\label{sec:boundedlinear_definitions}

We collect basic definitions and results on linear operators. Unless otherwise stated, we refer to \cite[Chapters 2 and 3]{BC17}. Let $A:\H\to \K$ be a linear operator, where $\H$ and $\K$ are real Hilbert spaces. We say that $A$ is \textit{bounded} and write $A\in \mathcal{B}(\H,\K)$ if $
\norm{A} \coloneqq \sup_{\norm{x}=1}\norm{Ax}<+\infty.
$
When $\H=\K$, we simply write $A\in \mathcal{B}(\H)$. Bounded linear operators are weakly continuous: if $x^k\toweak x$, then $Ax^k\toweak Ax$. 

An operator $A\in \mathcal{B}(\H)$ is \textit{self-adjoint} if $A^*=A$. The self-adjoint part of $A\in\mathcal{B}(\H)$ is denoted by $A_s\coloneqq \frac{1}{2}(A+A^*)$.
 The maximum and minimum spectral values of a self-adjoint operator $A$ are given by
\[
\lambda_{\max}(A)\coloneqq \sup_{\norm{x}=1}\inner{Ax}{x},
\qquad
\lambda_{\min}(A)\coloneqq \inf_{\norm{x}=1}\inner{Ax}{x},
\]
which are finite when $A\in \mathcal{B}(\H)$. 

$A\in \mathcal{B}(\H)$ is \textit{monotone on $D\subseteq \H$} if $\inner{Ax}{x}\geq 0$ for all $x\in D$. When $D=\H$, we write $A\succeq 0$ and say that $A$ is monotone. For $A,B\in \mathcal{B}(\H)$, we write $A\succeq B$ if $A-B\succeq 0$. We say that $A$ is \textit{$\alpha$-strongly monotone on $D$} if $\alpha>0$ and $\inner{Ax}{x}\geq \alpha \norm{x}^2$ for all $x\in D$. When $D=\H$, we simply say that $A$ is $\alpha$-strongly monotone. If $A\in \mathcal{B}(\H)$ is $\alpha$-strongly monotone and self-adjoint, $\norm{x}_{A} \coloneqq \sqrt{\inner{x}{Ax}}$ is an equivalent norm on $\H$; in particular,  $\alpha \norm{x}^2\leq \norm{x}_A ^2\leq \norm{A} \, \norm{x}^2$. Moreover, $A^{-1}$ is $\norm{A}^{-1}$-strongly monotone, and $\norm{x}_{A^{-1}}$ is a norm satisfying $\norm{A}^{-1} \, \norm{x}^2 \leq \norm{x}_{A^{-1}}^2 \leq \alpha^{-1} \, \norm{x}^2$.

We recall Young's inequality. 

\begin{fact}[Young's inequality]
	\label{lemma:Young}
	Let $W\in \mathcal{B}(\H)$ be a self-adjoint strongly monotone operator. Then 
	\[ |\inner{x}{y}|\leq \frac{1}{4}\norm{x}^2_{W} + \norm{y}_{W^{-1}}^2\]
	for all $x,y\in \H$.
\end{fact}
\begin{proof}
	Since $W$ is $\alpha$-strongly monotone, $W^{-1/2}$ is well-defined. Then $$|\inner{x}{y}| = |\inner{W^{1/2}x}{W^{-1/2}y}|\leq \norm{W^{1/2}x} \cdot \norm{W^{-1/2}y} \leq \frac{\varepsilon}{2}\norm{W^{1/2}x}^2 + \frac{1}{2\varepsilon}\norm{W^{-1/2}y}^2$$
	for any $\varepsilon>0$. Taking $\varepsilon = 1/2$ gives the desired result. 
\smartqedmark \end{proof}

For $A\in \mathcal{B}(\H,\K)$, the kernel and range are denoted by $\ker(A)$ and $\ran(A)$. We have the following relationships:
\begin{equation}
    \ker(A)^\perp = \overline{\ran(A^*)}
\qquad\text{and}\qquad
\ran(A)^\perp = \ker(A^*),
\label{eq:kernelrange_relationship}
\end{equation}
where $ \overline{\ran(A^*)}$ is the closure of $\ran(A^*)$.
Given a closed linear subspace $\Omega\subseteq \H$, we denote by $\Pi_{\Omega}\in \mathcal{B}(\H)$ the projector onto $\Omega$, i.e., $\Pi_{\Omega}x = \argmin_{y\in \Omega}\norm{x-y}$, which is a self-adjoint operator. If $A\in \mathcal{B}(\H,\K)$ has closed range, $\Pi_{\ran (A)} = AA^{\dagger}$, where $A^\dagger$ denotes the Moore--Penrose inverse of $A$. 

Let $\Omega\subseteq \H$ be a closed linear subspace. $\H$ admits the orthogonal decomposition $\H = \Omega \oplus \Omega^\perp$. Given $A\in \mathcal{B}(\H)$, we write its block decomposition (see \cite[Chapter II]{Conway}) with respect to $\H = \Omega\oplus \Omega^\perp$ as
\begin{equation}
A = \begin{pmatrix}
A_{11} & A_{12} \\
A_{21} & A_{22}
\end{pmatrix}
\coloneqq 
\begin{pmatrix}
\Pi_{\Omega}A\Pi_{\Omega} & \Pi_{\Omega}A\Pi_{\Omega^\perp} \\
\Pi_{\Omega^\perp}A\Pi_{\Omega} & \Pi_{\Omega^\perp}A\Pi_{\Omega^\perp}
\end{pmatrix}.
\label{eq:2x2decomposition}
\end{equation}
We identify the diagonal blocks $A_{11}$ and $A_{22}$ with their
restrictions to $\Omega$ and $\Omega^\perp$, respectively. Accordingly,
properties such as strong monotonicity and invertibility of $A_{11}$
(resp.\ $A_{22}$) are understood on $\Omega$ (resp.\ $\Omega^\perp$). When appearing in operator expressions on $\mathcal H$, inverses and
Moore--Penrose inverses of these blocks are understood to be extended
by zero on the complementary subspace.

If $A$ is self-adjoint, then $A_{12}=A_{21}^*$. The generalized Schur complements of $A_{11}$ and $A_{22}$ are defined by
\begin{align*}
A/A_{11} &  \coloneqq A_{22} - A_{21} A_{11}^\dagger A_{12},  \quad \text{and} \quad  
A/A_{22} \coloneqq A_{11} - A_{12} A_{22}^\dagger A_{21}.
\end{align*}
\begin{fact}\label{lemma:schurcomplementlemma}
Let $A\in \mathcal{B}(\H)$ be self-adjoint with block decomposition as in \eqref{eq:2x2decomposition}. Then the following hold:
\begin{enumerate}
	\item If $A_{11}$ is strongly monotone, then $A\succeq 0$ if and only if $A/A_{11}= A_{22} - A_{21} A_{11}^{-1}A_{12}\succeq 0$.
	\item If $A_{22}$ has closed range, then $A\succeq 0$ if and only if $A_{22}\succeq 0$, $A/A_{22}=A_{11}-A_{12}A_{22}^\dagger A_{21}\succeq 0$, and $\ran(A_{21})\subseteq \ran(A_{22})$. 
\end{enumerate}
\end{fact}
\begin{proof}
	 Part (i) is from \cite[Theorem 5.1]{MoslehianKianXu2019Positivity}. For part (ii), we note from \cite[Theorem 2.1]{Tarcsay2011} that $\ran(A_{22})=\ran(A_{22}^{1/2})$. By \cite[Lemma 2.3]{CONTINO2019214}, $A\succeq 0$ if and only if $A_{22}\succeq 0$, $A_{11}-D^*D\succeq 0$, and $\ran(A_{21})\subseteq \ran(A_{22})$, where $D$ satisfies $A_{21}=A_{22}^{1/2}D$. Taking $D= ((A_{22})^{1/2})^\dagger A_{21}$ gives the desired result. 
\smartqedmark \end{proof}

\subsection{Semimonotone operators}

Let $A:\H \toset \H$ be an operator. The \emph{domain} of $A$ is $\dom A :=\{x\in \H: Ax\neq \varnothing\}$, the \emph{graph} of $A$ is $\gra A :=\{(x,u)\in \H\times \H: u\in Ax\}$, the set of \emph{fixed points} of $A$ is $\Fix A :=\{x\in \H: x\in Ax\}$, and the set of \emph{zeros} of $A$ is $\zer(A) \coloneqq \{x\in \H: 0\in Ax \}$. The \emph{resolvent} of $A$ is defined by $J_A :=(\Id +A)^{-1}.$
If $D$ is a strongly monotone operator, $J_{D^{-1}A} = (\Id+ D^{-1}A)^{-1}$ is called a \emph{preconditioned resolvent}. The \textit{normal cone} to a closed convex set $\Omega$ is given by 
\[ N_{\Omega}(x) = \begin{cases}
    \{ z\in \H : \inner{z}{y-x}\leq 0~\forall y\in\Omega\} & \text{if}~x\in\Omega, \\
    \varnothing & \text{otherwise}.
\end{cases} \]
When $\Omega$ is a closed linear subspace, $N_{\Omega}(x) = \Omega^\perp$ for any $x\in \Omega$.
    
We consider the class of semimonotone operators defined in \cite{ELP25,EPLP25}, which is a generalization of $\alpha$-monotone and $\alpha$-comonotone operators considered in \cite{Bauschke2021,BDP22,DP18}.

\begin{definition}[Cocoercivity and semimonotonicity]
Let $\H$ be a real Hilbert space.
\begin{enumerate}
\item \textbf{$(U,V)$-semimonotonicity.}
Let $A:\H\rightrightarrows \H$ be set-valued and let $U,V\in\mathcal B(\H)$ be self-adjoint.
We say that $A$ is \emph{$(U,V)$-semimonotone} if
\[
\langle x-x',\, y-y'\rangle
\;\ge\;
\langle x-x',\, U(x-x')\rangle
+
\langle y-y',\, V(y-y')\rangle
\qquad
\forall (x,y),(x',y')\in\gra(A).
\]
Special cases: $A$ is called \emph{$U$-monotone} if it is $(U,0)$-semimonotone,
\emph{$V$-comonotone} if it is $(0,V)$-semimonotone, and \emph{monotone} if it is $(0,0)$-semimonotone.

\item \textbf{Maximal $(U,V)$-semimonotonicity.}
We say that $A$ is \emph{maximally $(U,V)$-semimonotone} if it is $(U,V)$-semimonotone and there is no
$(U,V)$-semimonotone operator $\widetilde A:\H\rightrightarrows\H$ such that $
\gra(A)\subsetneq \gra(\widetilde A).$
(The analogous terminology applies to the special cases above.)
\item \textbf{$W^{-1}$-cocoercivity.}
Let $B:\H\to \H$ be single-valued and let $W\in\mathcal B(\H)$ be self-adjoint and strongly monotone.
We say that $B$ is \emph{$W^{-1}$-cocoercive} if
\[
\langle x-x',\, Bx-Bx'\rangle \;\ge\; \norm{Bx-Bx'}_{W^{-1}}^{2}
\qquad \forall x,x'\in\H.
\]
That is, $B$ is $(0,W^{-1})$-semimonotone. 
If $W=\frac1\beta \Id$ for some $\beta>0$, then $B$ is called \emph{$\beta$-cocoercive}.
\end{enumerate}
\end{definition}

We recall important properties related to maximally monotone operators. 
\begin{fact}
\label{lemma:maximalmonotone_properties}
    Let $A:\H\toset \H$ be a monotone operator, and let $\gamma >0$.
    \begin{enumerate}
        \item $\dom (J_{\gamma A}) = \ran( \Id+\gamma A) =\H$ if and only if $A$ is maximally monotone. Moreover, $J_{\gamma A}$ is firmly nonexpansive (i.e., $1$-cocoercive).
        \item If $A:\H\to\H$ is continuous, then it is maximally monotone.  
        \item If $A$ is maximally monotone, then $A^{-1}$ is maximally monotone.
        \item If $A$ and $B$ are maximally monotone operators such that $\interior (\dom (A)) \cap \dom (B) \neq \varnothing$, then $A+B$ is maximally monotone.
        \item If $S$ is a self-adjoint strongly monotone operator and $A$ is maximally monotone, then $S^{-1}A$ is maximally monotone when $\H$ is endowed with the inner product $\inner{\cdot}{\cdot}_{S}$.
    \end{enumerate}
\end{fact}
\begin{proof}
   (i), (ii), (iii) and (v) hold by Proposition 23.8,  Corollary 20.28, Proposition 20.22, and Proposition 20.24 of \cite{BC17}, respectively. Part (iv) is from \cite{Rockafellar1970}. 
\smartqedmark \end{proof}

The following weak sequential closedness result will be useful in our analysis.

\begin{fact}
    \label{lemma:demiclosednessprinciple}
    Let $\Omega\subseteq \H$ be a closed linear subspace, $\bar{x}\in \H$ and $T:\H \toset \H$ be maximally monotone. Let $\{(u^k,v^k)\}$ be a sequence such that $v^k \in Tu^k$ and let $u^*,v^*\in \H$. Suppose that 
    \begin{equation}
    	u^k \toweak u^*, \quad v^k \toweak v^*, \quad \Pi_{\Omega}u^k \to 0, \quad \text{and} \quad \Pi_{\Omega^\perp}(v^k-\bar{x})\to 0,
    	\label{eq:demiclosedness_conditions}
    \end{equation}
    then $v^*\in Tu^*$.
\end{fact}
\begin{proof}
    Let $C= \Omega^\perp$ and $D= \bar{x} + \Omega$, so that $D-D = \Omega = (C-C)^\perp$. Meanwhile, $\Id - \Pi_C = \Pi_{C^\perp} = \Pi_\Omega$ and $(\Id - \Pi_D )v= v - (\bar{x} + \Pi_{\Omega}(v -\bar{x})) = \Pi_{\Omega^\perp}( v-\bar{x})$. Thus, by \eqref{eq:demiclosedness_conditions}, it follows that $(\Id - \Pi_C) u^k \to 0$ and $(\Id - \Pi_D)v^k \to 0$. Therefore, the conclusion holds by invoking \cite[Proposition 20.60]{BC17}.
\smartqedmark \end{proof}

\section{Proposed algorithm}\label{sec:gen_subspace}

In this section, we present our proposed algorithm and give an overview of the main result. The proofs and technical details are developed in \cref{sec:fundamentalinequality,sec:weak_z,sec:weak_x,sec:strong_x}.

\subsection{Reformulation and algorithm}

We consider the inclusion \eqref{eq:productspace}, namely
\begin{equation}\label{eq:productspace_recall}
    \text{find } \bx \in \bcalH \text{ such that } 0 \in \bA \bx + \bP \bB (\bR \bx) + N_{\hatOmega}(\bx),
\end{equation}
where $\bOmega_\bb = \{ \bx \in \bcalH : \bM^* \bx = \bb\}$ and $\bOmega = \ker (\bM^*)$; see also \eqref{eq:hatOmega}.  The problem classes \cref{eq:gen_prob_nocoupling,eq:gen_prob,eq:VI} fit the  unified formulation \eqref{eq:productspace_recall} under appropriate choices of $\bA,\bB,\bP,\bR,\bOmega_\bb$ as described in \cref{tab:productspace_embeddings}, subject to
the standing assumptions below.

We impose the following standing assumptions on \eqref{eq:productspace_recall}.
\begin{assumption}[Problem setting assumptions]\label{assume:blanket} The (nonlinear) operators $\bA$ and $\bB$, and linear operators $\bP$, $\bR$, and $\bM$ satisfy the following:
	\begin{assumptionlist}
		\item \label{assume:A_semimonotone} $\bA : \bcalH \rightrightarrows \bcalH$ is $(\bU,\bV)$-semimonotone.
		\item \label{assume:B_cocoercive} $\bB:\bcalK \to \bcalK$ is $\bW^{-1}$-cocoercive.
		\item \label{assume:PR} $\bP\in\mathcal{B}(\bcalK,\bcalH)$ and $\bR \in \mathcal{B}(\bcalH,\bcalK)$ satisfy $\bOmega \subseteq \ker (\bP^*-\bR) $.
		\item \label{assume:M_kernel_range} $\bM \in \mathcal{B}(\bcalH',\bcalH)$ has closed range, that is, $\ran (\bM) = \ker(\bM^*)^\perp = \bOmega^\perp$. 
	\end{assumptionlist}
\end{assumption}

\begin{table}[h]
	\centering
	\small
    \renewcommand{\arraystretch}{1.35}
	\newcolumntype{Y}{>{\raggedright\arraybackslash}X}
	\caption{Product space embeddings of the motivating problem classes.}
	\label{tab:productspace_embeddings}
	\begin{tabularx}{\textwidth}{@{}c@{}YYY@{}}
		\toprule
		Problem 
		& Operators and spaces
		& $\bP$ and $\bR$ 
		& $\bOmega$ and $\bb$\\
		\midrule
		
		\eqref{eq:gen_prob_nocoupling}
		&
		\(\bA=\diag(A_1,\ldots,A_n)\)
		
		\(\bB=\diag(B_1,\ldots,B_p)\)
		
		\(\bcalH=\H^n\) , \quad \(\bcalK=\H^p\),
		
		\(\bcalH'=\H^m\)
		
		&
		\(\bP^*\bx=\bR\bx=(u,\ldots,u)\in\H^p\)
		
		\(\forall \bx\in \bOmega = \ker(\bM^*)\)
		&
		 \(\bOmega =
		\{(u_1,\ldots,u_n):u_1=\cdots=u_n\}\)
		
		\(\mathbf b=0\)
		\\
		\midrule
		
		\eqref{eq:gen_prob}
		&
		
		\(\bA=\diag(A_1,\ldots,A_n,C_1,\ldots,C_r)\)
		
		\(\bB=\diag(B_1,\ldots,B_p)\)
		
		\(\bcalH=\H^n\times\K^r,\quad \bcalK=\H^p\)
		
		\(\bcalH'=\H^{n-1}\times\K^r\)
		&
		\(\bP^*\bx=\bR\bx=(u,\ldots,u)\in\H^p\)
		
		\(\forall \bx\in \bOmega = \ker(\bM^*)\)
		&
		 \(\bOmega=
		\{(u_1,\ldots,u_n,v_1,\ldots,v_r):
		u_1=\cdots=u_n,\;
		v_\ell =D_\ell u_1~\forall \ell\}\)
		
		\(\mathbf b= 0\)
		\\
		\midrule 
		
		\eqref{eq:VI}
		&
		\(\bA=\diag(A_1,\ldots,A_n)\)
		
		\(\bB=0\)
		
		\(\bcalH=\H_1\times\cdots\times\H_n,\)
		
		\(\bcalK=\{0\}, \quad \bcalH'=\H\)
		&
		\(\bP=0\)
		
		\(\bR=0\)
		&
		\( \bOmega = \{ (x_1,\dots,x_n): \sum_{i=1}^n D_ix_i=0 \}\)
		
		\(\mathbf b=b\)
		\\
		\midrule 

\begin{tabular}{c}
	\eqref{eq:VI}\\
	(lifted)
\end{tabular}

&
\(\bA=\diag(A_1,\ldots,A_n,0)\)

\(\bB=0\)

\(\bcalH=(\H_1\times\cdots\times\H_n)\times\H^n,\quad
\bcalK=\{0\}\)

\(\bcalH'=\H^n\times\H\)
&
\(\bP=0\)

\(\bR=0\)
&
\(	\bOmega =
\left\{
((x_i)_{i=1}^n,(D_ix_i)_{i=1}^n):\right. \)
\( \qquad \left.
\sum_{i=1}^nD_ix_i=0
\right\}.\)

\(\mathbf b=(0,b)\)
\\
		\bottomrule
	\end{tabularx}
\end{table}
Our proposed approach relies on a preconditioner-like operator $\bL$ and a vector
$\bq$ which allow us to reformulate \eqref{eq:productspace_recall} as a system of inclusions. Associated with $\bL$ and $\bq$, we define for $\gamma>0$ the operators $\Phi_{\gamma}: \bOmega^\perp \toset \bcalH$ given by
\begin{equation}
	\Phi_{\gamma}(\bw)
	\coloneqq
	(\gamma \bA+\gamma \bP\bB\bR+\bL)^{-1}(\bw+\bq),
	\qquad \bw\in\bOmega^\perp,
	\label{eq:Phi_gamma}
\end{equation}
and $\Psi_{\gamma}:\bcalH' \toset \bcalH'$ given by 
\begin{equation*}
	\Psi_{\gamma}(\bz)
	\coloneqq
	\bM^*\Phi_{\gamma}(\bM\bz)-\mathbf{b},
	\qquad \bz\in\bcalH'.
\end{equation*}
By Assumption~A(iv), $\ran(\bM)=\bOmega^\perp$, so that
$\Phi_\gamma(\bM\bz)$ is well defined as a set-valued expression for
every $\bz\in\bcalH'$. We impose the following standing
assumptions.

\begin{assumption}\label{assume:preconditioner} 	$\bL\in \mathcal{B}(\bcalH)$ and $\bq\in \bcalH$ satisfy the following.
	\begin{assumptionlist}
		\item \label{assume:blanket2} $
		\bL(\bOmega) \subseteq \bOmega^\perp$ and $
		\bL\bar{\bx}-\bq\in\bOmega^\perp$
		for some $\bar{\bx}\in\hatOmega.$\footnote{Since $\bL(\bOmega)\subseteq \bOmega^\perp$, the second condition implies that $\bL \bx' - \bq \in \bOmega^\perp$ for \textit{any} $\bx'\in \bOmega_\bb$. Moreover, note that for any $\bar{\bx}\in \hatOmega$, we may always take
			$\bq=\Pi_{\bOmega}(\bL\bar{\bx})$ so that
			$\bL\bar{\bx}-\bq\in\bOmega^\perp$.}
		 	
		\item \label{assume:L_s}
		The operator $\bL_s = \frac{1}{2}(\bL + \bL^*)$ satisfies $
		\ker(\bL_s)=\bOmega,$
		and $\bL_s$ is strongly monotone on $\bOmega^\perp$.\footnote{If $\bcalH$ is finite dimensional, this assumption is equivalent
			to having $\ker(\bL_s)=\bOmega$ and $\bL_s\succeq0$.}
	\end{assumptionlist}
\end{assumption}

The role of \cref{assume:blanket2} is clarified by the following equivalence
result.

\begin{proposition}\label{prop:equivalence_MLauxiliary}
	Suppose that \cref{assume:M_kernel_range} and \cref{assume:blanket2} hold,
	and let $\gamma>0$. Then $\bx\in\bcalH$ solves
	\eqref{eq:productspace_recall} if and only if there exists
	$\bz\in\bcalH'$ such that $(\bx,\bz)$ solves
	\begin{equation}
		\begin{cases}
			\bM\bz+\bq
			\in
			\gamma\bA\bx+\gamma\bP\bB(\bR\bx)+\bL\bx, \\[1mm]
			0=\bM^*\bx-\mathbf{b}.
		\end{cases}
		\label{eq:generalized_ML}
	\end{equation}
\end{proposition}

\begin{proof}
	If $(\bx,\bz)$ solves \eqref{eq:generalized_ML}, then $\bx \in \hatOmega$
	so that $\bx-\bar{\bx}\in\ker(\bM^*)=\bOmega$ and
	$N_{\hatOmega}(\bx)=\bOmega^\perp$. Thus,
	$\bL(\bx-\bar{\bx})\in\bOmega^\perp$ and therefore
	\begin{equation}
		\bL\bx-\bq
		=
		\bL(\bx-\bar{\bx})+(\bL\bar{\bx}-\bq)
		\in \bOmega^\perp
		\label{eq:Lx-v_in_Omegaperp}
	\end{equation}
	by \cref{assume:blanket2}. Consequently,
	$\bL\bx-\bq-\bM\bz\in\bOmega^\perp=N_{\hatOmega}(\bx)$.
	Hence, $\bx$ solves \eqref{eq:productspace_recall}.
	
	Conversely, if $\bx\in\bcalH$ solves \eqref{eq:productspace_recall},
	then $\bx\in\hatOmega$, i.e., $\bM^*\bx-\mathbf{b}=0$, and there exists
	$\bw\in\bOmega^\perp$ such that $
	0\in\gamma\bA\bx+\gamma\bP\bB(\bR\bx)+\bw.
	$
	By \eqref{eq:Lx-v_in_Omegaperp} and \cref{assume:M_kernel_range}, we have
	$
	\bL\bx-\bq-\bw\in\bOmega^\perp=\ran(\bM).
	$
	Hence, there exists $\bz$ such that
	$
	\bL\bx-\bq-\bw=\bM\bz.
	$
	Therefore, $(\bx,\bz)$ solves \eqref{eq:generalized_ML}. This completes the proof.
\smartqedmark \end{proof}

The equivalence above suggests the abstract iteration presented in \cref{algo:abstract_preconditioned_forward}. We refer to $\{\bx^k\}$ as the \textit{shadow sequence}.   Inspecting the main iterates $\{ \bz^k\}$, the iteration can be viewed as a forward scheme provided that $\Psi_{\gamma}$ is single-valued. The additional condition
\cref{assume:L_s} will be used later to establish cocoercivity of
$\Psi_\gamma$ for appropriate choices of $\gamma$. 
\begin{algorithm}[h]
	\caption{Abstract preconditioned forward scheme}
	\label{algo:abstract_preconditioned_forward}
	
	Let $\gamma>0$, let $(\lambda_k)_{k\in\mathbb{N}}\subset[0,+\infty)$,
	and let $\bz^0\in\bcalH'$. For each $k\in\mathbb{N}$, 
	
	\begin{enumerate}
		\item Find $\bx^k\in\bcalH$ such that $
		\bx^k \in \Phi_{\gamma}(\bM\bz^k),$
		that is, 
		\[
		\bM\bz^k+\bq
		\in
		\gamma\bA\bx^k+\gamma\bP\bB(\bR\bx^k)+\bL\bx^k.
		\]
		
		\item Update
		$
		\bz^{k+1}
		=
		\bz^k-\lambda_k(\bM^*\bx^k-\mathbf{b})
		\in
		\bz^k-\lambda_k\Psi_{\gamma}(\bz^k).
		$
	\end{enumerate}
\end{algorithm}

\subsection{Some realizations of the abstract scheme}\label{sec:realizations}

We illustrate how the abstract scheme can be used to derive concrete
splitting algorithms by choosing the operator-vector pair
$(\bL,\bq)$ according to the structure of the problem. In particular,
we give realizations for the linearly coupled inclusion
\eqref{eq:gen_prob} and the block-separable affine-constrained
inclusion \eqref{eq:VI}.
The multioperator inclusion \eqref{eq:gen_prob_nocoupling} has additional
consensus structure and will be treated separately in
\cref{sec:multi_inclusion}.

\begin{example}[A linearly coupled inclusion]
	\label{ex:realization_P2}
	Consider \eqref{eq:gen_prob}, with the product space embedding given in
	\cref{tab:productspace_embeddings}. Define
	\begin{align*}
		\bM(z_1,\ldots,z_{n-1},w_1,\ldots,w_r)
		\coloneqq
		\Bigg(
		&z_1-\sum_{\ell=1}^r D_\ell^*w_\ell,\,
		-z_1+z_2,\ldots,\\
		& -z_{n-2}+z_{n-1},-z_{n-1},\,
		w_1,\ldots,w_r
		\Bigg),
	\end{align*}
	so that $\bOmega=\ker(\bM^*)$.
	For
	$\bx=(x_1,\ldots,x_n,v_1,\ldots,v_r)\in\bcalH$, define
	$
	\bR\bx
	\coloneqq
	(x_{n-1},\ldots,x_{n-1})\in\H^p,
	$
	and, for $(y_1,\ldots,y_p)\in\H^p$, define
	$
	\bP(y_1,\ldots,y_p)
	=
	\left(
	0,\ldots,0,\,
	\sum_{j=1}^p y_j,\,
	0,\ldots,0
	\right),
	$
	where the nonzero component is the $n$th block.
	Then $\bP^*\bx=\bR\bx$ for every $\bx\in\bOmega$.
	
	Let $\alpha_i>0$, $i=1,\ldots,n-1$, and
	$\beta_\ell>0$, $\ell=1,\ldots,r$, and set
	$
	S\coloneqq\sum_{\ell=1}^r\beta_\ell D_\ell^*D_\ell.
	$
	Define $\bL=[L_{ij}]$ by $L_{11}
	=\alpha_1\Id$,  $L_{nn} =\alpha_{n-1}\Id$, $L_{n,n-1}
	=2S-2\alpha_{n-1}\Id,$ and
		\begin{align*}
		L_{ii}
		&=(\alpha_{i-1}+\alpha_i)\Id,
		&& i=2,\ldots,n-1,
		\\
		L_{i+1,i}
		&=-2\alpha_i\Id,
		&& i=1,\ldots,n-2,
		\\
		L_{n+\ell,n-1}
		&=-2\beta_\ell D_\ell,
		&& \ell=1,\ldots,r,
		\\
		L_{n+\ell,n}
		&=-2\beta_\ell D_\ell,
		&& \ell=1,\ldots,r,
		\\
		L_{n+\ell,n+\ell}
		&=2\beta_\ell\Id,
		&& \ell=1,\ldots,r,
	\end{align*}
	with all other blocks equal to zero, and let $\bq=0$. We first verify \cref{assume:blanket2}. If
	$
	\bx
	=
	(x,\ldots,x,D_1x,\ldots,D_rx)
	\in\bOmega,
	$
	then the sum of the first $n$ components of $\bL\bx$ is $2Sx$,
	while its $(n+\ell)$th component is $-2\beta_\ell D_\ell x$.
	By the definition of $S$, we have 
	 $\bL\bx\in\bOmega^\perp$. Since
	$\hatOmega=\bOmega$ and $\bq=0$, it follows that
	\cref{assume:blanket2} holds.
	
	Moreover, for any
	$\bx=(x_1,\ldots,x_n,v_1,\ldots,v_r)\in\bcalH$, a direct
	calculation yields
	\begin{align*}
		\inner{\bL_s\bx}{\bx}
		={}&
		\sum_{i=1}^{n-2}
		\alpha_i\norm{x_i-x_{i+1}}^2
		\\
		&+
		\inner{
			\left(
			\alpha_{n-1}\Id-\frac{1}{2}S
			\right)
			(x_{n-1}-x_n)
		}{
			x_{n-1}-x_n
		}
		\\
		&+
		\frac{1}{2}
		\sum_{\ell=1}^r
		\beta_\ell
		\norm{
			D_\ell(x_{n-1}+x_n)-2v_\ell
		}^2
		\\
		\geq{}&
		\sum_{i=1}^{n-2}
		\alpha_i\norm{x_i-x_{i+1}}^2
		+
		\left(
		\alpha_{n-1}-\frac{1}{2}\norm{S}
		\right)
		\norm{x_{n-1}-x_n}^2
		\\
		&+
		\frac{1}{2}
		\sum_{\ell=1}^r
		\beta_\ell
		\norm{
			D_\ell(x_{n-1}+x_n)-2v_\ell
		}^2.
	\end{align*}
	Thus, if $
	\alpha_{n-1}>\frac{1}{2}\norm{S},$
	then $\bL_s$ satisfies $\ker(\bL_s)=\bOmega$ and is strongly
	monotone on $\bOmega^\perp$, and therefore
	\cref{assume:L_s} holds.
	
	Writing
	$
	\bz^k
	=
	(z_1^k,\ldots,z_{n-1}^k,w_1^k,\ldots,w_r^k) $
    and
    $\bx^k
	=
	(x_1^k,\ldots,x_n^k,v_1^k,\ldots,v_r^k),
	$
	\cref{algo:abstract_preconditioned_forward} simplifies to
	\begin{equation}
	\begin{cases}
		x_1^k
		\in
		J_{\frac{\gamma}{\alpha_1}A_1}
		\left(
		\frac{1}{\alpha_1}
		\left(
		z_1^k
		-
		\sum\limits_{\ell=1}^r
		D_\ell^*w_\ell^k
		\right)
		\right),
		\\[4mm]
		x_i^k
		\in
		J_{\frac{\gamma}{\alpha_{i-1}+\alpha_i}A_i}
		\left(
		\frac{
			-z_{i-1}^k
			+z_i^k
			+2\alpha_{i-1}x_{i-1}^k
		}{
			\alpha_{i-1}+\alpha_i
		}
		\right),
		\quad i=2,\ldots,n-1,
		\\[4mm]
		x_n^k
		\in
		J_{\frac{\gamma}{\alpha_{n-1}}A_n}
		\left(
		\frac{1}{\alpha_{n-1}}
		\left(
		-z_{n-1}^k
		+
		(2\alpha_{n-1}\Id-2S)x_{n-1}^k
		-
		\gamma\sum\limits_{j=1}^p
		B_j(x_{n-1}^k)
		\right)
		\right),
		\\[4mm]
		v_\ell^k
		\in
		J_{\frac{\gamma}{2\beta_\ell}C_\ell}
		\left(
		D_\ell x_{n-1}^k
		+
		D_\ell x_n^k
		+
		\frac{1}{2\beta_\ell}w_\ell^k
		\right),
		\quad \ell=1,\ldots,r,
		\\[4mm]
		z_i^{k+1}
		=
		z_i^k
		-
		\lambda_k(x_i^k-x_{i+1}^k),
		\quad i=1,\ldots,n-1,
		\\[2mm]
		w_\ell^{k+1}
		=
		w_\ell^k
		-
		\lambda_k(v_\ell^k-D_\ell x_1^k),
		\quad \ell=1,\ldots,r.
	\end{cases}
    \label{eq:realization_P2}
	\end{equation}
	When $r=0$, the latter reduces to a weighted sequential
		forward-Douglas--Rachford method covered by \cite{DTT26} in the
		monotone setting. For $r\geq1$, it differs from the reduced-lifting
		primal-dual method in
		\cite{DPTT25}, while still requiring
		only one scalar-scaled resolvent evaluation of each $A_i$ and $C_\ell$
		per iteration. Unlike existing works, our framework provides convergence guarantees for \eqref{eq:realization_P2} under semimonotonicity assumptions beyond the classical
		monotone setting.
\end{example}

\begin{example}[A two-block affine constraint]
	\label{ex:realization_P3}
	Consider \eqref{eq:VI} with $n=2$, reformulated as \eqref{eq:productspace_recall} according to \cref{tab:productspace_embeddings}. 
	Choose $
	\bL=
	\begin{bmatrix}
		D_1^*D_1 & 0\\
		2D_2^*D_1 & D_2^*D_2
	\end{bmatrix}$, and 
	$\bq=
	\begin{bmatrix}
		0\\
		D_2^*b
	\end{bmatrix}.
	$
	A direct calculation shows that 
	$	\bL\bx=
	\bM(D_1x_1)
	\in \ran(\bM)=\bOmega^\perp$, for any $\bx = (x_1,x_2)\in \bOmega$. 
	 Moreover, if $\bar{\bx}=(\bar{x}_1,\bar{x}_2)\in\hatOmega$, then $\bL\bar{\bx}-\bq
	 =
	 \bM(D_1\bar{x}_1)
	 \in\bOmega^\perp.$ Thus, \cref{assume:blanket2} holds. 
	On the other hand, we have 
	\[
	\langle \bL_s\bx,\bx\rangle
	=
	\norm{D_1x_1+D_2x_2}^2
	=
	\|\bM^*\bx\|^2, \quad \forall \bx \in \bcalH.
	\]
	Hence, $\ker(\bL_s)=\bOmega$, and $\bL_s$ is strongly monotone on
	$\bOmega^\perp$. Thus, \cref{assume:L_s} holds.
	
	The abstract scheme \cref{algo:abstract_preconditioned_forward} becomes
	\begin{equation}
			\begin{cases}
			D_1^*z^k
			\in
			\gamma A_1x_1^k+D_1^*D_1x_1^k,
			\\[1mm]
			D_2^*(z^k+b)
			\in
			\gamma A_2x_2^k
			+2D_2^*D_1x_1^k
			+D_2^*D_2x_2^k,
			\\[1mm]
			z^{k+1}
			=
			z^k-\lambda_k(D_1x_1^k+D_2x_2^k-b).
		\end{cases}
		\label{eq:realization_multi-block}
	\end{equation}
	If $A_i=\hat{\partial} f_i$, one realization of \eqref{eq:realization_multi-block} is given by
	\[
	\begin{cases}
		x_1^k
		\in
		\argmin\limits_{x_1\in\H_1}
		\left\{
		f_1(x_1)
		+
		\frac{1}{2\gamma}\|D_1x_1-z^k\|^2
		\right\},
		\\[3mm]
		x_2^k
		\in
		\argmin\limits_{x_2\in\H_2}
		\left\{
		f_2(x_2)
		+
		\frac{1}{2\gamma}
		\|D_2x_2+2D_1x_1^k-b-z^k\|^2
		\right\},
		\\[3mm]
		z^{k+1}
		=
		z^k-\lambda_k(D_1x_1^k+D_2x_2^k-b).
	\end{cases}
	\]
	The resulting scheme is different from the classical two-block ADMM.
	In the special case $D_1=-D_2=\Id$ and $b=0$, however, it reduces
	directly to the relaxed Douglas--Rachford method. This contrasts with
	the classical ADMM-Douglas--Rachford correspondence, in which ADMM is
	viewed as Douglas--Rachford splitting applied to the dual problem. The convergence analysis developed here moreover applies under the
	semimonotonicity assumptions of the present framework.
\end{example}

\begin{example}[A multi-block affine constraint]
	\label{ex:realization_P3_multiblock}
	Consider \eqref{eq:VI}, reformulated as
	\eqref{eq:productspace_recall} according to the lifted embedding in
	\cref{tab:productspace_embeddings}.
	For
	$
	\bx=(x_1,\ldots,x_n,y_1,\ldots,y_n)\in\bcalH,
	$
	define
	\begin{align*}
		\bL\bx
		\coloneqq
		\Bigg(
		&D_1^*D_1x_1,\ldots,D_n^*D_nx_n,
		-2D_1x_1+y_1+\sum_{j=1}^ny_j,\ldots,
		-2D_nx_n+y_n+\sum_{j=1}^ny_j
		\Bigg),
	\end{align*}
	and let
	$
	\bq=0.$
	
	If $\bx\in\bOmega$, then $y_i=D_ix_i$ for all $i$ and
	$\sum_{i=1}^ny_i=0$, so that
	$
	\bL\bx
	=
	\bM(y_1,\ldots,y_n,0)
	\in\bOmega^\perp.
	$
	Similarly, if $\bar{\bx}\in\hatOmega$, then
	$\bar y_i=D_i\bar x_i$ and $\sum_{i=1}^n\bar y_i=b$, and hence
	$
	\bL\bar{\bx}-\bq
	=
	\bM(\bar y_1,\ldots,\bar y_n,b)
	\in\bOmega^\perp.
	$
	Thus, \cref{assume:blanket2} holds.
	
	Moreover, a direct calculation yields, for every
	$\bx=(x_1,\ldots,x_n,y_1,\ldots,y_n)\in\bcalH$,
	\[
	\inner{\bL_s\bx}{\bx}
	=
	\sum_{i=1}^n\norm{D_ix_i-y_i}^2
	+
	\norm{\sum_{i=1}^ny_i}^2
	=
	\norm{\bM^*\bx}^2.
	\]
	Hence, $\ker(\bL_s)=\bOmega$, and, by
	\cref{assume:M_kernel_range}, $\bL_s$ is strongly monotone on
	$\bOmega^\perp$. Thus, \cref{assume:L_s} holds.
	
	Writing
	$
	\bz^k=(z_1^k,\ldots,z_n^k,w^k)
	$
	and
	$
	\bx^k=(x_1^k,\ldots,x_n^k,y_1^k,\ldots,y_n^k), 
	$ and eliminating the variables $y_i^k$, 
	\cref{algo:abstract_preconditioned_forward} simplifies to
	\[
	\begin{cases}
		D_i^*z_i^k
		\in
		\gamma A_ix_i^k+D_i^*D_ix_i^k,
		\qquad i=1,\ldots,n,
		\\[3mm]
		s^k
		=
		\frac{1}{n+1}
		\sum\limits_{j=1}^n
		\left(
		2D_jx_j^k-z_j^k+w^k
		\right),
		\\[3mm]
		z_i^{k+1}
		=
		z_i^k-\lambda_k(z_i^k - D_ix_i^k-w^k +s^k),
		\qquad i=1,\ldots,n,
		\\[3mm]
		w^{k+1}
		=
		w^k-\lambda_k
		\left(
		s^k-b
		\right).
	\end{cases}
	\]
	In particular, the $n$ inclusions defining $x_i^k$ are decoupled
	and can be evaluated in parallel.
If $A_i=\hat{\partial}f_i$, then the first inclusion defining $x_i^k$
can be solved by choosing
\[
x_i^k
\in
\argmin_{x_i\in\H_i}
\left\{
f_i(x_i)
+
\frac{1}{2\gamma}
\norm{D_ix_i-z_i^k}^2
\right\},
\qquad i=1,\ldots,n.
\]
	The parallel $x_i$-updates are
	reminiscent of Jacobi-type (parallel) ADMM \cite{DengLaiPengYin2017}. The above iteration, however, is different since the coupling between the blocks is handled through the auxiliary variable $w^k$ and the quantity $s^k$.
\end{example}

\subsection{Main convergence theorem}

Our convergence analysis accommodates several semimonotonicity regimes. More precisely, we impose the following assumptions on the pair $(\bU,\bV)$ and on the interaction between $\bV$ and the forward operator $\bB$. Note that $\bU = (\bU_{ij})$ and $\bV = (\bV_{ij})$ denote the block decompositions of $\bU$ and $\bV$, respectively, with respect to $\bOmega$ as in \eqref{eq:2x2decomposition}. 
\begin{assumption}\label{assume:semimonotone_general}
	The semimonotonicity modulus $(\bU,\bV)$ of $\bA$ satisfies the following conditions.
	\begin{assumptionlist}
		\item \label{assume:U} Either $\bU=0$ or $\bU_{11}$ is strongly monotone. 
		\item \label{tab:cases} One of the following holds. 
		
		\begin{tabular}{@{} l p{0.74\linewidth} @{}}
			\toprule
			\textbf{Case I} & $\bV\equiv  0$. In this case, $\bB$ may be nonzero. \\ 
			\textbf{Case II} & $\bB=0$, $\bV_{22}\succeq0$, $\ran(\bV_{22})$ is closed, and
			$\ran(\bV_{21})\subseteq\ran(\bV_{22})$.
			\\
			
			\bottomrule
		\end{tabular}
	\end{assumptionlist}
\end{assumption}

Note that the strong monotonicity of $\bU_{11}$ (equivalently, strong monotonicity of $\bU$ on $\bOmega$) does not require $\bU$ to be monotone on $\bcalH$. In particular, $\bU$ may be nonmonotone. Thus, the convergence theorem applies in the following four configurations:
\begin{enumerate}[(C1)]
	\item Case~I holds and $\bU=0$;
	\item Case~I holds and $\bU_{11}$ is strongly monotone;
	\item Case~II holds and $\bU=0$;
	\item Case~II holds and $\bU_{11}$ is strongly monotone.
\end{enumerate}
The complete picture of the regimes considered can be visualized in \cref{tab:Gammastar}.

We next introduce the admissible values of the parameter $\gamma$
that arise in the convergence analysis. Define
	\begin{equation}
	\Lambda
	\coloneqq
	\Big\{
	\gamma>0:
	\bOmega^\perp
	\subseteq
	\dom(\Phi_{\gamma})
	\ \text{and}\
	\Phi_{\gamma}
	\ \text{is single-valued on } \bOmega^\perp
	\Big\},
	\label{eq:Lambda}
\end{equation}
where $\Phi_{\gamma}$ is given by \eqref{eq:Phi_gamma}.  Denote $\bX \coloneqq \bU/\bU_{11}$, 
\begin{align}
	\bTheta_{\gamma}& \coloneqq 
	\begin{cases}
		\frac{\gamma}{4}(\bP - \bR^*) \bW (\bP^* - \bR), & \text{in Case I},\\[4pt]
		-\frac{1}{\gamma}\,\bL^*(\bV/\bV_{22})\bL, & \text{in Case II},
		\end{cases} \notag \\ 
	\widehat{\bXi}_{\gamma} & \coloneqq 
		\bL_s - \bTheta_{\gamma} + \gamma \bX, \label{eq:bSigmahat}
\end{align}
and let 
\begin{equation}\label{eq:Gamma*}
	\Gamma{} \coloneqq \left\lbrace  \gamma>0: \lambda_{\max}\!\left((\bL_s^{1/2})^{\dagger}\left( \bTheta_{\gamma} - \gamma \bX\right)(\bL_s^{1/2})^{\dagger}\right)<1\right\rbrace .
\end{equation}   
The set $\Gamma{}$, or an estimate thereof, under different regimes (C1)--(C4), is presented in \cref{tab:Gammastar}. 
The roles of $\Lambda$ and $\Gamma{}$ are different. The condition
$\gamma\in\Lambda$ ensures that $\Phi_\gamma$ is single-valued on
$\bOmega^\perp$ and has the required domain, so that the iteration is
well defined. On the other hand, $\gamma\in\Gamma{}$ guarantees the
cocoercivity estimate for $\Psi_\gamma$ used in the convergence
analysis. Thus, the admissible values of $\gamma$ in the main theorem
belong to $\Lambda\cap\Gamma{}$.

The following theorem is the main result of this paper on the convergence of \cref{algo:abstract_preconditioned_forward}. For the convergence of the shadow sequence, Cases C1 and C3 (\textit{i.e.,} cases where $\bU = 0$) are covered by \cref{thm:mainresult}(ii). Cases C2 and C4 are covered by \cref{thm:mainresult}(iii). To focus on the main ideas, we defer the proof of this theorem to \cref{sec:proof_mainresult}.

\begin{theorem}[Main convergence theorem]\label{thm:mainresult}
Suppose that \cref{assume:blanket,assume:preconditioner,assume:semimonotone_general} hold, and let $\Lambda$ and $\Gamma{}$ be given in \eqref{eq:Lambda} and \eqref{eq:Gamma*}, respectively. 
    Fix $\gamma \in \Lambda \cap \Gamma{}$ and let $\beta_{\gamma}\coloneqq 
		\left\lVert \left(\widehat{\bXi}_{\gamma}^{1/2}\right)^\dagger \bM \right\rVert^{-2}$, where $\widehat{\bXi}_{\gamma}$ is given by \eqref{eq:bSigmahat}.  
	If $\lambda_k\in [0,2\beta_\gamma]$ satisfies $\sum_{k=0}^{\infty} \lambda_k (2\beta_\gamma-\lambda_k) = +\infty$,  then the following hold: 
	\begin{enumerate}
		\item \textit{(Weak convergence of the main sequence).} $\{\bz^k\}$ is bounded and converges weakly to some point $\bz^*\in \zer(\Psi_{\gamma})$. Moreover,
		$\Psi_{\gamma}(\bz^k)=\bM^*\bx^k-\bb\to 0$.
		\item  \textit{(Weak convergence of the shadow sequence).} If $\bA$ is maximally $\bV$-comonotone, $\bU = 0$, and $\{\bx^k\}$ is bounded, then there exists a solution $\bx^*$ of \eqref{eq:productspace_recall} such that $\bx^k\toweak \bx^*$.
		\item  \textit{(Strong convergence of the shadow sequence).} If $\bU$ is strongly monotone on $\bOmega$, then  there exists a solution $\bx^*$ of \eqref{eq:productspace_recall} such that $\bx^k\to \bx^*$.
	\end{enumerate}
\end{theorem}

\begin{table}[ht]
	\centering
	\renewcommand{\arraystretch}{1.25}
	\setlength{\tabcolsep}{8pt}
	\caption{\small Summary of notation. $\bU = (\bU_{ij})$ and $\bV = (\bV_{ij})$ denote the block decomposition of $\bU$ and $\bV$, respectively, with respect to $\bOmega$ as in \eqref{eq:2x2decomposition}.}
	\label{tab:notation}
	
	\begin{tabular}{@{} l p{0.74\linewidth} @{}}
		\toprule
		Symbol & Definition \\
		\midrule
		$\beta_{\gamma}$ &
		$\left\lVert \left(\widehat{\bXi}_{\gamma}^{1/2}\right)^\dagger \bM \right\rVert^{-2}
		$
		\\[6pt]
		
		$\Lambda$ & $\Big\{
		\gamma>0:
		\bOmega^\perp
		\subseteq
		\dom(\Phi_{\gamma})
		\ \text{and}\
		\Phi_{\gamma}
		\ \text{is single-valued on } \bOmega^\perp
		\Big\}$ \\[6pt]
		
		$\Gamma{}$ & $\left\lbrace  \gamma>0: \lambda_{\max}\!\left((\bL_s^{1/2})^{\dagger}\left( \bTheta_{\gamma} - \gamma \bX\right)(\bL_s^{1/2})^{\dagger}\right)<1\right\rbrace $. Spectral condition ensuring the cocoercivity of $\Psi_\gamma$
		whenever $\gamma\in\Lambda$. \\[6pt]
		
		\midrule 
		
		$\bX$ & $\bU/\bU_{11}$  
		\\[6pt]
		
		$\bTheta_{\gamma}$ &
		$
		\begin{cases}
			\frac{\gamma}{4}(\bP - \bR^*) \bW (\bP^* - \bR), & \text{in Case I}\\[4pt]
			-\frac{1}{\gamma}\,\bL^*(\bV/\bV_{22})\bL, & \text{in Case II}
		\end{cases}
		$ 
		\\[8pt]
		
		$\widehat{\bXi}_{\gamma}$ &
		$
		\bL_s - \bTheta_{\gamma} + \gamma \bX
		$
		\\[6pt]
		\midrule 
		
		$\rho$ &
		$
		\lambda_{\max}\!\left( (\bL_s^{1/2})^\dagger \Bigl[(\bP - \bR^*) \bW (\bP^* - \bR) -4\bX\Bigr] (\bL_s^{1/2})^\dagger \right)
		$
		\\[8pt]
		
		$\rho_1$ &
		$
		\lambda_{\max}\!\left( (\bL_s^{1/2})^\dagger (\bP - \bR^*) \bW (\bP^* - \bR) (\bL_s^{1/2})^\dagger \right)
		$
		\\[8pt]
		
		$\rho_2$ &
		$
		\lambda_{\min}\!\left( (\bL_s^{1/2})^\dagger \bX (\bL_s^{1/2})^\dagger \right)
		$
		\\[8pt]
		
		$\rho_3$ &
		$
		\lambda_{\min}\!\left( (\bL_s^{1/2})^\dagger \bL^* (\bV/\bV_{22}) \bL (\bL_s^{1/2})^\dagger \right)
		$
		\\[10pt]
		
		$\rho_{\ell},\rho_{r}$ &
		$\rho_{\ell}<\rho_{r}$ are the (real) roots of the quadratic equation
		$\rho_2\gamma^2+\gamma+\rho_3=0$ with $\rho_2\rho_3<1/4$.
		\\
		
		\bottomrule
	\end{tabular}
\end{table}

\begin{table}[ht]
	\centering
	\renewcommand{\arraystretch}{1.35}
	\setlength{\arrayrulewidth}{0.8pt}
	\caption{\small The set $\Gamma{}$ (or a nonempty subset thereof) under the different
		semimonotonicity regimes. Together with the requirement
		$\gamma\in\Lambda$, these give the admissible values of $\gamma$ in
		\cref{thm:mainresult}. The cases where we only provide a nonempty estimate of $\Gamma{}$ are marked with $^\ddagger$. The shaded (gray) entries correspond to Case~I in \cref{tab:cases}, while the unshaded (white) entries correspond to Case~II. The parameters $\rho$, $\rho_1$, $\rho_2,$ and $\rho_3$ are defined in \cref{tab:notation}.}
	\label{tab:Gammastar}
	\begin{threeparttable}
		\begin{tabular}{|c|c!{\vrule width \Thick}c!{\vrule width \Thick}c|c|}
			
			\hline
			\multicolumn{2}{|c!{\vrule width \Thick}}{} &
			\multirow{2}{*}{$\bU = 0$} &
			\multicolumn{2}{c|}{$\bU_{11}\ \text{strongly monotone}$}\\
			\cline{4-5}
			\multicolumn{2}{|c!{\vrule width \Thick}}{} &
			& \multicolumn{1}{c|}{$\bU/\bU_{11}\succeq 0$} &
			\multicolumn{1}{c|}{$\bU/\bU_{11}\not\succeq 0$}\\
			\HThick
			
			\multicolumn{2}{|c!{\vrule width \Thick}}{$\bB \neq 0$, $\bV = 0$} &
			\cellcolor{gray!25}$\left(0,\;\frac{4}{\rho_1}\right)$ &
			\multicolumn{2}{c|}{\cellcolor{gray!25}$\left(0,\;\frac{4}{\rho_{+}}\right)$}\\
			\hline
			
			\multirow{3}{*}{$\bB=0$, $\bV_{22}\succeq 0$} &
			$\bV=0$ &
			\multicolumn{2}{c|}{\cellcolor{gray!25}$\left(0,\;+\infty\right)$} &
			\cellcolor{gray!25}$\left(0,\;\frac{1}{-\rho_2}\right)$\\
			\cline{2-5}
			
			& $\bV/\bV_{22}\succeq 0$\quad \tnote{\textit{a}} &
			\multicolumn{2}{c|}{$\left(0,\;+\infty\right)$} &
			$ \left(0,\;\frac{1}{-\rho_2}\right)\subseteq \Gamma{}$ \quad \tnote{\textit{b} }\\
			\cline{2-5}
			
			& $\bV/\bV_{22}\not\succeq 0$\quad \tnote{\textit{a}} &
			\multicolumn{1}{c|}{$\left(-\rho_3,\;+\infty\right)$} & \multicolumn{1}{c|}{$\left(-\rho_3,\;+\infty\right)\subseteq \Gamma{}$ \quad \tnote{\textit{b} }}  &
			$ \left(\rho_l,\;\rho_r\right)\subseteq \Gamma{}$ \quad  \tnote{\textit{b} } \tnote{\textit{c}}\\
			\hline
		\end{tabular}
		
		\begin{tablenotes}[flushleft]
			\footnotesize
			\item[\textit{a}] Additionally, it is assumed here that $\ran(\bV_{22})$ is closed and $\ran(\bV_{21})\subseteq \ran(\bV_{22})$.
			\item[\textit{b}] The actual set $\Gamma{}$ is given by $$\Gamma{} = \{\gamma >0 : \gamma + \lambda_{\min}\left( (\bL_s^{1/2})^\dagger  [\gamma^2 \bX +\bL^*(\bV/\bV_{22})\bL](\bL_s^{1/2})^\dagger \right) >0\}.$$
			\item[\textit{c}]  For this estimate, assuming that $\rho_2\rho_3<1/4$, the quadratic equation $\rho_2\gamma^2+\gamma+\rho_3=0$ has two roots $\rho_l<\rho_r$. For any $\gamma\in(\rho_l,\rho_r)\neq \varnothing $, the inequality $\rho_2\gamma^2+\gamma+\rho_3>0$ holds so that $\gamma \in \Gamma{}$. 
		\end{tablenotes}	
	\end{threeparttable}
\end{table}

\begin{example}
	We briefly indicate how \cref{assume:semimonotone_general} is satisfied in
	\cref{ex:realization_P2,ex:realization_P3,ex:realization_P3_multiblock} for some scalar-modulus cases.
	
	\begin{enumerate}
	\item Consider the linearly coupled realization \eqref{eq:gen_prob} in
	\cref{ex:realization_P2}. Suppose that $A_i$ is
	$(\mu_i\Id,\nu_i\Id)$-semimonotone for $i=1,\ldots,n$, and that
	$C_\ell$ is $(\eta_\ell\Id,\xi_\ell\Id)$-semimonotone for
	$\ell=1,\ldots,r$. Set
	$
	\bU\coloneqq
	\diag(
	\mu_1\Id,\ldots,\mu_n\Id,
	\eta_1\Id,\ldots,\eta_r\Id
	)
	$
	and define $\bV$ analogously using $\nu_i$ and $\xi_\ell$.
	If $\bB\neq0$, then \cref{assume:semimonotone_general} holds when
	$\bV=0$, and either $\bU=0$ or $\bU$ is strongly monotone on
	$\bOmega$. The latter condition is equivalent to the strong
	monotonicity of
	$
	\left(\sum_{i=1}^n\mu_i\right)\Id
	+
	\sum_{\ell=1}^r\eta_\ell D_\ell^*D_\ell
	$ on $\H$.
	
	When $\bB=0$, nonzero $\bV$ can also be allowed through
	Case~II. Suppose, for simplicity, that either $\nu_i>0$ for every
	$i$, or there is a unique $j$ such that $\nu_j<0$, $\nu_i>0$ for
	$i\neq j$, and
	$
	\sum_{i=1}^n \nu_i^{-1}<0.
	$
	Define $
	\kappa\coloneqq
	\left(\sum_{i=1}^n\nu_i^{-1}\right)^{-1},$ 
	$\bD x\coloneqq(D_1x,\ldots,D_rx),
	$
	and
	$
	\bXi\coloneqq
	\diag(\xi_1\Id,\ldots,\xi_r\Id).
	$
	By an argument analogous to
	\cref{lemma:UV_specialcase}(ii), $\bV_{22}$ is strongly monotone
	whenever $\bXi+\kappa\bD\bD^*$ is strongly monotone on $\K^r$.
	Consequently,
	$\bV_{22}\succeq0$ and
	$\ran(\bV_{22})=\bOmega^\perp$.
	In particular, $\ran(\bV_{22})$ is closed and
	$\ran(\bV_{21})\subseteq\ran(\bV_{22})$, so the requirements of
	Case~II are satisfied.
		
		\item Consider the affine-constrained realization in
		\cref{ex:realization_P3}. Suppose that $A_1$ and $A_2$ are
		$\mu_1\Id$-monotone and $\mu_2\Id$-monotone, respectively, and set
		$\bU\coloneqq\diag(\mu_1\Id,\mu_2\Id)$. Assume, for simplicity, that
		$D_2$ is invertible, and define
		$a\coloneqq\inf_{\|x\|=1}\|D_2^{-1}D_1x\|^2$ and
		$b\coloneqq\|D_2^{-1}D_1\|^2$. Since
		$\bOmega=\{(x_1,x_2):D_1x_1+D_2x_2=0\}$, every element of $\bOmega$ can be
		written as $(x,-D_2^{-1}D_1x)$. Hence, $\bU$ is strongly monotone on
		$\bOmega$ if and only if $\mu_1+\mu_2 a>0$ when $\mu_2\geq\mu_1$, and
		$\mu_1+\mu_2 b>0$ when $\mu_2<\mu_1$. 	 
		
		\item Consider the multi-block affine-constrained realization in
		\cref{ex:realization_P3_multiblock}. Suppose that each $A_i$ is
		$\mu_i\Id$-monotone, and set
		$
		\bU
		\coloneqq
		\diag
		\left(
		\mu_1\Id,\ldots,\mu_n\Id,
		0,\ldots,0
		\right),$ and $
		\bV=0.
		$
		Hence, $\bU$ is strongly monotone on $\bOmega$ if and only if there
		exists $c>0$ such that
		$
		\sum_{i=1}^n
		\mu_i\norm{x_i}^2
		\geq
		c
		\sum_{i=1}^n
		\left(
		\norm{x_i}^2+\norm{D_ix_i}^2
		\right)
		$
		whenever
		$
		\sum_{i=1}^nD_ix_i=0.
		$
		Equivalently, since each $D_i$ is bounded, this holds if and only if
		there exists $\widetilde c>0$ such that
		$
		\sum_{i=1}^n
		\mu_i\norm{x_i}^2
		\geq
		\widetilde c
		\sum_{i=1}^n\norm{x_i}^2$  whenever 
		$\sum_{i=1}^nD_ix_i=0.
		$
	In particular, this condition may hold even when some $\mu_i<0$.
	For $n=2$ and invertible $D_2$, it reduces to the condition in
	(ii) above, with the same constants $a$ and $b$, and therefore allows
	one of the operators to be nonmonotone.
	\end{enumerate}
	
\end{example}

\section{Multioperator inclusions}\label{sec:multi_inclusion}

In this section, we consider the inclusion problem \eqref{eq:gen_prob_nocoupling}, which we recall as follows:
\begin{align}
\text{find } x \in \H \text{ such that } 0
&\in \sum_{i=1}^n A_i x + \sum_{j=1}^p B_j x .\label{eq:gen_prob_recall} \tag{P1}
\end{align}
Throughout this section, we assume that the solution set of \eqref{eq:gen_prob_recall} is nonempty. We work under the following assumptions.
\begin{assumption}\label{assume:blanket_multioperator}
    $A_i:\H\toset \H$ ($i=1,\dots,n$) and $B_j : \H \to \H$ ($j=1,\dots,p$) satisfy the following:
    \begin{enumerate}
        \item $A_i$ is maximally $(U_i,V_i)$-semimonotone.
        \item $B_j$ is $W_j^{-1}$-cocoercive.
    \end{enumerate}
\end{assumption}

In the following, we present a general splitting algorithm for solving \eqref{eq:gen_prob_recall} under \cref{assume:blanket_multioperator}. For convenience in our forthcoming discussions, we introduce some notations. We let $1_n$ denote the vector $1_n = (1,1,\dots,1)^\top \in \Re^n$. We use $\otimes$ to denote the Kronecker product. In particular, $1_n\otimes x = (x,x,\dots,x)\in \H^n$ for $x\in \H$. Given a matrix $M = (M_{ij}) \in \mathbb{R}^{n\times m}$, we define $M\otimes \Id:\H^m \to \H^n$ by $((M\otimes \Id)\bz)_i = \sum_{j=1}^m M_{ij}z_j$ for all $\bz = (z_1,\dots,z_m)\in \H^m$ and all $i=1,\dots,n$. 

We first derive a general product space splitting scheme for
\eqref{eq:gen_prob_recall}, then specialize the convergence results
of \cref{sec:gen_subspace}, and finally compare the resulting
framework with existing multioperator splitting methods.

\subsection{Product space reformulation and algorithms}

Let $\bcalH = \H^n$ and $\bcalK = \H^p$. Throughout this section, we define the operators $\bA : \bcalH\toset \bcalH$ and $\bB : \bcalK \to \bcalK$ as
\begin{equation}\label{eq:AB_prodspace}
\bA\coloneqq \diag(A_1,\dots,A_n),\qquad \text{and} \qquad 
\bB\coloneqq \diag(B_1,\dots,B_p).
\end{equation}
That is, $\bA \bx = A_1 x_1\times A_2x_2\times \cdots\times  A_n x_n$ and $\bB \bs = (B_1s_1,\dots,B_ps_p)$ for all $\bx = (x_1,\dots,x_n)\in \bcalH$ and $\bs =(s_1,\dots,s_p) \in \bcalK$. 
We also fix $\bOmega\subseteq \bcalH$ to be the diagonal subspace
\begin{equation}
	\bOmega\coloneqq \{ 1_n \otimes x : x\in \H\} = \{(x,\dots,x):x\in\H\} .\notag 
\end{equation}
Note that
\begin{equation}\label{eq:bOmegaperp}
	\bOmega^\perp = \{ \bu = (u_1,\dots,u_n) : (1_n^\top \otimes \Id)\bu = \sum_{i=1}^n u_i = 0 \}.
\end{equation}
 Finally, we let $\bP\in \mathcal{B}(\bcalK,\bcalH)$ and $\bR \in \mathcal{B}(\bcalH,\bcalK)$ such that 
\begin{equation}\label{eq:PR}
	\bP^* (1_n\otimes x) = \bR (1_n\otimes x) = {1}_p\otimes x \qquad \forall x\in \H.  
\end{equation}
That is, $\bP^* (1_n \otimes \Id) = \bR (1_n \otimes \Id) = 1_p\otimes \Id$. Note that if $\bP=[P_{ij}]$ and $\bR=[R_{ji}]$ are the block
representations of $\bP$ and $\bR$, respectively, then
\eqref{eq:PR} is equivalent to
\[
\sum_{i=1}^n P_{ij}^*
=
\sum_{i=1}^n R_{ji}
=
\Id,
\qquad j=1,\ldots,p.
\]
With these, we have the following equivalent product space reformulation of \eqref{eq:gen_prob_recall}.
\begin{proposition}[Product space reformulation]
    \label{prop:prodspace_equivalent}
  	$\bx \in \zer \left(\bA + \bP \bB \bR + N_{\bOmega}\right)$ if and only if $\bx = 1_n\otimes x $ with $x\in \zer \left(  \sum_{i=1}^n A_i  + \sum_{j=1}^p B_j \right) $. That is, \eqref{eq:productspace_recall} is equivalent to \eqref{eq:gen_prob_recall}.
\end{proposition}
\begin{proof}
Suppose that $\bx \in \zer \left(\bA + \bP \bB \bR + N_{\bOmega}\right)$. Then $\bx\in \bOmega$ so that $\bx = 1_n \otimes x$ for some $x\in \H$. In addition, there exists $\ba =(a_1,\dots,a_n) \in \bA \bx$ such that $\ba + \bP \bB (\bR \bx) \in N_{\bOmega}(\bx) = \bOmega^\perp$, where $a_i\in A_i x$. Thus,  by \eqref{eq:bOmegaperp}, we have 
\begin{align}
	0 &= (1_n^\top \otimes \Id)\,\ba + (1_n^\top \otimes \Id)\,\bP\bB\big(\bR(1_n\otimes x)\big) \notag\\
	&= \sum_{i=1}^n a_i + \big(\bP^*(1_n\otimes \Id)\big)^* \bB\big(\bR(1_n\otimes x)\big) \notag\\
	&= \sum_{i=1}^n a_i + (1_p^\top \otimes \Id)\,\bB(1_p\otimes x)
	&& \text{(by \eqref{eq:PR})} \notag \\
	&= \sum_{i=1}^n a_i + \sum_{j=1}^p B_j x
	&& \text{(by \eqref{eq:AB_prodspace})} \notag \label{eq:use_B_def}
\end{align}
Hence, $x$ is a zero of $\sum_{i=1}^n A_i  + \sum_{j=1}^p B_j $. The converse holds by reversing the arguments. 
\smartqedmark \end{proof}

Let $\bcalH'=\H^m$ for some $m\geq n-1$, and take
$\mathbf b= 0$ and $\bq=0$.
In view of \cref{prop:equivalence_MLauxiliary}, we choose
$\bM\in\mathcal B(\bcalH',\bcalH)$ and
$\bL\in\mathcal B(\bcalH)$ so that
\cref{assume:M_kernel_range,assume:blanket2} hold.
In the present consensus setting, these conditions reduce to the
following.
\begin{assumption}\label{assume:blanket2_components}
	The linear operators $\bM = [M_{ij}]$ and $\bL = [L_{ij}]$ with $M_{ij},L_{ij}\in \mathcal{B}(\H)$ satisfy $\ker(\bM^*) = \bOmega$, $\bM$ has closed range, and $\sum_{i=1}^n\sum_{j=1}^n L_{ij} = 0$.
\end{assumption}

We decompose $\bL$ as $\bL = \bD - \bN$ where 
\begin{equation}\label{eq:DN}
	\bD \coloneqq \diag (D_1, \dots, D_n) \qquad \text{and} \qquad \bN \coloneqq \bD - \bL,
\end{equation}
with $D_i \coloneqq  L_{ii}$, that is, $\bD$ is the diagonal part of $\bL$. If $\bD$ is invertible, then recalling \eqref{eq:Phi_gamma}, we have
\begin{align}
	\bx \in \Phi_{\gamma}(\bM \bz) \quad & \Longleftrightarrow \quad \bM \bz \in \gamma \bA \bx + \gamma \bP \bB (\bR \bx) + \bL \bx \notag \\
	& \Longleftrightarrow \quad \bM \bz + \bN \bx - \gamma \bP \bB (\bR \bx) \in (\gamma \bA + \bD )\bx \notag \\
	& \Longleftrightarrow \quad \bD^{-1} \bM \bz + \bD^{-1}\bN \bx - \gamma \bD^{-1}\bP \bB (\bR \bx) \in (\gamma \bD^{-1} \bA + \Id)\bx \notag \\
	& \Longleftrightarrow \quad \bx \in J_{\gamma \bD^{-1} \bA} \left( \bD^{-1} \bM \bz + \bD^{-1}\bN \bx - \gamma \bD^{-1}\bP \bB (\bR \bx) \right). \label{eq:Phigamma_J_Dinverse_A}
\end{align}
Hence, \cref{algo:abstract_preconditioned_forward} can be written in componentwise form as in \cref{algo:full_implicit}. In general, note that \cref{algo:full_implicit} is an \textit{implicit algorithm}. In the case that $D_i$, $N_{ij}$, $M_{ij}$, $P_{ij}$ and $R_{jl}$ are scalar operators, \cref{algo:full_implicit} is identical to the splitting algorithm proposed in \cite{DTT26}.

\begin{algorithm}[h]
	\caption{Generalized Forward-Backward Splitting}\label{algo:full_implicit}
	
	Let $\bM =[M_{ij}] \in \mathcal{B}(\H^m,\H^n)$, $\bN =[N_{ij}]\in \mathcal{B}(\H^n)$, $\bP =[P_{ij}]\in \mathcal{B}(\H^p,\H^n)$,
	$\bR =[R_{ij}]\in \mathcal{B}(\H^n,\H^p)$, and $\bD =\diag(D_1,\dots,D_n)$ with self-adjoint strongly monotone
	$D_i\in \mathcal{B}(\H)$ for $i=1,\dots,n$. Let $\gamma \in (0, +\infty)$ and
	$(\lambda_k)_{k\in \mathbb{N}}\subset [0, +\infty)$. Let $z^0_1, \dots, z^0_{m}\in \H$.
	For each $k\in \mathbb{N}$, compute
	\begin{align*}
		\begin{cases}
			x^k_i \in J_{\gamma D_i^{-1}A_i}\Bigg(D_i^{-1}\sum_{j=1}^m M_{ij}z^k_j
			+D_i^{-1}\sum_{j=1}^n N_{ij}x^k_j
			-\gamma D_i^{-1}\sum_{j=1}^p P_{ij}B_j\!\left(\sum_{l=1}^n R_{jl}x^k_l\right)\Bigg), \\
			\qquad  \quad (\text{for } i = 1,2,\dots,n), \\[2mm]
			z^{k+1}_i = z^k_i -\lambda_k\sum_{j=1}^n M_{ji}^* x^k_j,\qquad  (\text{for } i = 1,2,\dots,m).
		\end{cases}
	\end{align*}
\end{algorithm}

Specializing to the case where $\bL$ and $\bR$ are lower triangular operators\footnote{For a block operator $T=[T_{ij}]$, possibly rectangular, we call
	$T$ lower triangular if $T_{ij}=0$ whenever $i<j$, and strictly
	lower triangular if $T_{ij}=0$ whenever $i\leq j$.} and $\bP$ is strictly lower triangular yields the \textit{explicit algorithm} in \cref{algo:full}.

\begin{algorithm}
	\caption{Generalized Explicit Forward-Backward Splitting}\label{algo:full}
	
	Let $\bM =[M_{ij}] \in \mathcal{B}(\H^m,\H^n)$, let $\bN =[N_{ij}]\in \mathcal{B}(\H^n)$ and $\bP =[P_{ij}]\in \mathcal{B}(\H^p, \H^n)$ be strictly lower triangular,
	$\bR =[R_{ij}]\in \mathcal{B}(\H^n,\H^p)$ be lower triangular, and $\bD =\diag(D_1,\dots,D_n)$ with self-adjoint strongly monotone
	$D_i\in \mathcal{B}(\H)$ for $i=1,\dots,n$. Let $\gamma \in (0, +\infty)$ and
	$(\lambda_k)_{k\in \mathbb{N}}\subset [0, +\infty)$. Let $z^0_1, \dots, z^0_{m}\in \H$.
	For each $k\in \mathbb{N}$, compute
	\begin{align*}
		\begin{cases}
			x^k_1 &\in J_{\gamma D_1^{-1}A_1}\Bigg(D_1^{-1}\sum_{j=1}^m M_{1j}z^k_j\Bigg), \\
			x^k_i &\in J_{\gamma D_i^{-1}A_i}\Bigg(D_i^{-1}\sum_{j=1}^m M_{ij}z^k_j +D_i^{-1}\sum_{j=1}^{i-1} N_{ij}x^k_j \\
			& \qquad \qquad \qquad \qquad -\gamma D_i^{-1}\sum_{j=1}^{\min\{i-1, p\}} P_{ij}B_j\left(\sum_{l=1}^j R_{jl}x^k_l\right)\Bigg),\\ 
			&( i = 2,3,\dots,n), \\
			z^{k+1}_i\!\!\!\!\! &=z^k_i -\lambda_k\sum_{j=1}^n M_{ji}^* x^k_j,\qquad (\text{for } i = 1,2,\dots,m).
		\end{cases}
	\end{align*}
\end{algorithm}

\subsection{Convergence theorem}

We first verify that the product space operators
$\bA,\bB,\bP$, and $\bR$ satisfy
\cref{assume:blanket}(i)--(iii). To this end, define
\begin{equation*}\label{eq:UVW}
	\bU \coloneqq \diag(U_1,\dots,U_n), \quad \bV \coloneqq \diag(V_1,\dots,V_n), \quad \text{and} \quad \bW \coloneqq \diag(W_1,\dots,W_p),
\end{equation*}
where $(U_i,V_i)$ and $W_j$ are the moduli of semimonotonicity and cocoercivity in \cref{assume:blanket_multioperator}. With these, it is straightforward to verify that \cref{assume:blanket}(i)--(iii) hold. We state it formally as follows. Note that the stronger conclusion that $\bA$ is \textit{maximally} $(\bU,\bV)$-semimonotone is not required in \cref{assume:blanket}(i). Nevertheless, we prove that maximality holds in \cref{app:AssumptionA_holds}.

\begin{proposition}\label{prop:assume_blanket_gensubspace_holds}
	Under \cref{assume:blanket_multioperator}, the following hold:
	\begin{enumerate}
		\item $\bA$ is maximally $(\bU,\bV)$-semimonotone.
		\item $\bB$ is $\bW^{-1}$-cocoercive.
		\item $\bOmega \subseteq \ker(\bP^*-\bR)$.
	\end{enumerate}
\end{proposition}

\begin{theorem}[Convergence of \cref{algo:full_implicit}]
	\label{thm:convergence_implicitalgorithm}
	Suppose that
	\cref{assume:L_s,assume:semimonotone_general,assume:blanket_multioperator,assume:blanket2_components}
	hold. Let $\gamma\in\Lambda\cap\Gamma{}$, and let
	$\beta_\gamma>0$ be as in \cref{tab:notation}. Let
	$\{(\bx^k,\bz^k)\}$ be generated by \cref{algo:full_implicit}.
	If $\{\lambda_k\}\subset[0,2\beta_\gamma]$ satisfies
	$\sum_{k=0}^\infty
	\lambda_k(2\beta_\gamma-\lambda_k)=+\infty$, then the following hold:
	\begin{enumerate}
		\item (Weak convergence of the main sequence). The sequence $\{\bz^k\}$ is bounded and converges weakly to
		some point $\bz^*\in\zer(\Psi_\gamma)$.
		
		\item (Weak convergence of the shadow sequence). If $A_i$ is maximally $V_i$-comonotone and $U_i=0$
		for all $i\in\{1,\dots,n\}$, and $\{\bx^k\}$ is bounded, then
		there exists a solution $x^*$ of \eqref{eq:gen_prob_recall} such
		that $x_i^k\toweak x^*$ for all $i\in\{1,\dots,n\}$. Moreover, the boundedness of $\{\bx^k\}$ is automatic if the
		first rows of $\bN$ and $\bP$ are zero and
		$J_{\gamma D_1^{-1}A_1}$ is Lipschitz continuous.
		
		\item (Strong convergence of the shadow sequence). If $\bU$ is strongly monotone on $\bOmega$, then there
		exists a solution $x^*$ of \eqref{eq:gen_prob_recall} such that
		$x_i^k\to x^*$ for all $i\in\{1,\dots,n\}$.
	\end{enumerate}
\end{theorem}
\begin{proof}
	By \cref{prop:assume_blanket_gensubspace_holds},
	\cref{assume:blanket}(i)--(iii) hold under
	\cref{assume:blanket_multioperator}. Moreover,
	\cref{assume:blanket2_components} guarantees
	\cref{assume:blanket}(iv). Recalling that $\bb=0$ and $\bq=0$ in the
	present setting, \cref{assume:blanket2_components} also guarantees
	\cref{assume:blanket2}. Hence, assertion~(i) follows from \cref{thm:mainresult}, while
	assertions~(ii)--(iii) follow from
	\cref{thm:mainresult} and \cref{prop:prodspace_equivalent}.
	
	It remains to verify the additional claim in (ii). By
	\cref{thm:mainresult}(i), $\bM^*\bx^k\to\mathbf{0}$.
	Since $\ran(\bM)=\bOmega^\perp$, it follows that
	$\Pi_{\bOmega^\perp}\bx^k\to\mathbf{0}$
	(cf.\ \cref{lemma:Pi_omegaperp_xk}), and hence
	$x_i^k-x_1^k\to0$ for every $i=2,\dots,n$.
	If the first rows of $\bN$ and $\bP$ are zero, then
	$
	x_1^k
	=
	J_{\gamma D_1^{-1}A_1}
	\bigl(D_1^{-1}(\bM\bz^k)_1\bigr).
	$
	Since $\{\bz^k\}$ is bounded by (i) and
	$J_{\gamma D_1^{-1}A_1}$ is Lipschitz continuous,
	$\{x_1^k\}$ is bounded. Consequently, since
	$x_i^k-x_1^k\to0$ for every $i=2,\dots,n$,
	$\{\bx^k\}$ is bounded.
\smartqedmark \end{proof}
The preceding theorem provides convergence guarantees for the
general implicit \cref{algo:full_implicit}, provided that
$\gamma\in\Lambda\cap\Gamma{}$. While $\Gamma{}$ can be characterized
or estimated using \cref{tab:Gammastar}, verifying the condition
$\gamma\in\Lambda$ may not be straightforward in the general implicit setting.
We therefore focus in the remainder of this subsection on the
explicit \cref{algo:full}. Its triangular structure allows membership
in $\Lambda$ to be verified through the individual resolvents
$J_{\gamma D_i^{-1}A_i}$. Under the following assumption, we obtain
explicit nonempty subsets of $\Lambda$; the proofs are deferred to
\cref{app:prop_Lambda_nonempty,app:prop_Lambda_nonempty_special}.

\begin{assumption}\label{assume:semimonotone}
	For $i=1,\dots,n$, the semimonotonicity modulus $(U_i,V_i)$ in \cref{assume:blanket_multioperator}(i) satisfies $U_iV_i = \theta_i \Id$ with $\theta_i<1/4$. 
\end{assumption}

\begin{lemma}[Estimate of $\Lambda$]\label{prop:Lambda_nonempty}
	Suppose that \cref{assume:blanket_multioperator}(i) and \cref{assume:semimonotone} hold, and let $\bD$ and $\bN$ be given by \eqref{eq:DN}, where $D_i$ is self-adjoint and $\delta_i$-strongly monotone for all $i\in \{ 1,\dots, n\}$, and  $\bN$ 
	is strictly lower triangular. Assume further that $\bP$ is strictly lower triangular and $\bR$ is lower triangular. For each $i\in \{1,\dots,n\}$, denote $
		\hat{U}_i \coloneqq \frac{1+2\varrho_i}{1+\varrho_i}U_i $ and $\hat{V}_i \coloneqq (1+2\varrho_i)V_i,$
	where 
	$\varrho_i \coloneqq  -\frac{1}{2}+\frac{1}{2\sqrt{1-4\theta_i}}$. For every $i$ such that $\hat V_i\not\succeq0$, assume that
\begin{equation}
	\label{eq:C(U,V)}
	\begin{cases}
		1+\|\hat U_i\|\lambda_{\min}(\hat V_i)>0,
		& \text{if }\hat U_i\succeq0,\\[1mm]
		-\lambda_{\min}(\hat V_i)
		\bigl(
		-\|D_i\|\lambda_{\min}(\hat U_i)
		+\delta_i\|\hat U_i\|
		\bigr)<\delta_i,
		& \text{if }\hat U_i\not\succeq0.
	\end{cases}
\end{equation} Let\footnote{We use the conventions
		$\max\varnothing=0$ and $\min\varnothing=+\infty$.}	
	\begin{equation}\label{eq:sigma_varsigma}
		\sigma \coloneqq \max \left\lbrace -\frac{\norm{D_i}\,\lambda_{\min}(\hat V_i)}
		{1+\norm{\hat U_i}\lambda_{\min}(\hat V_i)} : \hat{V}_i \not\succeq 0\right\rbrace ~\text{and} ~ \varsigma \coloneqq \min \left\lbrace -\frac{\delta_i}{\lambda_{\min}(\hat U_i)} : \hat{U}_i \not\succeq 0 \right\rbrace,
	\end{equation}
	and suppose that $\sigma < \varsigma$. Then for any $\gamma \in  (\sigma,\varsigma)$, the resolvent $J_{\gamma\bD^{-1}\bA}$ is single-valued, has full domain, and is Lipschitz continuous. Hence, $\varnothing \neq (\sigma,\varsigma) \subseteq \Lambda$ where $\Lambda$ is given by \eqref{eq:Lambda}. 
\end{lemma}
We now specialize to the case where each $D_i$ and all the $U_i$'s and $V_i$'s are scalar operators. In this regime, one can obtain a more refined sufficient interval
$(\sigma,\varsigma)\subseteq\Lambda$; see \cref{remark:scalar_UVD}.

\begin{lemma}[Estimate of $\Lambda$ for scalar $U_i,V_i,D_i$]\label{prop:Lambda_nonempty_specialcase}
	Suppose that \cref{assume:blanket_multioperator}(i) and \cref{assume:semimonotone} hold, and let $\bD = \diag(D_1,\dots,D_n)$, where $D_i=\delta_i \Id$  with $\delta_i>0$ for all $i\in \{ 1,\dots, n\}$. Assume further that $\bN$ and $\bP$
	are strictly lower triangular and that $\bR$ is lower triangular. Suppose that $(U_i,V_i)=(\mu_i\Id,\nu_i\Id)$ for all $i\in \{1,\dots,n\}$, and let $(\sigma,\varsigma)$ be given by 
	\begin{align}
		\sigma
		&\coloneqq
		\max\left\{
		\frac{-2\delta_j\nu_j}
		{1+\sqrt{1-4\mu_j\nu_j}}
		:\nu_j<0
		\right\},
		\notag\\
		\varsigma
		&\coloneqq
		\min\left\{
		\frac{-\delta_i(1+\sqrt{1-4\mu_i\nu_i})}
		{2\mu_i}
		:\mu_i<0
		\right\}.
		\label{eq:sigma_varsigma_specialcase}
	\end{align}
		and suppose that $\sigma<\varsigma$. Then for any $\gamma \in  (\sigma,\varsigma)$, the resolvent $J_{\gamma\bD^{-1}\bA}$ is single-valued, has full domain, and is Lipschitz continuous. Hence, $\varnothing \neq (\sigma,\varsigma) \subseteq \Lambda$ where $\Lambda$ is given by \eqref{eq:Lambda}. 	
\end{lemma}

\begin{remark}[Final range of admissible $\gamma$]
	\label{remark:gamma_range} \text{}
	\begin{enumerate}
		\item By \cref{thm:convergence_implicitalgorithm}, convergence
		holds whenever $\gamma\in\Lambda\cap\Gamma{}$. By the preceding
		lemmas, $\Lambda\supseteq(\sigma,\varsigma)$, while $\Gamma{}$
		is given in \cref{tab:Gammastar}. Thus, defining
		\begin{equation*}
			\Delta\coloneqq(\sigma,\varsigma)\cap\Gamma{},
		\end{equation*}
		we obtain an explicit admissible range of parameters for
		\cref{algo:full}.
		
		\item For ease of reference,
		\cref{tab:Lambdastar,tab:Lambdastar_scalarUV} summarize the
		interval $(\sigma,\varsigma)$ across the various cases, which are
		primarily the cases where our convergence results apply. In the
		four cases where $(\sigma,\varsigma)=(0,+\infty)$, we used the
		fact that $\bU\succeq0$ and $\bV\succeq0$ by
		\cref{lemma:schurcomplementlemma}. In
		\cref{tab:Lambdastar_scalarUV}, we recast the conditions
		``$\bU_{11}$ is strongly monotone'' and
		``$\bV_{22}\succeq0$, $\ran(\bV_{22})$ is closed and
		$\ran(\bV_{21})\subseteq\ran(\bV_{22})$'' into equivalent
		characterizations given in \cref{lemma:UV_specialcase}.
		
		\item Note that $\Delta$ is always nonempty, except possibly  when 
		$\bU/\bU_{11}\not\succeq0$ and
		$\bV/\bV_{22}\not\succeq0$.
		
		\item Since $\bN$ and $\bP$ are strictly lower triangular in
		\cref{algo:full}, their first rows are zero. Moreover, for
		$\gamma\in(\sigma,\varsigma)$,
		$J_{\gamma D_1^{-1}A_1}$ is Lipschitz continuous by
		\cref{prop:Lambda_nonempty,prop:Lambda_nonempty_specialcase}.
		Hence, the boundedness condition in
		\cref{thm:convergence_implicitalgorithm}(ii) is automatically
		satisfied for \cref{algo:full}.
	\end{enumerate}
\end{remark}

\begin{table}[ht]
	\centering
	\renewcommand{\arraystretch}{1.35}
	\setlength{\arrayrulewidth}{0.8pt}
	\caption{\small The interval $(\sigma,\varsigma)$, with $\sigma$ and $\varsigma$ given by \eqref{eq:sigma_varsigma}, provides an \textit{estimate} of the set of all $\gamma>0$ for which $J_{\gamma \bD^{-1}\bA}$ is single-valued, has full domain, and is Lipschitz continuous.}
	\label{tab:Lambdastar}
	\begin{threeparttable}
		\begin{tabular}{|c|c!{\vrule width \Thick}c!{\vrule width \Thick}c|c|}
			\hline
			\multicolumn{2}{|c!{\vrule width \Thick}}{} &
			\multirow{2}{*}{$\bU = 0$} &
			\multicolumn{2}{c|}{$\bU_{11}\ \text{strongly monotone}$}\\
			\cline{4-5}
			\multicolumn{2}{|c!{\vrule width \Thick}}{} &
			& \multicolumn{1}{c|}{$\bU/\bU_{11}\succeq 0$} &
			\multicolumn{1}{c|}{$\bU/\bU_{11}\not\succeq 0$}\\
			\HThick
			
			\multirow{3}{*}{$\bV_{22}\succeq 0$} &
			$\bV=0$ &
			\multicolumn{2}{c|}{\multirow{2}{*}{$\left(0,\;+\infty\right)$}} &
			\multirow{2}{*}{$\left(0,\;\varsigma \right)$}\\
			\cline{2-2}
			
			& $\bV/\bV_{22}\succeq 0$\quad \tnote{\textit{a}} &
			\multicolumn{2}{c|}{} & \\
			\cline{2-5}
			
			& $\bV/\bV_{22}\not\succeq 0$\quad \tnote{\textit{a}} &
			\multicolumn{2}{c|}{$(\sigma,+\infty)$ \tnote{\textit{b}}} &
			$\left(\sigma,\;\varsigma\right)$ \tnote{\textit{b}} \tnote{\textbf{c}}\\
			\hline
		\end{tabular}
		
		\begin{tablenotes}[flushleft]
			\footnotesize
			\item[\textit{a}] Additionally, it is assumed here that $\ran(\bV_{22})$ is closed and $\ran(\bV_{21})\subseteq \ran(\bV_{22})$.
			\item[\textit{b}] For all indices $i$ such that $\hat{V}_i\not\succeq 0$, we assume that condition \eqref{eq:C(U,V)} with $(\hat{U},\hat{V})=(\hat{U}_i,\hat{V}_i)$ holds.
			\item[\textit{c}] We assume that $\sigma<\varsigma$.
		\end{tablenotes}	
	\end{threeparttable}
\end{table}

\begin{table}[ht]
	\centering
	\renewcommand{\arraystretch}{1.35}
	\setlength{\arrayrulewidth}{0.8pt}
	\caption{\small The interval $(\sigma,\varsigma)$, with $\sigma$ and $\varsigma$ given by \eqref{eq:sigma_varsigma_specialcase}, provides an estimate of the set of all $\gamma>0$ for which $J_{\gamma \bD^{-1} \bA}$ is single-valued, has full domain, and is Lipschitz continuous, \textit{when $U_i$, $V_i$ and $D_i$ are scalar operators}. In view of \cref{lemma:UV_specialcase}(ii), $\sigma$ can be simplified as $\sigma= \frac{2\delta_j(-\nu_j)_+ }{1+\sqrt{1-4\mu_j\nu_j}}$ where $\nu_j = \min_i \nu_i$. 
	}
	\label{tab:Lambdastar_scalarUV}
	\begin{threeparttable}
		\begin{tabular}{|c!{\vrule width \Thick}c!{\vrule width \Thick}c|c|}
			\hline
			\multicolumn{1}{|c!{\vrule width \Thick}}{} &
			\multirow{2}{*}{$\mu_i=0$ ($\forall i$)} &
			\multicolumn{2}{c|}{$\sum_{i=1}^n \mu_i > 0$}\\
			\cline{3-4}
			\multicolumn{1}{|c!{\vrule width \Thick}}{} &
			& \multicolumn{1}{c|}{$\mu_i\geq 0$} &
			\multicolumn{1}{c|}{$\exists i$ s.t. $\mu_i<0$}\\
			\HThick
			
			$\nu_i=0$ ($\forall i$) &
			\multicolumn{2}{c|}{\multirow{2}{*}{$\left(0,\;+\infty\right)$}} &
			\multirow{2}{*}{$\left(0,\;\varsigma \right)$}\\
			\cline{1-1}
			
			$\nu_i\geq 0$ ($\forall i$) &
			\multicolumn{2}{c|}{} &
			\\
			\cline{1-4}
			
			$\exists j \text{ s.t. }\nu_j<0$, $\nu_i>0$ ($\forall i\neq j$), $\sum_{i=1}^n \nu_i^{-1}<0 $ &
			\multicolumn{2}{c|}{$(\sigma,+\infty)$} &
			$\left(\sigma,\;\varsigma\right)$ \tnote{\textit{a}}\\
			\hline
		\end{tabular}
		\begin{tablenotes}[flushleft]
			\footnotesize
			\item[\textit{a}] We assume that $\sigma<\varsigma$.
		\end{tablenotes}		
	\end{threeparttable}
\end{table}

\subsection{Comparisons with existing works}\label{sec:comparisons}

To the best of our knowledge, \cref{algo:full_implicit,algo:full} provide the first schemes for \eqref{eq:gen_prob_recall} at this level of generality. In particular, they allow non-scalar operators $M_{ij}$, $N_{ij}$, $P_{ij}$, $R_{ij}$, and $D_i$, thereby also accommodating preconditioned resolvents. By contrast, existing methods for \eqref{eq:gen_prob_recall} are restricted to scalar choices of these operators. Furthermore, in the $(U,V)$-semimonotone setting, including the cases where either $U=0$ or $V=0$, our work is the first to address the general algorithmic framework. \cref{tab:existingworks} provides a broad overview, and each related work is discussed in detail in the following subsections.  In particular, the existing works listed in the table treat only specific algorithmic instances.
\begin{table}[ht]
	\centering
	\renewcommand{\arraystretch}{1.35}
	\setlength{\arrayrulewidth}{0.8pt}
\caption{\small Overview of existing works. An entry indicates that the cited work treats an instance of the
	corresponding regime; its precise scope is summarized below the table. Prior related works address only restricted instances of the present framework. More precisely, they assume that $\bM=M\otimes \Id$, $\bL=L\otimes \Id$, $\bP=P\otimes \Id$, and $\bR=R\otimes \Id$ for some real matrices $M,L,P,R$ of appropriate dimensions, and that $\bD=D\otimes \Id$, so that standard resolvents of scalar multiples of the $A_i$'s are used. Moreover, the semimonotonicity moduli $(U_i,V_i)$ and cocoercivity modulus $W_j$ are also taken to be scalar operators.}
	\label{tab:existingworks}
\begin{threeparttable}
	\begin{tabular}{|c|p{4.8cm}!{\vrule width \Thick}c!{\vrule width \Thick}c|c|}
		\hline
		\multicolumn{2}{|c!{\vrule width \Thick}}{} &
		\multirow{2}{*}{$\mu_i=0$ $(\forall i)$} &
		\multicolumn{2}{c|}{$\sum_{i=1}^n \mu_i>0$} \\
		\cline{4-5}
		\multicolumn{2}{|c!{\vrule width \Thick}}{} &
		& \multicolumn{1}{c|}{$\mu_i\geq 0$ $(\forall i)$} &
		\multicolumn{1}{c|}{$\exists i$ s.t.\ $\mu_i<0$} \\
		\HThick
		
		\multicolumn{2}{|c!{\vrule width \Thick}}{$\bB\neq 0$, $\nu_i=0$ $(\forall i)$} &
		\shortstack{DTT26 \\ DP21 } & DP21 
		 & DP21 
		 \\
		\hline
		
		\multirow{3}{*}{$\bB=0$} &
		$\nu_i=0$ $(\forall i)$ &
		\shortstack{DTT26  \\ AT25  \\ ADT25 \\ BDP22 } &
		\shortstack{AT25  \\ DP19 } &
		\shortstack{AT25  \\ DP19} \\
		\cline{2-5}
		
		& $\nu_i\geq 0$ $(\forall i)$ &
		\shortstack{ADT25  \\ BDP22 } &
		\shortstack{EPLP25 } &
		\shortstack{EPLP25 } \\
		\cline{2-5}
		
		& \shortstack[l]{$\exists j$ s.t.\ $\nu_j<0$, $\nu_i>0$ $(\forall i\neq j)$,\\
			$\sum_{i=1}^n \nu_i^{-1}<0$} &
		\shortstack{ADT25  \\ BDP22 } &
		\shortstack{EPLP25} &
		\shortstack{EPLP25 } \\
		\hline
	\end{tabular}
	\begin{tablenotes}[flushleft]
		\footnotesize
		\item DTT26 \cite{DTT26} treats general $n$ and arbitrary $M$, $L$, $P$, and $R$.
		\item AT25 \cite{AT25} and ADT25 \cite{ADT25} treat general $n$, but only for one specific choice of $M$ and $L$.
		\item EPLP25 \cite{EPLP25} covers only the case $n=2$ and one specific choice of $M$ and $L$, and is restricted to finite-dimensional Hilbert spaces.
		\item DP19 \cite{DP18} and BDP22 \cite{BDP22} cover only the case $n=2$ and one specific choice of $M$ and $L$.
		\item DP21 \cite{DP21} treats only $(n,p)=(2,1)$ with one specific choice of $M$ and $L$. 
	\end{tablenotes}
\end{threeparttable}
\end{table}

\subsubsection{General operator splitting in \cite{DTT26}}
The work \cite{DTT26} considers the case where $\bM=M\otimes \Id$, $\bL=L\otimes \Id$, $\bP=P\otimes \Id$, and $\bR=R\otimes \Id$ for some real matrices $M,L,P,R$ of appropriate dimensions, and where $\bD=D\otimes \Id$. The assumptions in \cite[Assumption 3.2]{DTT26} coincide with \cref{assume:blanket_multioperator,assume:blanket2_components,assume:semimonotone} in the special case $U_i=V_i=0$ and $W_j=\omega\Id$ for all $j\in \{1, \dots, p\}$. In addition, in view of the decomposition of $\bL$ in \eqref{eq:DN}, \cite{DTT26} assumes that
\begin{equation}
	\label{eq:DTT_assumption}
	L_s-\frac{1}{2}MM^\top \succeq 0.
\end{equation}
In particular, \eqref{eq:DTT_assumption} is strictly stronger than \cref{assume:L_s}. Thus, our framework is based on weaker assumptions in several respects. Moreover, since $J_{\gamma D_1^{-1}A_1}$ is Lipschitz continuous
in this setting, when the first rows of $\bN$ and $\bP$ are zero,
\cref{thm:convergence_implicitalgorithm}(ii) recovers the same weak
convergence conclusion as \cite[Theorem~3.10]{DTT26}.

We now compare the parameter conditions ensuring weak convergence. According to \cref{thm:mainresult}, we require $\gamma\in \Lambda\cap \Gamma{}$ and $\lambda_k\in [0,2\beta_{\gamma}]$, where
\begin{equation*}
	\Gamma{} = \left(0,\frac{4}{\rho_1}\right) \quad \text{and} \quad \beta_{\gamma} = \norm{\left(\hat{\Xi}^{1/2}_{\gamma}\right)^\dagger M}^{-2},
\end{equation*}
with $\rho_1=\omega\,\lambda_{\max}\!\left((L_s^{1/2})^\dagger QQ^\top (L_s^{1/2})^\dagger\right)$ and $\hat{\Xi}_{\gamma}=L_s-\frac{\gamma\omega}{4}QQ^\top$, where $Q=P-R^\top$. On the other hand, \cite[Theorem 3.10]{DTT26} requires $\gamma\in \Lambda\cap \Gamma^{\rm DTT}$ and $\lambda_k\in [0,2\beta_{\gamma}^{\rm DTT}]$, where
\begin{equation*}
	\Gamma^{\rm DTT} = \left(0,\frac{2}{\omega \tau}\right) \quad \text{and} \quad \beta_{\gamma}^{\rm DTT} = \frac{2-\gamma \omega\tau}{4},
\end{equation*}
with $\tau=\norm{Q^\top (M^\top)^\dagger}^2$. Under the more restrictive assumption \eqref{eq:DTT_assumption}, the admissible ranges of both $\gamma$ and $\lambda_k$ obtained in this work are larger than those in \cite{DTT26}. Indeed, by \cref{prop:compare_gamma_upperbound},
\begin{equation*}
	\Gamma^{\rm DTT} \subseteq \Gamma{} \quad \text{and} \quad \beta_{\gamma}^{\rm DTT} \leq \beta_{\gamma},
\end{equation*}
and these relations can be strict; see \cref{ex:strictlylarger_gamma_beta}. In particular, when $p=0$ so that $Q=0$, we recover the classical range $\lambda_k\in [0,2]$ for the Douglas--Rachford algorithm for two operators, obtained by choosing $L$ and $M$ as in \cref{ex:strictlylarger_gamma_beta}, whereas \cite{DTT26} restricts $\lambda_k$ to $[0,1]$. 

We emphasize that these stronger convergence guarantees, obtained under weaker assumptions and in a framework that extends beyond the setting of \cite{DTT26}, are proved by simpler and more streamlined arguments. As shown in \cref{sec:proof_mainresult}, this is a consequence of our abstract subspace reformulation, which reduces the analysis to decompositions of the relevant operators with respect to the constraint set $\bOmega$ and its orthogonal complement. In particular, the weak convergence of the shadow sequence does not require a lengthy componentwise unwrapping of all $n$ resolvents appearing in the algorithm, but only the more concise subspace-based decomposition carried out in \cref{lemma:maximalmonotonedecomposition} followed by a direct application of \cref{lemma:demiclosednessprinciple}.

\subsubsection{Multioperator Douglas--Rachford algorithm under generalized monotonicity and comonotonicity}

Outside the monotone setting with $\bB=0$, to the best of our knowledge, the only works treating the multioperator case are \cite{AT25,ADT25}, and both are restricted to a specific algorithm. In \cite{AT25}, the algorithm is obtained by applying the Douglas--Rachford method to the weighted product space reformulation of \eqref{eq:productspace_recall} based on \cite{Cam22}. In the case of equal weights, this reduces to the multioperator Douglas--Rachford algorithm considered in \cite{ADT25}. Moreover, the equal-weight version is a special instance of \cref{algo:full}; see \cref{prop:DR_multioperator}.

As in the previous subsection, the convergence conclusions obtained in the present work recover those in \cite{AT25,ADT25}; the main difference lies in the assumptions and in the admissible parameter ranges, although it is not immediately clear to us whether the parameter conditions in \cite{AT25,ADT25} are more restrictive than those required in the present work. Nevertheless, several distinctions can already be observed in the scalar semimonotone setting. First, consider the case $\bU=\diag(\mu_1,\dots,\mu_n)\otimes \Id$ and $\bV=0$, which is the setting treated in \cite{AT25}. By \cref{thm:convergence_implicitalgorithm}(iii) and \cref{remark:gamma_range}(i), our convergence result applies whenever $\bU_{11}$ is strongly monotone, which, by \cref{lemma:UV_specialcase}, holds as soon as $\sum_{i=1}^n \mu_i>0$. In particular, the ordering of the moduli $\mu_i$ plays no role. By contrast, \cite[Theorem~4.14]{AT25} explicitly singles out the
last operator through the condition $\mu_n\neq0$, whereas our
condition is invariant under permutations of the $\mu_i$'s.

Next, consider the case $\bU=0$ and $\bV=\diag(\nu_1,\dots,\nu_n)\otimes \Id$, which is the setting treated in \cite{ADT25}. Again, our condition given by \cref{lemma:UV_specialcase} is invariant under permutations of the $\nu_i$'s. In particular, if exactly one modulus is negative, then that negative modulus may correspond to any one of the operators $A_i$. In contrast, the condition in \cite{ADT25} singles out the last operator: the negative modulus must be associated with $A_n$. 

Finally, we also emphasize that the present analysis is based on a shorter and more unified argument, whereas the proofs in \cite{AT25,ADT25} require more extensive case-by-case computations.

\subsubsection{Douglas--Rachford for two operators under semimonotonicity}
The first work treating the genuinely semimonotone case, that is, when $\bU$ and $\bV$ are not both zero, appears to be \cite{EPLP25}, which considers the Douglas--Rachford algorithm in the finite-dimensional setting and only for the case $n=2$. As noted in \cref{prop:DR_twooperator}, when $n=2$, the interval $\Gamma{}$ in \cref{tab:Gammastar} can be computed explicitly, whereas in the general case only estimates are available; see \cref{tab:Gammastar}. In this two-operator setting, the admissible range of $\gamma$ obtained here exactly coincides with that reported in \cite[Table 3]{EPLP25}, and the same is true for the relaxation range of $\lambda_k$; see also \cref{prop:DR_twooperator}.

The main point of departure is that our framework, when specialized to $n=2$, also covers the infinite-dimensional setting. In addition, our analysis is based on a simpler viewpoint: we treat the Douglas--Rachford method as a forward scheme, rather than as a preconditioned proximal point algorithm applied to the primal-dual reformulation of \eqref{eq:productspace_recall}.

In the infinite-dimensional setting, still with $n=2$, the weak convergence of the shadow sequence $\{\bx^k\}$ was established in \cite[Theorem 4.2]{BCP20} under the assumptions $\nu_1+\nu_2\geq 0$ and $\mu_1=\mu_2=0$. Their approach relies on alternative demiclosedness principles based on \cite[Proposition 20.60]{BC17}, a generalization of \cref{lemma:demiclosednessprinciple}, together with a weighted inner product on the product space. By contrast, for the case $\nu_1+\nu_2>0$, our proof of \cref{thm:mainconvergence_U=0} does not require such additional machinery or the accompanying lengthy calculations: it relies only on the subspace-based decomposition in \cref{lemma:maximalmonotonedecomposition} and a direct application of \cref{lemma:demiclosednessprinciple}.

\section{Proof of the main convergence theorem}\label{sec:proof_mainresult}

In this section, we discuss the proof of \cref{thm:mainresult}. Its three assertions are established separately in \cref{thm:weakconvergence_zk,thm:mainconvergence_U=0,thm:mainconvergence_U>0}. 

\subsection{Fundamental inequality}\label{sec:fundamentalinequality}

The forward scheme interpretation motivates a direct approach to weak convergence of $\{\bz^k\}$: it suffices to show that $\Psi_{\gamma}$ is cocoercive; cf. \cite[Theorem 26.14]{BC17} and \cite[Propositions 2.9 and 3.3]{BDP22}. For this, it is enough to prove that there exists $\beta>0$ such that
\begin{equation*}
	\inner{\bM^*\Delta \bx}{\Delta \bz} \geq \beta \norm{\bM^*\Delta \bx}^2
	= \inner{\beta \bM \bM^* \Delta \bx }{\Delta \bx},
\end{equation*}
where $\Delta \bz \coloneqq \bz - \bz'$, $\Delta \bx \coloneqq \bx - \bx'$, and $\bx =  \Phi_{\gamma}(\bM \bz)$, $\bx'= \Phi_{\gamma}(\bM \bz')$. To this end, we provide	 a lower bound for $\inner{\bM^*\Delta \bx}{\Delta \bz}$ expressed solely in terms of $\Delta \bx$.

\begin{lemma}
	\label{lemma:fundamentalinequality}
	Suppose that \cref{assume:blanket}(i),(ii),(iv) and \cref{assume:semimonotone_general} hold. Let $\gamma \in \Lambda$, where $\Lambda$ is given by \eqref{eq:Lambda}. For any $\bz,\bz' \in \bcalH'$, let $\bx =  \Phi_{\gamma}(\bM \bz)$, $\bx'= \Phi_{\gamma}(\bM \bz')$. We have 
	\begin{equation*}
		\label{eq:estimate}
		\inner{\Psi_{\gamma}(\bz) - \Psi_{\gamma}(\bz')}{\bz-\bz'} = \inner{\bM^* \Delta \bx}{\Delta \bz}\geq \inner{\bXi_{\gamma} \Delta \bx}{\Delta \bx},
	\end{equation*}
	where $\Delta \bz \coloneqq \bz - \bz'$, $\Delta \bx \coloneqq \bx - \bx'$ and $\bXi_{\gamma}$ is given by 
	\begin{align}
		 \bXi_{\gamma}   &\coloneqq \bL_s + \gamma \bU - \bTheta_{\gamma} \notag \\
		\bTheta_{\gamma} & \coloneqq  \begin{cases}
			\frac{\gamma}{4}(\bP - \bR^*) \bW (\bP^* - \bR)  & \text{in Case I}, \\
			- \frac{1}{\gamma} \bL^* \bQ \bL &\text{in Case II},
		\end{cases} \label{eq:bTheta}
	\end{align}
	with $\bQ \coloneqq\bV / \bV_{22}$.
\end{lemma}

\begin{proof}
	Since $\bM \bz + \bq -  \gamma \bP \bB (\bR \bx) -\bL \bx \in \gamma \bA \bx $ and $\bM \bz' + \bq -  \gamma \bP \bB (\bR \bx') -\bL \bx' \in \gamma \bA \bx' $, it follows from \cref{assume:blanket}(i) that 
	\[ \frac{1}{\gamma}\inner{\bM \Delta \bz - \gamma \bP \Delta \bs - \bL \Delta \bx}{ \Delta \bx} \geq  \inner{\Delta \bx}{\bU \Delta \bx}+ \frac{1}{\gamma^2}  \inner{\Delta \by}{\bV \Delta \by},\]
	where $\Delta \bs \coloneqq   \bB (\bR \bx) -  \bB (\bR \bx')$ and $\Delta \by \coloneqq \bM \Delta \bz - \gamma \bP \Delta \bs - \bL \Delta \bx$. It follows that 
	\begin{equation}
		\inner{\bM \Delta \bz }{ \Delta \bx} \geq  \gamma \inner{\bP \Delta \bs}{\Delta \bx}  + \inner{\Delta \bx}{\bL \Delta \bx}+\gamma \inner{\Delta \bx}{\bU \Delta \bx} +\frac{1}{\gamma}  \inner{\Delta \by}{\bV \Delta \by}.
		\label{eq:monotonicityinequality}
	\end{equation}
	To obtain a lower bound for the first term,
	we write
	\begin{align}
		\inner{\bP \Delta \bs}{\Delta \bx} & = \inner{\Delta \bs}{\bR \Delta \bx} + \inner{\Delta \bs}{(\bP^*-\bR)\Delta \bx}\notag \\
		& \geq \norm{\Delta \bs}_{\bW^{-1}}^2 - \left( \frac{1}{4}\norm{(\bP^*-\bR)\Delta \bx}_{\bW}^2 + \norm{\Delta \bs}^2_{\bW^{-1}} \right)\notag \\
		& = - \frac{1}{4}\norm{(\bP^*-\bR)\Delta \bx}_{\bW}^2 \label{eq:PBR_monotone}
	\end{align}
	where the inequality holds by cocoercivity of $\bB$ (\cref{assume:blanket}(ii)) and Young's inequality (\cref{lemma:Young}). The  last term in \eqref{eq:monotonicityinequality} is zero in Case I.
	In Case II, note that $\Delta \by + \bL \Delta \bx \in \ran (\bM) = \bOmega^{\perp}$ by \cref{assume:blanket}(iv). Invoking \cref{lemma:QP}, we have $\inner{\Delta \by}{\bV \Delta \by} \geq \inner{\bL \Delta \bx}{\bQ \bL \Delta \bx}$. This proves the claim. 
\smartqedmark \end{proof}

\subsection{Weak convergence of the main sequence}\label{sec:weak_z}

The following lemma converts the estimate in
\cref{lemma:fundamentalinequality} into a sufficient condition for the
cocoercivity of $\Psi_\gamma$.

\begin{lemma}[Cocoercivity of $\Psi_{\gamma}$]
	\label{lemma:cocoercivityofPsigamma}
	Suppose that \cref{assume:blanket}(i),(ii),(iv) and
	\cref{assume:semimonotone_general} hold, and let
	$\gamma\in\Lambda$. Suppose that there exists a self-adjoint
	operator $\widehat{\bXi}_{\gamma}\preceq\bXi_{\gamma}$ such that
	$\ker(\widehat{\bXi}_{\gamma})\subseteq\bOmega$ and
	$\widehat{\bXi}_{\gamma}$ is strongly monotone on
	$\ker(\widehat{\bXi}_{\gamma})^\perp$. If
	$
	\beta_{\gamma}
	\coloneqq
	\norm{
		\left(\widehat{\bXi}_{\gamma}^{1/2}\right)^\dagger
		\bM
	}^{-2},
	$
	then $\Psi_{\gamma}$ is $\beta_{\gamma}$-cocoercive.
\end{lemma}

\begin{proof}
	Since
	$\ker(\widehat{\bXi}_{\gamma})
	\subseteq\bOmega=\ker(\bM^*)$,
	\cref{lemma:beta_existence}(i) gives
	$
	\widehat{\bXi}_{\gamma}
	-
	\beta_{\gamma}\bM\bM^*
	\succeq 0.
	$
	On the other hand, by
	\cref{lemma:fundamentalinequality} and
	$\widehat{\bXi}_{\gamma}\preceq\bXi_{\gamma}$,
	\begin{align*}
		\inner{
			\Psi_{\gamma}(\bz)-\Psi_{\gamma}(\bz')
		}{
			\bz-\bz'
		}
		&\geq
		\inner{\bXi_{\gamma}\Delta\bx}{\Delta\bx}\geq
		\inner{\widehat{\bXi}_{\gamma}\Delta\bx}{\Delta\bx}\geq
		\beta_{\gamma}
		\inner{\bM\bM^*\Delta\bx}{\Delta\bx}
		\end{align*}
	That is, $	\inner{
		\Psi_{\gamma}(\bz)-\Psi_{\gamma}(\bz')
	}{
		\bz-\bz'
	} \geq \beta_{\gamma}
	\norm{
	\Psi_{\gamma}(\bz)-\Psi_{\gamma}(\bz')
	}^2.$
	Thus, $\Psi_{\gamma}$ is $\beta_{\gamma}$-cocoercive.
\smartqedmark \end{proof}

The following proposition provides an explicit subset of $\Gamma$.
\begin{proposition}
	\label{prop:Gammastar_cocoercivity}
	Suppose that \cref{assume:blanket,assume:preconditioner,assume:semimonotone_general} hold. Let $\Lambda$ be given by \eqref{eq:Lambda}, $\bTheta_{\gamma}$ be given by \eqref{eq:bTheta},
	\begin{equation}
		\widehat{\bXi}_{\gamma}
		\coloneqq \bL_s-\bTheta_{\gamma}+\gamma\bX, \text{ where } \bX\coloneqq \bU/\bU_{11}, \label{eq:widehatXi_and_X}
	\end{equation}
	and define
	\begin{equation}
		\Gamma{}
		\coloneqq
		\left\lbrace
		\gamma>0:
		\lambda_{\max}\left(
		(\bL_s^{1/2})^{\dagger}
		\left(\bTheta_{\gamma}-\gamma\bX\right)
		(\bL_s^{1/2})^{\dagger}
		\right)<1
		\right\rbrace.\label{eq:Gamma*_proofsection}
	\end{equation}
	Then 
	\begin{enumerate}
		\item $\bY\coloneqq \bU-\bX\succeq0$. 
		\item If $\gamma\in\Lambda\cap\Gamma{}$ and
		$\beta_{\gamma}
			\coloneqq
			\norm{
				\left(\widehat{\bXi}_{\gamma}^{1/2}\right)^{\dagger}\bM
			}^{-2},$
		then $\Psi_{\gamma}$ is
		$\beta_{\gamma}$-cocoercive.
		\item $\Gamma{}$ can be written
		explicitly as follows:
		\begin{enumerate}
			\item In Case I, $\Gamma{} = \left(0, \frac{4}{\max\{ 0,\rho\}}\right)$ where $$
				\rho\coloneqq
				\lambda_{\max}\left(
				(\bL_s^{1/2})^{\dagger}
				\left[
				(\bP-\bR^*)\bW(\bP^*-\bR)-4\bX
				\right]
				(\bL_s^{1/2})^{\dagger}
				\right).$$
			
			\item In Case II, $$
				\Gamma{}
				=
				\left\lbrace
				\gamma>0:
				\gamma+
				\lambda_{\min}\left(
				(\bL_s^{1/2})^{\dagger}
				\left[
				\gamma^2\bX+
				\bL^*(\bV/\bV_{22})\bL
				\right]
				(\bL_s^{1/2})^{\dagger}
				\right)>0
				\right\rbrace.$$
		\end{enumerate}
	\end{enumerate}

\end{proposition}

\begin{proof}
	We first show that \(\bY\succeq0\). If
	\(\bU\equiv0\), then \(\bX=0\), and the claim is
	immediate. Otherwise, we write the block decomposition of \(\bU\)
	with respect to \(\bcalH=\bOmega\oplus\bOmega^\perp\), as
	$
		\bU
		=
		\begin{pmatrix}
			\bU_{11} & \bU_{12}\\
			\bU_{21} & \bU_{22}
		\end{pmatrix}.
	$
	Noting that $\bOmega\subseteq\ker(\bX)$, then the block decomposition of $\bX$ is given by
	\begin{equation*}
		\bX
		=
		\bU/\bU_{11}
		=
		\begin{pmatrix}
			0 & 0\\
			0
			&
			\bU_{22}-\bU_{21}\bU_{11}^{-1}\bU_{12}
		\end{pmatrix},
	\end{equation*}
	where $\bU_{11}^{-1}$ is well-defined since $\bU_{11}$ is strongly monotone. 
	Consequently, the block decomposition of $\bY = \bU-\bX$ is given by 
	\begin{equation*}
		\bY = \bU-\bX
		=
		\begin{pmatrix}
			\bU_{11} & \bU_{12}\\
			\bU_{21}
			&
			\bU_{21}\bU_{11}^{-1}\bU_{12}
		\end{pmatrix}.
	\end{equation*}
	By \cref{lemma:schurcomplementlemma}(i), $
		\bY\succeq0 $ if and only if $\bY/ \bU_{11}\succeq 0$. The latter condition holds since $\bY / \bU_{11}= 0$.
	Therefore,
		$\bY
		\succeq0.$ This proves (i). 	
	
	To prove (ii), we first observe that
	\begin{equation}
		\bOmega\subseteq\ker(\bTheta_{\gamma}).
		\label{eq:Omega_in_ker(Theta)}
	\end{equation}
	In Case I, this follows from \cref{assume:blanket}(iii). In
	Case II, we have
	$\bL(\bOmega)\subseteq\bOmega^\perp$ by \cref{assume:blanket2}, while $\bQ=\bV/\bV_{22}$ vanishes on
	$\bOmega^\perp$. With these, $\bQ\bL$ vanishes on $\bOmega$, and
	thus $\bL^*\bQ\bL$ also vanishes on $\bOmega$. This proves
	\eqref{eq:Omega_in_ker(Theta)} in Case II.
	
	Denote
	$
		\bZ_{\gamma}
		\coloneqq
		\bTheta_{\gamma}-\gamma\bX,
$ so that $\widehat{\bXi}_{\gamma}
=
\bL_s-\bZ_{\gamma}.$
	By \eqref{eq:Omega_in_ker(Theta)} and
	$\bOmega\subseteq\ker(\bX)$, we have
	$
		\bOmega\subseteq\ker(\bZ_{\gamma}).
	$
	If $\gamma\in\Gamma{}$, then
	$
		\lambda_{\max}\left(
		(\bL_s^{1/2})^{\dagger}
		\bZ_{\gamma}
		(\bL_s^{1/2})^{\dagger}
		\right)<1.
	$
	It follows from \cref{lemma:beta_existence}(ii) and
	\cref{assume:L_s} that
	$
	\ker(\widehat{\bXi}_{\gamma})
	=
	\ker(\bL_s)
	=
	\bOmega,
	$
	and that $\widehat{\bXi}_{\gamma}$ is strongly monotone on
	$\bOmega^\perp$. Moreover, $\widehat{\bXi}_{\gamma}$ is
	self-adjoint. Since
	$
	\bXi_{\gamma}
	=
	\widehat{\bXi}_{\gamma}
	+\gamma(\bU-\bX)
	=
	\widehat{\bXi}_{\gamma}+\gamma \bY
	$
	and $\bY\succeq0$ by part~(i), we have
	$\widehat{\bXi}_{\gamma}\preceq\bXi_{\gamma}$.
	Hence, the claim follows from
	\cref{lemma:cocoercivityofPsigamma}.
	
	Finally, for (iii), we derive the explicit expressions for $\Gamma{}$. In
	Case I, $
		\bZ_{\gamma}
		=
		\frac{\gamma}{4}
		\left[
		(\bP-\bR^*)\bW(\bP^*-\bR)-4\bX
		\right].$
	Therefore,
	\begin{align*}
		&\lambda_{\max}\left(
		(\bL_s^{1/2})^{\dagger}
		\bZ_{\gamma}
		(\bL_s^{1/2})^{\dagger}
		\right)<1\\
		&\quad\Longleftrightarrow\quad
		\frac{\gamma}{4}
		\lambda_{\max}\left(
		(\bL_s^{1/2})^{\dagger}
		\left[
		(\bP-\bR^*)\bW(\bP^*-\bR)-4\bX
		\right]
		(\bL_s^{1/2})^{\dagger}
		\right)<1\\
		&\quad\Longleftrightarrow\quad
		\frac{\gamma\rho}{4}<1.
	\end{align*}
	If \(\rho\leq0\), the last inequality holds for every
	\(\gamma>0\). If \(\rho>0\), it is equivalent to
	\(0<\gamma<4/\rho\). This proves the expression for \(\Gamma{}\)
	in Case I.
	
	In Case II,
	\begin{equation*}
		\bZ_{\gamma}
		=
		-\frac{1}{\gamma}\bL^*\bQ\bL-\gamma\bX
		=
		-\frac{1}{\gamma}
		\left[
		\bL^*\bQ\bL+\gamma^2\bX
		\right].
	\end{equation*}
	Therefore,
	\begin{align*}
		&\lambda_{\max}\left(
		(\bL_s^{1/2})^{\dagger}
		\bZ_{\gamma}
		(\bL_s^{1/2})^{\dagger}
		\right)<1\\
		&\quad\Longleftrightarrow\quad
		-\frac{1}{\gamma}
		\lambda_{\min}\left(
		(\bL_s^{1/2})^{\dagger}
		\left[
		\gamma^2\bX+\bL^*\bQ\bL
		\right]
		(\bL_s^{1/2})^{\dagger}
		\right)<1\\
		&\quad\Longleftrightarrow\quad
		\gamma+
		\lambda_{\min}\left(
		(\bL_s^{1/2})^{\dagger}
		\left[
		\gamma^2\bX+\bL^*\bQ\bL
		\right]
		(\bL_s^{1/2})^{\dagger}
		\right)>0.
	\end{align*}
	Since \(\bQ=\bV/\bV_{22}\), this proves the expression for
	\(\Gamma{}\) in Case II.
\smartqedmark \end{proof}

\begin{remark}\label{rem:Gamma-star-discussion}
	The set $\Gamma{}$, or an estimate thereof, is summarized in \cref{tab:Gammastar}. In particular, $\Gamma{}$ is nonempty in all cases except possibly when $\bU/\bU_{11}\not\succeq 0$ and $\bV/\bV_{22}\not\succeq 0$. 	
	\begin{enumerate}		
		\item Consider the constants $\rho_1$, $\rho_2$, $\rho_3$ defined in \cref{tab:notation}. Since $\bW \succeq 0$, we have $\rho_1\geq 0$. On the other hand, since $\bOmega=\ker(\bL_s)=\ker(\bL_s^{1/2})=\ker\!\big((\bL_s^{1/2})^\dagger\big)$, it follows that $\rho_2\leq 0$ and $\rho_3\leq 0$.
		
		\item The three estimates in \cref{tab:Gammastar} (marked with $^\ddagger$) are obtained by reducing the relevant spectral conditions to scalar inequalities. Indeed, note that
		\[
		\lambda_{\min}\!\left( (\bL_s^{1/2})^\dagger\,[\gamma^2 \bX+\bL^*\bQ\bL]\,(\bL_s^{1/2})^\dagger \right)
		\;\ge\;
		\rho_2\gamma^2+\rho_3,
		\]
		and therefore $\gamma\in\Gamma{}$ whenever $\rho_2\gamma^2+\gamma+\rho_3>0$. Hence, a sufficient condition for $\Gamma{}$ to be nonempty when $\bU/\bU_{11}\not\succeq 0$ and $\bV/\bV_{22}\not\succeq 0$ is $\rho_2\rho_3<1/4$. 
	\end{enumerate}
\end{remark}

With the above result, we immediately obtain
\cref{thm:mainresult}(i), together with the additional
Fej\'er-monotonicity and rate statements below.

\begin{theorem}[Weak convergence of the main sequence]
	\label{thm:weakconvergence_zk}
	Suppose that \cref{assume:blanket,assume:preconditioner,assume:semimonotone_general} hold. Let $\Lambda$ and $\Gamma{}$ be given by \eqref{eq:Lambda} and \eqref{eq:Gamma*_proofsection}, respectively, and let $\gamma \in \Lambda\cap \Gamma{}$. Let $\widehat{\bXi}_{\gamma} $ be given by \eqref{eq:widehatXi_and_X}, and denote $\beta_\gamma \coloneqq \norm{\left(\widehat{\bXi}_{\gamma}^{1/2}\right)^\dagger\bM}^{-2}$. 
	If $\lambda_k\in [0,2\beta_\gamma]$ satisfies $\sum_{k=0}^{\infty} \lambda_k (2\beta_\gamma-\lambda_k) = +\infty$,  then the following hold: 
	\begin{enumerate}
		\item $\{\bz^k\}$ is Fej\'er monotone with respect to $\zer(\Psi_{\gamma})$ (and therefore bounded).
		\item $\{ \Psi_{\gamma}(\bz^k)\} $ converges strongly to zero.
		\item $\{ \bz^k\}$ converges weakly to some point $\bz^*\in \zer(\Psi_\gamma)$.
		\item If $\liminf_{k\to\infty}\lambda_k (2\beta_\gamma-\lambda_k) >0$, then $\norm{\Psi_{\gamma}(\bz^k)} = o (1/\sqrt{k})$.
	\end{enumerate}
\end{theorem}
\begin{proof}
	By the standing solvability assumption and
	\cref{prop:equivalence_MLauxiliary},
	$\zer(\Psi_{\gamma})\neq\varnothing$.
	Moreover, \cref{prop:Gammastar_cocoercivity} asserts that
	$\Psi_{\gamma}$ is $\beta_{\gamma}$-cocoercive.
	Then claims (i)--(iv) follow from combining
	\cite[Propositions 2.9 and 3.3]{BDP22}.
\smartqedmark \end{proof}

\subsection{Weak convergence of the shadow sequence}\label{sec:weak_x}

We now establish the weak convergence of $\{ \bx^k\}$ when $\bU\equiv  0$. To this end, we first recast $\bx = \Phi_{\gamma}(\bM \bz)$ as a monotone inclusion with a maximally monotone operator.

\begin{lemma}
	\label{lemma:maximalmonotonedecomposition}
	Suppose that \cref{assume:blanket}(ii)-(iv), \cref{assume:blanket2}, and
	\cref{assume:semimonotone_general} hold, and let $\gamma \in \Lambda$. Then
	\begin{equation}
		\bx = \Phi_{\gamma}(\bM \bz) \quad \Longleftrightarrow \quad  \bv (\bx,\bz) \in \mathbf{T}\bu(\bx,\bz),
		\label{eq:maximalmonotoneform}
	\end{equation}
	where
	\begin{enumerate}
		\item In Case I, $\mathbf{T}\coloneqq \left(\bA+\bP \bB\bR + \widetilde{\bU} \right)^{-1}$ with $\widetilde{\bU} \coloneqq \frac{1}{4}(\bP-\bR^*)\bW(\bP^* - \bR)$, and
		\begin{equation}
			\begin{array}{l}
				\bu(\bx,\bz) \coloneqq \frac{1}{\gamma}\underbrace{\big(\bM \bz - \Pi_{\bOmega^\perp} ( \bL \bx - \bq) + \gamma \Pi_{\bOmega^\perp} \widetilde{\bU} \Pi_{\bOmega^\perp}\bx\big)}_{\in \bOmega^\perp}
				- \frac{1}{\gamma} \Pi_{\bOmega}\bL \Pi_{\bOmega^\perp}(   \bx -\bar{\bx}),\\
				\bv(\bx,\bz) \coloneqq \bx .
			\end{array}
			\label{eq:uv_case1}
		\end{equation}
		\item In Case II, $\mathbf{T} \coloneqq \bA^{-1} - \widetilde{\bV} $ with $\widetilde{\bV} \coloneqq \bV - \bV_{22}$, and
		\begin{equation}
			\begin{array}{l}
				\bu(\bx,\bz)\coloneqq  \frac{1}{\gamma}(\underbrace{\bM\bz -  \Pi_{\bOmega^\perp} ( \bL \bx - \bq)}_{\in \bOmega^\perp} - \Pi_{\bOmega}\bL \Pi_{\bOmega^\perp}(   \bx -\bar{\bx})), \\
				\bv(\bx,\bz) \coloneqq \bx - \frac{1}{\gamma}\widetilde{\bV}(\underbrace{\bM\bz - \Pi_{\bOmega^\perp} ( \bL \bx - \bq)}_{\in \bOmega^\perp} - \Pi_{\bOmega}\bL \Pi_{\bOmega^\perp}(\bx-\bar{\bx}) ),
			\end{array}
			\label{eq:uv_case2}
		\end{equation}
	\end{enumerate}
	where  $\mathbf{T}$ is maximally monotone if $\bA$ is maximally $\bV$-comonotone. Moreover, there exists $c>0$ such that  
	\begin{equation}   
		\max\{ \norm{\Pi_{\bOmega}\bu(\bx,\bz)},  \norm{\Pi_{\bOmega^\perp}\left( \bv(\bx,\bz)-\bar{\bx}\right) }\} \leq c\,\norm{\Pi_{\bOmega^\perp}(\bx-\bar{\bx})}.
		\label{eq:bounds_uv_projections}
	\end{equation}
\end{lemma}
\begin{proof}
	We first note the following identities. Using the decomposition w.r.t. $\bOmega \oplus \bOmega^\perp$ and \cref{assume:blanket2}, we have  \begin{align*}
		\bL \bx - \bq
		& = \Pi_{\bOmega} (\bL \Pi_{\bOmega} \bx + \bL \Pi_{\bOmega^\perp} \bx - \bq) + \Pi_{\bOmega^\perp} ( \bL  \bx-\bq)\notag  \\ 
		& = \Pi_{\bOmega}( \bL \Pi_{\bOmega^\perp} \bx -\bq) + \Pi_{\bOmega^\perp} ( \bL \bx - \bq) \\
		& = \Pi_{\bOmega}( \bL \Pi_{\bOmega^\perp} \bx -\bL\bar{\bx}) + \Pi_{\bOmega^\perp} ( \bL  \bx - \bq)  \\
		& = \Pi_{\bOmega}( \bL \Pi_{\bOmega^\perp} \bx -\bL\Pi_{\bOmega^\perp}\bar{\bx}) + \Pi_{\bOmega^\perp} ( \bL  \bx - \bq), \label{eq:Lx}
	\end{align*}
	where the second equation holds since $\bL\Pi_{\bOmega}\bx\in \bOmega^\perp$,  the third equation holds since $\bL\bar{\bx} - \bq \in \bOmega^\perp$, and the last holds since $\bL\Pi_{\bOmega}\bar{\bx}\in \bOmega^\perp$. Similarly, we have $\widetilde{\bU} =\Pi_{\bOmega^\perp} \widetilde{\bU} \Pi_{\bOmega^\perp}$ by \cref{assume:blanket}(iii). With these identities, the equivalence \eqref{eq:maximalmonotoneform} follows by rearranging the terms of the inclusion $\bM \bz +\bq \in \gamma \bA \bx + \gamma \bP\bB (\bR \bx) + \bL \bx$ to isolate the operator $\mathbf{T}$. We now prove the rest of the claim for each case.
	
	\noindent\textit{Case I.} 
	$\bP\bB\bR + \widetilde{\bU}$ is monotone by \eqref{eq:PBR_monotone}; therefore, since it is continuous, it is maximally monotone (with full domain) by \cref{lemma:maximalmonotone_properties}(ii). Consequently, $\mathbf{T}=(\bA+\bP\bB\bR + \widetilde{\bU})^{-1}$ is maximally monotone by \cref{lemma:maximalmonotone_properties}(iii)--(iv). Meanwhile, since $ \Pi_{\bOmega}(\bu(\bx,\bz)) = -\frac{1}{\gamma}  \Pi_{\bOmega}\bL \Pi_{\bOmega^\perp} (\bx-\bar{\bx})$, \eqref{eq:bounds_uv_projections} holds with $c \coloneqq \max\{ 1, \frac{1}{\gamma}\norm{\bL}\}$. 
	
	\medskip
	\noindent\textit{Case II.} Since $\widetilde{\bV} = \bV - \Pi_{\bOmega^\perp} \bV \Pi_{\bOmega^\perp} $, then $\bV - \widetilde{\bV} = \bV_{22} \succeq 0$. Thus, 
	$$\mathbf{T} = \bA^{-1} -\widetilde{\bV} = (\bA^{-1} - \bV ) + (\bV - \widetilde{\bV})$$ 
	is maximally monotone by \cref{lemma:V-comonotone} and \cref{lemma:maximalmonotone_properties}(iv).  To prove \eqref{eq:bounds_uv_projections}, note that 
	\begin{align*}
		\Pi_{\bOmega}(\bu(\bx,\bz)) & = 
		-\frac{1}{\gamma} \Pi_{\bOmega} \bL \Pi_{\bOmega^\perp}(\bx-\bar{\bx}) \quad \text{and} \\
		\Pi_{\bOmega^\perp} (\bv(\bx,\bz)-\bar{\bx}) & = 
		\Pi_{\bOmega^\perp}(\bx-\bar{\bx}) + \frac{1}{\gamma}\Pi_{\bOmega^\perp}\widetilde{\bV}\Pi_{\bOmega}\bL\Pi_{\bOmega^\perp}(\bx-\bar{\bx}) ,
	\end{align*}
	where the second identity holds since $\bOmega^\perp \subseteq \ker(\Pi_{\bOmega^\perp}\widetilde{\bV})$. Hence, \eqref{eq:bounds_uv_projections} holds with $c\coloneqq \max\{\frac{1}{\gamma}\norm{\bL},1+\frac{1}{\gamma}\norm{\widetilde{\bV}\bL} \}$. This completes the proof.
\smartqedmark \end{proof}

We now establish the convergence of  $\{ \bx^k\}$ in the following theorem, which proves \cref{thm:mainresult}(ii).

\begin{theorem}[Weak convergence of the shadow sequence when $\bU\equiv 0$]
	\label{thm:mainconvergence_U=0}
	Suppose that the hypotheses of \cref{thm:weakconvergence_zk} hold. Moreover, assume that $\bU=0$ and $\bA$ is maximally $\bV$-comonotone. If $\{\bx^k\}$ is bounded, then there exists a solution $\bx^*$ of \eqref{eq:productspace_recall} such that $\bx^k\toweak \bx^*$. 
\end{theorem}

\begin{proof}
Since $\{\bx^k\}$ is bounded, we can let
$\{\bx^{k_j}\}$ be an arbitrary weakly convergent subsequence, say
$\bx^{k_j}\toweak\bx^*$.
	Since $\Psi_{\gamma}(\bz^k)=\bM^*\bx^k - \bb \to 0$ by \cref{thm:weakconvergence_zk}(ii), it follows that $\bx^* \in \hatOmega$ and in addition, $\Pi_{\bOmega^\perp}(\bx^k - \bar{\bx})\to 0$ by \cref{lemma:Pi_omegaperp_xk}. 
	On the other hand, by \cref{thm:weakconvergence_zk}(iii), there exists $\bz^*\in\zer(\Psi_{\gamma})$ such that $\bz^{k}\toweak \bz^*$, and hence $\bz^{k_j}\toweak \bz^*$.
	
	By \cref{lemma:maximalmonotonedecomposition}, the relation $\bx^k=\Phi_{\gamma}(\bM\bz^k)$ is equivalent to $\bv^k\in\mathbf{T}\bu^k$, where $\bu^k = \bu(\bx^k,\bz^k)$ and $\bv^k = \bv(\bx^k,\bz^k)$, with $(\bu,\bv)$ given in \eqref{eq:uv_case1} and \eqref{eq:uv_case2}.
	Since affine maps with bounded linear parts are weakly continuous, $\bu^{k_j} \toweak \bu(\bx^*,\bz^*)$ and $\bv^{k_j}\toweak \bv(\bx^*,\bz^*)$.
	Furthermore, $\Pi_{\bOmega}\bu^k\to0$ and $ \Pi_{\bOmega^\perp}(\bv^k-\bar{\bx}) \to0$ by \eqref{eq:bounds_uv_projections}, noting that $\Pi_{\bOmega^\perp}(\bx^k-\bar{\bx})\to 0$. Since $\mathbf{T}$ is maximally monotone by \cref{lemma:maximalmonotonedecomposition}, we conclude from  \cref{lemma:demiclosednessprinciple} that $\bv(\bx^*,\bz^*) \in \mathbf{T} \bu(\bx^*,\bz^*)$, and therefore $\bx^*=\Phi_{\gamma}(\bM\bz^*)$ by invoking again \cref{lemma:maximalmonotonedecomposition}. Since $\gamma\in\Lambda$, the point
	$\Phi_{\gamma}(\bM\bz^*)$ is unique. Hence, as the weakly
	convergent subsequence was arbitrary, every weak cluster point of
	$\{\bx^k\}$ coincides with $\Phi_{\gamma}(\bM\bz^*)$.
	Consequently, $\bx^k\toweak\bx^*$.
\smartqedmark \end{proof}

\subsection{Strong convergence of the shadow sequence}\label{sec:strong_x}

Finally, we prove that $\{\bx^k\}$ is strongly convergent when $\bU$ is strongly monotone on $\bOmega$, which is precisely the claim of \cref{thm:mainresult}(iii). 
\begin{theorem}[Strong convergence of the shadow sequence when $\bU$ is strongly monotone on $\bOmega$]\label{thm:mainconvergence_U>0}
	Suppose that the hypotheses of \cref{thm:weakconvergence_zk} hold. If $\bU$ is strongly monotone on $\bOmega$, then there exists a solution $\bx^*$ of \eqref{eq:productspace_recall} such that $\bx^k\to \bx^*$.
\end{theorem}

\begin{proof}
	Let $\bz^*\in\zer(\Psi_\gamma)$ and set
	$\bx^*\coloneqq\Phi_\gamma(\bM\bz^*)$. Then $\bx^*$ solves
	\eqref{eq:productspace_recall} by
	\cref{prop:equivalence_MLauxiliary}. Meanwhile, recall that $\bU=\bX+\bY$, where
	$\bX=\bU/\bU_{11}$ and $\bY\succeq0$ (see \cref{prop:Gammastar_cocoercivity}(i)). Moreover,
	$\bXi_\gamma=\widehat{\bXi}_\gamma+\gamma\bY$, with
	$\widehat{\bXi}_\gamma\succeq\beta_\gamma\bM\bM^*$.
	Hence, by \cref{lemma:fundamentalinequality},
	\begin{align*}
		0
		&\leq
		\gamma\inner{\bY(\bx^k-\bx^*)}{\bx^k-\bx^*} \\
		&\leq
		\inner{\bXi_\gamma(\bx^k-\bx^*)}{\bx^k-\bx^*}
		\leq
		\inner{\Psi_\gamma(\bz^k)}{\bz^k-\bz^*}.
	\end{align*}
	Since $\{\bz^k\}$ is bounded by
	\cref{thm:weakconvergence_zk}(i) and
	$\Psi_\gamma(\bz^k)\to0$ by
	\cref{thm:weakconvergence_zk}(ii), the right-hand side converges
	to zero. Therefore,
	$\inner{\bY(\bx^k-\bx^*)}{\bx^k-\bx^*}\to0$, or equivalently,
	$\bY^{1/2}(\bx^k-\bx^*)\to0$.
	
	On the other hand,
	$\Psi_\gamma(\bz^k)=\bM^*\bx^k-\bb\to0$ and
	$\bx^*\in\hatOmega$. Thus,
	$\Pi_{\bOmega^\perp}(\bx^k-\bx^*)\to0$ by
	\cref{lemma:Pi_omegaperp_xk}. Since $\bY^{1/2}$ is bounded,
	$\bY^{1/2}\Pi_{\bOmega^\perp}(\bx^k-\bx^*)\to0$. Consequently,
	\[
	\bY^{1/2}\Pi_{\bOmega}(\bx^k-\bx^*)
	=
	\bY^{1/2}(\bx^k-\bx^*)
	-
	\bY^{1/2}\Pi_{\bOmega^\perp}(\bx^k-\bx^*)
	\to0.
	\]
	Finally, $\bOmega\subseteq\ker(\bX)$, so $\bY=\bU$ on
	$\bOmega$. Since $\bU$ is strongly monotone on $\bOmega$, there
	exists $\alpha>0$ such that
	\[
	\alpha\norm{\Pi_{\bOmega}(\bx^k-\bx^*)}^2
	\leq
	\inner{
		\bY\Pi_{\bOmega}(\bx^k-\bx^*)
	}{
		\Pi_{\bOmega}(\bx^k-\bx^*)
	}
	=
	\norm{
		\bY^{1/2}\Pi_{\bOmega}(\bx^k-\bx^*)
	}^2
	\to0.
	\]
	Hence, $\Pi_{\bOmega}(\bx^k-\bx^*)\to0$. Together with
	$\Pi_{\bOmega^\perp}(\bx^k-\bx^*)\to0$, this gives
	$\bx^k\to\bx^*$.
\smartqedmark \end{proof}

\section{Concluding remarks}

We developed a unified framework for solving inclusion problems involving a backward-accessible set-valued operator, a forward-accessible single-valued operator, and an affine constraint. The framework also offers flexibility in tailoring the resulting splitting scheme to the structure of the problem. In particular, it allows general bounded linear coefficient operators and preconditioned resolvents, going beyond standard scalar-parameter constructions. Under suitable semimonotonicity and cocoercivity assumptions, we established weak convergence of the main sequence and, under additional conditions, weak or strong convergence of the corresponding shadow sequence to a solution of the inclusion problem. Several existing methods also arise as special cases of the proposed framework and are extended to broader settings. In several important cases, the corresponding convergence conditions are weaker and the admissible parameter ranges are larger than those available in closely related works. A key feature of the analysis is its reliance on the orthogonal decomposition $\bcalH =\bOmega\oplus\bOmega^\perp$, which yields a modular proof strategy and avoids algorithm-specific componentwise arguments. It would be interesting to investigate whether our approach can be extended to settings in which the forward operator is merely monotone and Lipschitz continuous.

\paragraph{Acknowledgements.} This work was initiated during MND's 2026 visit to the University of Tokyo, whose hospitality he gratefully acknowledges. MND was partially supported by the Australian Research Council (ARC) under Discovery Project DP230101749. AT was supported by the Grant-in-Aid for Scientific Research (B), JSPS, under Grant No. 23K28041.

\appendix 

\section*{Appendix}
\crefalias{section}{appendix} 
\crefalias{subsection}{appendix}
\renewcommand{\thesection}{\Alph{section}}

\section{Auxiliary results on bounded linear operators}

\subsection{Main tools for the convergence analysis of the general algorithm}

The results presented in this section are the key technical lemmas used in the proofs in \cref{sec:proof_mainresult}. 
First, we calculate a lower bound for a quadratic program, as used in \cref{lemma:fundamentalinequality}. 
\begin{lemma}
	\label{lemma:QP}
	Let $V\in \mathcal{B}(\H)$ be self-adjoint and $\Omega \subseteq \H$ be a closed linear subspace. Let $V=(V_{ij})$ denote the block decomposition of $V$ with respect to $\Omega$ as in \eqref{eq:2x2decomposition}. If  $\ran (V_{22})$ is closed, $\ran (V_{21})\subseteq \ran(V_{22})$, and $V_{22}\succeq 0$, then for any $x\in \H$, $\min_{y+x\in \Omega^{\perp}}  \inner{y}{Vy} = \inner{x}{Qx}$, where $Q  = V/V_{22} = V_{11} - V_{12} V_{22}^{\dagger}V_{21}$. 
\end{lemma}
\begin{proof}
	We  change the variable to $z \coloneqq  y+ \Pi_{\Omega}x$. Note that  $y+x =z +\Pi_{\Omega^\perp}x\in \Omega^{\perp}$ if and only if $z\in \Omega^{\perp}$. Hence, for any $y $ such that $y+x\in \Omega^{\perp}$, we have
	\begin{align*}
		\inner{y}{Vy} & = \inner{z -\Pi_{\Omega}x}{Vz -V\Pi_{\Omega}x} \\
		& = \inner{z}{Vz} - 2\inner{\Pi_{\Omega}x}{Vz} + \inner{\Pi_{\Omega}x}{V\Pi_{\Omega}x} \\
		& = \inner{\Pi_{\Omega^{\perp}}z}{ V \Pi_{\Omega^\perp}z} - 2\inner{\Pi_{\Omega}x}{V\Pi_{\Omega^\perp}z} +  \inner{x}{\Pi_{\Omega}V\Pi_{\Omega}x} \\
		& = g(z), 
	\end{align*}
	where $g(z) \coloneqq \inner{z}{ V_{22} z} - 2\inner{V_{21}x}{z} +  \inner{x}{V_{11}x}$ is a function with domain $\Omega^{\perp}$. Thus, minimizing $\inner{y}{Vy}$ over $\{ y: y+x\in \Omega^{\perp}\}$ is equivalent to the (unconstrained) optimization $\min_{z\in\Omega^\perp} g(z)$. Since $V_{22}\succeq 0$, the minimizer of $g$ is any point $z^*$ such that $V_{22}z^* = V_{21}x$, which has a solution by the range assumption. In particular, $z^* = (V_{22})^{\dagger}V_{21}x$ is an optimal solution. The claim
	follows by calculating $g(z^*)$.
\smartqedmark \end{proof}

The next result provides a sufficient condition when $A\succeq B$ under circumstances encountered in \cref{sec:weak_z}. 

\begin{lemma}
	\label{lemma:beta_existence}
	Let $A,B\in \mathcal{B}(\H)$ be self-adjoint operators such that $\ker(A)\subseteq \ker(B)$ and $A$ is strongly monotone on $\ker(A)^\perp$. 
	\begin{enumerate}
		\item If $\lambda_{\max} ((A^{1/2})^{\dagger} B (A^{1/2})^{\dagger}) \leq 1$, then $A-B \succeq 0$. In particular, if $M\in \mathcal{B}(\K,\H)$ and $\tau >0$ such that $\ker(A) \subseteq \ker(M^*)$ and $\tau \norm{(A^{1/2})^\dagger M}^2 \leq 1$, then $A-\tau MM^*\succeq 0$. 
		\item If $\lambda_{\max} ((A^{1/2})^{\dagger} B (A^{1/2})^{\dagger}) < 1$, then $\ker(A-B) = \ker (A)$ and $A-B$ is strongly monotone on $\ker(A)^\perp = \ker(A-B)^\perp$. 
	\end{enumerate}
\end{lemma}

	\begin{proof}
	Since $A$ is strongly monotone on $\ker(A)^\perp$, it has closed range by \cite[Chapter XI, Proposition 6.1(a)]{Conway} so that $\ran(A) = \ker(A)^\perp$ by \eqref{eq:kernelrange_relationship}. 
	By the closedness of $\ran(A)$, it also follows from  \cite[Theorem 2.1]{Tarcsay2011} that $\ran(A) = \ran (A^{1/2})$. Consequently, $\Pi_{\ker(A)^\perp} = \Pi_{\ran (A^{1/2})} = A^{1/2}(A^{1/2})^\dagger = (A^{1/2})^\dagger A^{1/2}$. Thus, since $\ker(A) \subseteq \ker(B)$, we have $B = \Pi_{\ker(A)^\perp} B \Pi_{\ker(A)^\perp} = A^{1/2}(A^{1/2})^\dagger B (A^{1/2})^\dagger A^{1/2}$. In turn,
	\begin{equation}
		A-B = A^{1/2} ( \Id - (A^{1/2})^\dagger B (A^{1/2})^\dagger) A^{1/2}. 
		\label{eq:A-B_decomposed}
	\end{equation}
	Thus, $ \Id - (A^{1/2})^\dagger B (A^{1/2})^\dagger\succeq 0$ implies that $A-B\succeq 0$. This proves part (i). To prove part (ii), let $x\in \H$ and write $x=x_1+x_2$ with $x_1\in \ker(A)$ and $x_2\in \ker(A)^\perp$. Since $\ker(A)\subseteq \ker(B)$,  
	\begin{align*}
		\inner{x}{(A-B)x} & = \inner{x_2}{(A-B)x_2} \\
		& =\inner{A^{1/2}x_2}{( \Id - (A^{1/2})^\dagger B (A^{1/2})^\dagger) A^{1/2}x_2} \\
		& \geq (1- \lambda_{\max}( (A^{1/2})^\dagger B (A^{1/2})^\dagger))\norm{A^{1/2}x_2}^2 \\
		& \geq (1- \lambda_{\max}( (A^{1/2})^\dagger B (A^{1/2})^\dagger))\alpha \norm{x_2}^2
	\end{align*}
	where the second equality holds by \eqref{eq:A-B_decomposed}, while $\alpha$ is the modulus of strong monotonicity of $A$ on $\ker(A)^\perp$. From this, we get that $A-B$ is strongly monotone on $\ker (A)^\perp$ and $\ker(A-B) \subseteq \ker(A)$. Finally, since $\ker(A)\subseteq \ker(B)$, we have that $\ker(A) \subseteq \ker(A-B)$. 
	\smartqedmark \end{proof}

\begin{lemma}
	\label{lemma:Pi_omegaperp_xk}
	Let $M$ be a bounded linear operator such that $\ran(M)$ is closed. If $\{x^k\}$ is a sequence such that $M^*x^k - b \to 0$, then $\Pi_{\ran (M)}(x^k-\bar{x}) \to 0$ for any $\bar{x}$ such that $M^*\bar{x}=b$.
\end{lemma}
\begin{proof}
	We have $\Pi_{\ran (M)}(x^k-\bar{x}) = MM^{\dagger} (x^k-\bar{x}) = (M^{\dagger})^* M^* (x^k-\bar{x}) = (M^{\dagger})^* (M^* x^k-b)\to 0$, where the first equality follows by noting the closedness of $\ran(M)$ and using \cite[Proposition 3.30(ii)]{BC17}, 
	and the second equality holds since projections are self-adjoint.
\smartqedmark \end{proof}

\subsection{Additional results}

The following proposition provides a characterization of strong monotonicity of $\bU_{11}$ and the conditions $\bV_{22}\succeq 0$ and $\ran(\bV_{21})\subseteq \ran(\bV_{22})$, under the simplified setting that the $U_i$'s and $V_i$'s are scalar operators. 
\begin{proposition}\label{lemma:UV_specialcase}
	Let $\bU=U\otimes \Id$ and $\bV=V\otimes \Id$, where $
	U=\diag(\mu_1,\dots,\mu_n)$ and $V=\diag(\nu_1,\dots,\nu_n).$
	Then the following hold:
	\begin{enumerate}
		\item $\bU_{11}$ is strongly monotone if and only if $
		\sum_{i=1}^n \mu_i>0.$
		
		\item $\ran (\bV_{22})$ is closed. Moreover, the conditions $
		\bV_{22}\succeq 0$ and
		$\ran(\bV_{21})\subseteq \ran(\bV_{22})$
		hold if and only if either
		\begin{enumerate}
			\item $\nu_i\geq 0$ for all $i$, or
			\item there exists a unique index $j\in\{1,\dots,n\}$ such that $\nu_j<0$, $\nu_i>0 $ for all $i\neq j$, and $
			\sum_{i=1}^n \nu_i^{-1}<0.$
		\end{enumerate}
	\end{enumerate}
\end{proposition}
\begin{proof}
	Denote $\Omega=\Span\{1_n\}\subseteq \Re^n$, so that $\Omega^\perp=\{x\in\Re^n:1_n^\top x=0\}
	$. Since $
	\bU=U\otimes \Id$ and $\bV=V\otimes \Id,$
	then
	$
	\bU_{11}=U_{11}\otimes \Id,
	$
	$\bV_{21}=V_{21}\otimes \Id,$
	and 
	$\bV_{22}=V_{22}\otimes \Id.
	$
	For (i), since $U=\diag(\mu_1,\dots,\mu_n)$, we have $
	U_{11}=\frac1n\left(\sum_{i=1}^n \mu_i\right)\Id.$
	Hence, $\bU_{11}$ is strongly monotone if and only if $U_{11}$ is strongly monotone, which is equivalent to $\sum_{i=1}^n \mu_i>0$.
	
	For (ii), $\bV_{22}=V_{22}\otimes \Id$ implies that $\ran(\bV_{22})$ is closed. Further, the identities above give
	\[
	\bV_{22}\succeq 0 \iff V_{22}\succeq 0,
	\qquad
	\ran(\bV_{21})\subseteq \ran(\bV_{22})
	\iff
	\ran(V_{21})\subseteq \ran(V_{22}).
	\]
	Thus, to prove the equivalence claim in (ii), it suffices to show that $
	V_{22}\succeq 0$ and  $\ran(V_{21})\subseteq \ran(V_{22})$
	if and only if either (a) or (b) holds. 
	We divide the proof into three parts:
	
	\medskip 
	\noindent \textbf{Part 1} (Characterization of $V_{22}\succeq 0$). 
	Note that $V_{22}\succeq 0$ if and only if $x^\top V_{22}x\geq 0$ for all $x\in \Omega^\perp$. Meanwhile, for any $x\in \Omega^\perp$, we have 
	\[ \inner{V_{22}x}{x} =
	\inner{\Pi_{\Omega^\perp}Vx}{x}
	=
	\inner{ Vx}{x}
	=
	\sum_{i=1}^n \nu_i x_i^2.
	\]
	Hence, $V_{22}\succeq 0$ if and only if $
	\sum_{i=1}^n \nu_i x_i^2\geq 0$ for all $x\in \Omega^\perp.$ We claim that $
	V_{22}\succeq 0$
	if and only if either (i) all $\nu_i\geq 0$, or (ii) exactly one $\nu_i$ is negative, all the others are positive, and
	$
	\sum_{i=1}^n\nu_i^{-1}\leq 0.$

	\noindent ($\Longrightarrow$) If (i) holds, we are done. Suppose otherwise. We first show that at most one $\nu_i$ can be negative. Indeed, if $\nu_i<0$ and $\nu_j<0$ for some $i\neq j$, choose $
	x= e_i-e_j\in \Omega^\perp$, where $e_i$ denotes the $i$th canonical basis vector of $\Re^n$.
	Then
	$
	\inner{V_{22}x}{x}=\inner{Vx}{x}=\nu_i+\nu_j<0,
	$
	contrary to $V_{22}\succeq 0$. Hence, there is at most one negative $\nu_j$.
	
	Assume now that $\nu_j<0$ for some $j$. Then all $\nu_i\geq 0$ for $i\neq j$. Observe that if $\nu_i=0$ for some $i\neq j$, then with $x= e_j-e_i\in \Omega^\perp$ we get
	$
	\inner{V_{22}x}{x}=\nu_j<0,
	$
	again a contradiction. Thus, (ii) holds.
	
	\noindent ($\Longleftarrow$) If (i) holds, then it is clear that $V_{22}\succeq 0$. Suppose that (ii) holds, that is, $\nu_j<0$ and $\nu_i>0$ for all $i\neq j$. For $x\in\Omega^\perp$, the constraint $\sum_{i=1}^n x_i=0$ gives $
	x_j=-\sum_{i\neq j}x_i.
	$
	Hence
	\begin{equation}\label{eq:xVx_substitute}
		\inner{Vx}{x}
		=\nu_j\Bigl(\sum_{i\neq j}x_i\Bigr)^2+\sum_{i\neq j}\nu_i x_i^2.
	\end{equation}
	By the Cauchy--Schwarz inequality,
	\[
	\Bigl(\sum_{i\neq j}x_i\Bigr)^2 = \Bigl(\sum_{i\neq j}\sqrt{\nu_i}x_i \cdot \frac{1}{\sqrt{\nu_i}}\Bigr)^2
	\leq
	\Bigl(\sum_{i\neq j}\nu_i x_i^2\Bigr)\Bigl(\sum_{i\neq j}\frac1{\nu_i}\Bigr).
	\]
	Since $\nu_j<0$, we obtain 
	$
	\nu_j\Bigl(\sum_{i\neq j}x_i\Bigr)^2
	\ge
	\nu_j\Bigl(\sum_{i\neq j}\frac1{\nu_i}\Bigr)\Bigl(\sum_{i\neq j}\nu_i x_i^2\Bigr).
	$
	Therefore, by \eqref{eq:xVx_substitute}, 
	\begin{equation}\label{eq:xVx_lowerbound}
		\inner{Vx}{x}
		\geq
		\Bigl(1+\nu_j\sum_{i\neq j}\frac1{\nu_i}\Bigr)\sum_{i\neq j}\nu_i x_i^2
		=
		\nu_j\Bigl(\sum_{i=1}^n\frac1{\nu_i}\Bigr)\sum_{i\neq j}\nu_i x_i^2.
	\end{equation}
	Since $\nu_j<0$ and $\sum_{i\neq j}\nu_i x_i^2\geq 0$, we conclude that $\inner{Vx}{x} \geq 0$ for all $x\in \Omega^\perp$ and therefore $V_{22}\succeq 0$.

	\medskip
	\noindent \textbf{Part 2} (Characterization of the range condition). 
	Note that since $V_{22}$ is symmetric and positive semidefinite,
	\begin{equation}
		\label{eq:rangecondition}
		\ran(V_{21})\subseteq \ran(V_{22}) = \ker(V_{22})^\perp 
		\quad\Longleftrightarrow\quad
		\ran(V_{21})\perp \ker(V_{22}).
	\end{equation}
	Meanwhile, $
	\Pi_\Omega=\frac1n1_n 1_n^\top ,$ and $\Pi_{\Omega^\perp}=\Id -\Pi_\Omega.
	$
	Since $
	\Pi_\Omega x=\frac{1_n^\top x}{n}1_n,$
	we obtain
	\[
	V_{21}x
	=\Pi_{\Omega^\perp}V\Pi_\Omega x
	=\frac{1_n^\top x}{n}\,\Pi_{\Omega^\perp}(V1_n).
	\]
	If we set $v\coloneqq (\nu_1,\dots,\nu_n)^\top$, then $V1_n=v$, so $
	\ran(V_{21})=\Span\{\Pi_{\Omega^\perp}v\}.$
	By \eqref{eq:rangecondition}, we have 	
	\begin{equation}
		\label{eq:rangecondition2}
		\ran(V_{21})\subseteq \ran(V_{22})
		\quad\Longleftrightarrow\quad
		\Pi_{\Omega^\perp}v\perp \ker(V_{22}), \quad\text{where } v\coloneqq (\nu_1,\dots,\nu_n)^\top.
	\end{equation}

	\medskip 
	\noindent \textbf{Part 3} (Proof of the main claim).
	
	\noindent ($\Longrightarrow$) Suppose that $V_{22}\succeq 0$ and $\ran(V_{21})\subseteq \ran(V_{22})$. By Part 1, we conclude that either (a)  holds, or exactly one $\nu_i$ is negative, all the others are positive and $\sum_{i=1}^n \nu_i^{-1} \leq 0$. Suppose that  $\sum_{i=1}^n \nu_i^{-1} = 0$, so that
	$
	w:=\bigl(\nu_1^{-1},\dots,\nu_n^{-1}\bigr)^\top \in \Omega^\perp.
	$
	Then $Vw=1_n,$ and therefore
	\[
	V_{22}w=\Pi_{\Omega^\perp}Vw=\Pi_{\Omega^\perp}1_n=0.
	\]
	Thus, $w\in \ker(V_{22})$.
	
	On the other hand,
	\[
	\langle \Pi_{\Omega^\perp}v,w\rangle
	=
	\langle v,w\rangle
	=
	\sum_{i=1}^n \nu_i\frac1{\nu_i}
	=
	n\neq 0,
	\]
	where the first equality holds since $w\in\Omega^\perp$. Hence, $\Pi_{\Omega^\perp}v\not\perp \ker(V_{22})$, and therefore by \eqref{eq:rangecondition2},  $
	\ran(V_{21})\not\subseteq \ran(V_{22})$, which is a contradiction. Thus, $\sum_{i=1}^n\nu_i^{-1}<0$. In other words, (b) holds.

	\noindent ($\Longleftarrow$) We consider each case separately.
	
	\noindent \emph{Case 1:} Suppose that $\nu_i\geq 0$ for all $i$. By Part 1, we have $V_{22}\succeq 0$. Now let $z\in \Omega^\perp$ such that $z\in \ker(V_{22})$.  Then
	$
	0=\inner{V_{22}z}{z}=\inner{Vz}{z}=\sum_{i=1}^n \nu_i z_i^2.
	$
	As each $\nu_i\geq 0$, it follows that $\nu_i z_i=0$ for every $i$, hence $Vz=0$. Therefore
	\[
	\inner{ \Pi_{\Omega^\perp }v}{z}
	= \inner{  v}{\Pi_{\Omega^\perp} z} = 
	\inner{V1_n}{z} = \inner{1_n}{Vz}= 0,
	\]
	because $z\in \Omega^\perp$. Thus, $\Pi_{\Omega^\perp}v\perp \ker(V_{22})$. By \eqref{eq:rangecondition2} in Part 2, we conclude that $\ran (V_{21}) \subseteq \ran (V_{22})$. 
	
	\smallskip
	\noindent
	\emph{Case 2:} Suppose that exactly one $\nu_j$ is negative, all others are positive, and $\sum_i \nu_i^{-1}<0$. Similarly, we have from Part 1 that $V_{22}\succeq 0$.  On the other hand, by \eqref{eq:xVx_lowerbound}, $\inner{Vx}{x}>0$ for all $x\in \Omega^\perp\setminus \{ 0\}$.
	Hence, $V_{22}$ is positive definite on $\Omega^\perp$, and therefore $
	\ran(V_{22})=\Omega^\perp.$
	Since $\ran(V_{21})\subseteq \Omega^\perp$, the range inclusion holds. 
	
	This completes the proof.
\smartqedmark \end{proof}

The following result is the key tool in proving that the admissible range of the stepsize $\gamma$ and the relaxation parameters $\lambda_k$ obtained in the present work are strictly larger than the ones derived in \cite{DTT26}. 
\begin{proposition}\label{prop:compare_gamma_upperbound}
	Let $L\in \mathbb{R}^{n\times n}$,
	$Q\in \mathbb{R}^{n\times p}$ and $M\in \mathbb{R}^{n\times m}$ be matrices, and let $\omega>0$.
	Define
	$
	\rho\coloneqq \omega\,\lambda_{\max}\!\left((L_s^{1/2})^\dagger QQ^\top (L_s^{1/2})^\dagger\right),$, 
	$\tau\coloneqq \norm{Q^\top (M^\top )^\dagger}^2,
	$
	and let $\Omega\coloneqq \ker(M^\top)$. Suppose that $L_s\succeq 0$, $\ker(L_s)=\Omega$, $
	\Omega\subseteq \ker(Q^\top)$, and 
	\begin{equation}
		\label{eq:DTT_assumption_app}
		L_s-\frac12 MM^\top \succeq 0.
	\end{equation}
	Then the following hold:
	\begin{enumerate}
		\item If $\tau,\rho>0$, then $\frac{2}{\omega\tau} \leq \frac{4}{\rho}$, that is, 	$\left(0,\frac{2}{\omega\tau}\right)\subseteq \left(0,\frac{4}{\rho}\right).$ 
		\item  For any $\gamma\in \left(0,\frac{2}{\omega\tau}\right)$, define
		\[
		\Xi_{\gamma}\coloneqq L_s-\frac{\gamma\omega}{4}QQ^\top,
		\qquad
		\beta_{\gamma}\coloneqq \norm{(\Xi_{\gamma}^{1/2})^\dagger M}^{-2},
		\qquad
		\beta_{\gamma}^{\rm DTT}\coloneqq \frac{2-\gamma\omega\tau}{4}.
		\]
		Then
		$
		\beta_{\gamma}\geq \beta_{\gamma}^{\rm DTT}.
		$
	\end{enumerate}
\end{proposition}

\begin{proof}
	We first rewrite $\rho$ as
	\begin{align}
		\rho
		&= \omega\,\lambda_{\max}\!\left((L_s^{1/2})^\dagger QQ^\top(L_s^{1/2})^\dagger\right) \notag\\
		&= \omega\,\lambda_{\max}\!\left(\bigl((L_s^{1/2})^\dagger Q\bigr)\bigl((L_s^{1/2})^\dagger Q\bigr)^\top\right) \notag\\
		&= \omega\,\lambda_{\max}\!\left(\bigl((L_s^{1/2})^\dagger Q\bigr)^\top \bigl((L_s^{1/2})^\dagger Q\bigr)\right) \notag\\
		&= \omega\,\lambda_{\max}(Q^\top L_s^\dagger Q), \label{eq:rho_equivalent}
	\end{align}
	where we used $(L_s^{1/2})^\dagger (L_s^{1/2})^\dagger = L_s^\dagger$. Next, we rewrite $\tau$ as
	\begin{align}
		\tau
		&= \lambda_{\max}\!\left( Q^\top (M^\top )^\dagger \bigl(Q^\top (M^\top )^\dagger\bigr)^\top \right) \notag\\
		&= \lambda_{\max}\!\left( Q^\top (M^\top )^\dagger M^\dagger Q \right) \notag\\
		&= \lambda_{\max}\!\bigl(Q^\top(MM^\top)^\dagger Q\bigr). \label{eq:tau_equivalent}
	\end{align}
	
	Since $\Omega=\ker(M^\top)$, we have $\ran(MM^\top)=\ran(M)=\Omega^\perp$. Hence,$MM^\top|_{\Omega^\perp}$ is an operator on $\Omega^\perp$.
	If $s\in \Omega^\perp\setminus\{0\}$, then $s\notin \ker(M^\top)=\Omega$, and hence $
	\inner{MM^\top s}{s}=\norm{M^\top s}^2>0.$
	Thus, $MM^\top|_{\Omega^\perp}$ is positive definite, and therefore bijective on $\Omega^\perp$. On the other hand, by \eqref{eq:DTT_assumption_app}, $L_s|_{\Omega^\perp}$ is also positive definite. Since $\ker(L_s)=\Omega$ (i.e., $\ran(L_s)=\Omega^\perp)$, $L_s$ is bijective on $\Omega^\perp$.
	
	Hence,
	$
	L_s|_{\Omega^\perp} \succeq \frac12 (MM^\top)|_{\Omega^\perp},
	$
	and consequently
	$
	(L_s|_{\Omega^\perp})^{-1}\preceq 2\bigl((MM^\top)|_{\Omega^\perp}\bigr)^{-1}.
	$
	Therefore, on $\Omega^\perp$,
	$
	L_s^\dagger|_{\Omega^\perp} \preceq 2(MM^\top)^\dagger|_{\Omega^\perp}.
	$
	Equivalently,
	\begin{equation}
		\label{eq:Ldagger_MMtop}
		\inner{L_s^\dagger s}{s}\leq 2\,\inner{(MM^\top)^\dagger s}{s}
		\qquad \forall s\in \Omega^\perp.
	\end{equation}
	
	Now let $x\in \mathbb{R}^p$ with $\norm{x}=1$. Since $\Omega\subseteq \ker(Q^\top)$, we have
	$\ran(Q)\subseteq \Omega^\perp$, and hence $Qx\in \Omega^\perp$. It follows from \eqref{eq:Ldagger_MMtop} that
	\begin{align*}
		\inner{x}{Q^\top L_s^\dagger Qx}
		&= \inner{Qx}{L_s^\dagger Qx}\\
		&\leq2\,\inner{(MM^\top)^\dagger Qx}{Qx}\\
		&= 2\,\inner{x}{Q^\top (MM^\top)^\dagger Qx}\\
		&\leq2\,\lambda_{\max}\!\bigl(Q^\top (MM^\top)^\dagger Q\bigr)\\
		&= 2\tau,
	\end{align*}
	where the last equality follows from \eqref{eq:tau_equivalent}. Taking the supremum over all $x\in\mathbb{R}^p$ with $\norm{x}=1$ and using \eqref{eq:rho_equivalent}, we obtain
	$
	\frac{\rho}{\omega}\leq 2\tau.
	$
	Equivalently,
	$
	\frac{4}{\rho}\geq\frac{2}{\omega\tau},
	$
	which proves the claim of (i).
	
	We now compare $\beta_\gamma$ and $\beta_\gamma^{\rm DTT}$. Since  $\ran(Q)\subseteq \Omega^\perp=\ran(M)$, and therefore
	$Q=\Pi_{\ran (M)} Q = MM^\dagger Q.$
	Hence, for any $y\in \Re^n$
	\begin{align*}
		\inner{y}{QQ^\top y} & = \inner{y}{MM^\dagger Q Q^\top (M^\dagger)^\top M^\top y} \\
		& = \inner{M^\top y}{M^\dagger Q Q^\top (M^\dagger)^\top (M^\top y)} \\ 
		& \leq \norm{Q^\top (M^\dagger)^\top}^2 \inner{M^\top y}{M^\top y} \\
		& = \tau \inner{y}{MM^\top y}.
	\end{align*}
	That is, $QQ^\top \preceq \tau MM^\top$. Consequently,
	\begin{align}
		\Xi_\gamma
		&= L_s-\frac{\gamma\omega}{4}QQ^\top\succeq L_s-\frac{\gamma\omega\tau}{4}MM^\top \succeq \left(\frac12-\frac{\gamma\omega\tau}{4}\right)MM^\top = \beta_\gamma^{\rm DTT}\,MM^\top. \label{eq:Xi_lower_bound_polished}
	\end{align}
	Since $\gamma<\frac{2}{\omega\tau}$, we have $\beta_\gamma^{\rm DTT}>0$. Thus, $\Xi_\gamma|_{\Omega^\perp}$ is positive definite, whose range is contained in $\Omega^\perp$ since $\Omega\subseteq \ker (\Xi_{\gamma})$. Similar to preceding arguments, \eqref{eq:Xi_lower_bound_polished} yields
	$
	\Xi_\gamma^\dagger|_{\Omega^\perp}\preceq \frac{1}{\beta_\gamma^{\rm DTT}}(MM^\top)^\dagger|_{\Omega^\perp}.
	$ Therefore,
	\begin{align*}
		\norm{(\Xi_\gamma^{1/2})^\dagger M}^2
		&= \lambda_{\max}(M^\top \Xi_\gamma^\dagger M) \\
		& = \max_{\norm{y}=1} \! \inner{\Xi_{\gamma}^\dagger M y}{My}, \qquad (\text{where } My \in \Omega^\perp) \\
		& \leq  \max_{\norm{y}=1} \! \frac{1}{\beta_\gamma^{\rm DTT}} \inner{(MM^\top)^\dagger My}{My} \\ 
		& = \frac{1}{\beta_\gamma^{\rm DTT}}\lambda_{\max}\!\bigl(M^\top (MM^\top)^\dagger M\bigr).
	\end{align*}
	Now $M^\top (MM^\top)^\dagger M=M^\dagger M$ is the orthogonal projector onto $\ran(M^\top)$, and therefore its largest eigenvalue is at most $1$. Hence
	$
	\norm{(\Xi_\gamma^{1/2})^\dagger M}^2\leq \frac{1}{\beta_\gamma^{\rm DTT}}.
	$
	Taking reciprocals, we obtain
	$
	\beta_\gamma
	=
	\norm{(\Xi_\gamma^{1/2})^\dagger M}^{-2}
	\geq  \beta_\gamma^{\rm DTT},
	$
	as claimed.
\smartqedmark \end{proof}

\begin{example}\label{ex:strictlylarger_gamma_beta}
	Let $L = \begin{bmatrix}
		1 & 0\\
		-2 & 1
	\end{bmatrix}$ so that 
	$L_s=
	\begin{bmatrix}
		1 & -1\\
		-1 & 1
	\end{bmatrix}$. Further, let $M=
	\begin{bmatrix}
		1\\
		-1
	\end{bmatrix}$, $P=\begin{bmatrix}
	0\\
	1
	\end{bmatrix}$, $R=
	\begin{bmatrix}
	1 & 0
	\end{bmatrix},
	$
and let  $ Q\coloneqq P-R^\top.$ 
	Then
	$
	MM^\top=L_s,
	$
	and hence
	$
	L_s-\frac12 MM^\top=\frac12 L_s\succeq 0.
	$ Moreover,
	$
	\ker(L_s)=
	\ker(M^\top) = \Omega,
	$ where $\Omega \coloneqq \{ (x,x):x\in \Re\}$. 
	Observe that
	$
	Q = \begin{bmatrix}
		-1\\1
	\end{bmatrix}\in\Omega^\perp
	$, and  therefore $\ran(Q)\subseteq \Omega^\perp$.

By direct calculation, $Q^\top (M^\top)^\dagger = -1$ and therefore $\tau=\norm{Q^\top (M^\top)^\dagger}^2=1.$ On the other hand, $L_s^\dagger=\frac14 \begin{bmatrix}
		1 & -1\\
		-1 & 1
	\end{bmatrix}.$
	A direct computation gives
	$
	Q^\top L_s^\dagger Q = 1.
	$
	Hence,
	$
	\lambda_{\max}(Q^\top L_s^\dagger Q)=1,
	$
	and therefore
	$
	\rho=\omega\,\lambda_{\max}(Q^\top L^\dagger Q)=\omega.
	$ Consequently, $\frac{4}{\rho} > \frac{2}{\omega\tau}$.

Now, observe that
$
QQ^\top=
\begin{bmatrix}
	1 & -1\\
	-1 & 1
\end{bmatrix}
=L_s.
$
Thus,
$
\Xi_\gamma
=
L_s-\frac{\gamma\omega}{4}QQ^\top
=
\left(1-\frac{\gamma\omega}{4}\right)L_s.
$
Therefore,
$
(\Xi_\gamma^{1/2})^\dagger
=
\left(1-\frac{\gamma\omega}{4}\right)^{-1/2}(L_s^{1/2})^\dagger.
$
Direct calculations give $
\beta_\gamma
= \norm{(\Xi_\gamma^{1/2})^\dagger M}^{-2} = 1-\frac{\gamma\omega}{4}$. On the other hand, since $\tau=1$, we also have $\beta_\gamma^{\rm DTT} = \frac{2-\gamma\omega}{4}. $
Therefore,
$\beta_\gamma-\beta_\gamma^{\rm DTT} =\left(1-\frac{\gamma\omega}{4}\right)-\frac{2-\gamma\omega}{4}
=\frac12>0.$
\end{example}

\section{Auxiliary results on semimonotone operators}

The following is a generalization of \cite[Corollary 4.10]{ELP25}.

\begin{proposition}
	\label{prop:equivalence_semimonotone}
	Let $U,V\in \mathcal{B}(\H)$ be self-adjoint operators such that $UV=\theta \Id$ with $\theta<1/4$. Let $\varrho = -\frac{1}{2}+\frac{1}{2\sqrt{1-4\theta}}$, and define 
	\begin{equation}
		\label{eq:hatUV}
		\hat{U} \coloneqq \frac{1+2\varrho}{1+\varrho}U \quad \text{and} \quad \hat{V} \coloneqq (1+2\varrho)V.
	\end{equation}
	Then $A$ is (maximally) $(U,V)$-semimonotone if and only if $\hat{A} \coloneqq (A-\hat{U})^{-1} - \hat{V}$ is (maximally) monotone.
\end{proposition}
\begin{proof}
	Note that $\varrho$ is a root of the quadratic equation $\varrho(1+\varrho) = \theta(1+2\varrho)^2$ such that $1+2\varrho>0$ and $1+\varrho>0$. Hence, $\hat{U}\hat{V} = \varrho \Id = \hat{V}\hat{U}$. We note further that 
	\begin{align*}
		\begin{array}{rl}
			y\in Ax \quad & \Longleftrightarrow \quad x\in (A-\hat{U})^{-1} (y-\hat{U}x) \\
			& \Longleftrightarrow \quad x-\hat{V}(y-\hat{U}x) \in \hat{A}(y-\hat{U}x) \\
			& \Longleftrightarrow \quad  (y-\hat{U}x , (1+\varrho)x-\hat{V}y ) \in \gra (\hat{A}),
		\end{array}
	\end{align*}
	where we used $\hat{V}\hat{U}=\varrho \Id$ in the last line. In other words, $(x,y)\in \gra(A)$ if and only if $\Phi(x,y) \in \gra(\hat{A})$, where $\Phi(x,y) \coloneqq (s,t) \coloneqq (y-\hat{U}x, (1+\varrho)x-\hat{V}y)$. Note that $\Phi:\H\times \H \to \H \times \H$ is a bijection with inverse $\Phi^{-1}(s,t) = (t+\hat{V}s, \hat{U}t+(1+\varrho)s)$. 
	
	Meanwhile, for any $(x,y),(x',y')\in \H\times \H$, and $(s,t),(s',t')\in \H\times \H$ with $(s,t) = \Phi(x,y)$ and $(s',t')=\Phi(x',y')$, we have 
	\begin{align}
		\begin{array}{rl}
			\inner{\Delta s}{\Delta t} \geq 0 & \Longleftrightarrow  \quad \inner{\Delta y - \hat{U}\Delta x}{(1+\varrho)\Delta x - \hat{V}\Delta y} \geq 0 \\
			& \Longleftrightarrow \quad (1+2\varrho)\inner{\Delta y }{\Delta x}- (1+\varrho)\inner{\hat{U}\Delta x}{\Delta x} - \inner{\Delta y}{\hat{V}\Delta y} \geq 0 \\
			& \Longleftrightarrow \quad \inner{\Delta y}{\Delta x} \geq \inner{U\Delta x}{\Delta x} + \inner{\Delta y}{V\Delta y},
		\end{array}
		\label{eq:equivalence_inequalities}
	\end{align}
	where $\Delta x= x-x'$, $\Delta y = y-y'$, $\Delta s = s-s'$ and $\Delta t = t-t'$. Thus, $\Phi$ is a bijection such that $(x,y)\in \gra(A)$ if and only if $\Phi(x,y)\in \gra(\hat{A})$, and \eqref{eq:equivalence_inequalities} shows that the $(U,V)$-semimonotonicity inequality for pairs in $\gra(A)$ holds if and only if the monotonicity inequality for the corresponding pairs in $\gra(\hat{A})$ holds. Therefore, $A$ is $(U,V)$-semimonotone if and only if $\hat{A}$ is monotone, and the same equivalence holds for maximality since $\Phi$ preserves strict graph inclusions.
\smartqedmark \end{proof}

The following extends the known relationship between  $\alpha \Id$-monotone and $\alpha \Id$-comonotone operators; cf. \cite[Lemmas 2.5 and 2.7]{Bauschke2021}.
\begin{corollary}
	\label{lemma:V-comonotone}
	The following are equivalent:
	\begin{enumerate}[(a)]
		\item $A$ is maximally $V$-comonotone.
		\item $A^{-1}-V$ is maximally monotone.
		\item $A^{-1}$ is maximally $V$-monotone.
	\end{enumerate}
\end{corollary}
\begin{proof}
	The equivalence of (a) and (c) is immediate from the definitions. The equivalence of (a) and (b) follows by setting $U=0$ in \cref{prop:equivalence_semimonotone}.
\smartqedmark \end{proof}

Finally, we prove that the resolvent $J_{\gamma D^{-1}A} = (\Id + \gamma D^{-1}A)^{-1}$ is single-valued and has full domain for an appropriately chosen $\gamma$.
\begin{proposition}\label{prop:resolvent_singlevalued_fulldomain}
	Let $A:\H\toset \H$, and let $\hat{U},\hat{V},D\in \mathcal{B}(\H)$ be self-adjoint operators where $D$ is strongly monotone. Suppose that $\hat{A} \coloneqq (A-\hat{U})^{-1} - \hat{V}$ is maximally monotone and $\gamma>0$ is chosen such that $R_{\gamma}\coloneqq \gamma \hat{U}+D$ and  $S_{\gamma} \coloneqq \gamma R_{\gamma}^{-1} + \hat{V}$ are  strongly monotone. Then
	\begin{enumerate}
		\item  $J_{\gamma D^{-1} A}$ is single-valued and has full domain. 
		\item $J_{\gamma D^{-1}A}$ is $\ell$-Lipschitz continuous, with $\ell \coloneqq \norm{R_{\gamma}^{-1}} \left(\norm{D} + \gamma \norm{S_{\gamma}^{-1/2}} \, \norm{S_{\gamma}^{-1/2} R_{\gamma}^{-1}D} \right)$.
	\end{enumerate}
\end{proposition}
\begin{proof}
	Let $w\in \H$. Note that 
	\begin{align*}
		y=J_{\gamma D^{-1}A}w \quad & \Longleftrightarrow \quad Dw\in \gamma (A-\hat{U})y + (\gamma \hat{U}+D)y =  \gamma (A-\hat{U})y + R_{\gamma} y\\ 
		& \Longleftrightarrow \quad  y \in (A-\hat{U})^{-1} z \quad \text{where} \quad z =\frac{1}{\gamma}(Dw - R_{\gamma}y) \\
		& \Longleftrightarrow \quad y- \hat{V}z \in \hat{A}z\quad \text{where} \quad y= R_{\gamma}^{-1}(Dw-\gamma z) \\
		& \Longleftrightarrow \quad R_{\gamma}^{-1}(Dw-\gamma z) -\hat{V}z \in \hat{A}z \\ 
		& \Longleftrightarrow \quad  R_{\gamma}^{-1} Dw \in \hat{A}z + (\gamma R_{\gamma}^{-1} + \hat{V})z \\
		& \Longleftrightarrow \quad S_{\gamma}^{-1}R_{\gamma}^{-1} Dw \in (S_{\gamma}^{-1}\hat{A} + \Id ) z.
	\end{align*}
	Since $\hat{A}$ is maximally monotone and $S_{\gamma}$ is strongly monotone, $S_{\gamma}^{-1} \hat{A}$ is maximally monotone when $\H$ is equipped with the inner product $\inner{\cdot}{\cdot}_{S_{\gamma}}$ by \cref{lemma:maximalmonotone_properties}(v). In turn, $\ran (S_{\gamma}^{-1} \hat{A}+\Id) = \H$ by \cref{lemma:maximalmonotone_properties}(i), and therefore there exists a unique $z\in \H$ such that $S_{\gamma}^{-1}R_{\gamma}^{-1} Dw \in (S_{\gamma}^{-1}\hat{A} + \Id ) z$. By the above equivalences, if we set $y=R_{\gamma}^{-1}(Dw-\gamma z)$, then $y = J_{\gamma D^{-1}A}w$. This proves that $J_{\gamma D^{-1}A}$ is single-valued with full domain. 
	
	To prove (ii), note that by \cref{lemma:maximalmonotone_properties}(i), $(S_{\gamma}^{-1}\hat{A} + \Id )^{-1} $ is $1$-Lipschitz in the $S_{\gamma}$-norm. If $z_i = (S_{\gamma}^{-1}\hat{A} + \Id ) ^{-1}(S_{\gamma}^{-1} R_{\gamma}^{-1}Dw_i)$ for $i=1,2$, then since  $\norm{S_{\gamma}^{-1}}^{-1} \norm{\Delta z}^2 \leq \norm{\Delta z}_{S_{\gamma}}^2$ (see \cref{sec:boundedlinear_definitions}), we have  
	\[ \norm{\Delta z}\leq  \norm{S_{\gamma}^{-1}}^{1/2}\norm{ \Delta z}_{S_{\gamma}} \leq  \norm{S_{\gamma}^{-1/2}}\norm{ S_{\gamma}^{-1} R_{\gamma}^{-1}D\Delta w}_{S_{\gamma}} =  \norm{S_{\gamma}^{-1/2}} \, \norm{S_{\gamma}^{-1/2} R_{\gamma}^{-1}D\Delta w} ,\]
	where $\Delta z \coloneqq z_1-z_2$ and $\Delta w \coloneqq w_1-w_2$. Then the claim immediately follows, noting that $y_i \coloneqq R_{\gamma}^{-1} (Dw_i - \gamma z_i) = J_{\gamma D^{-1}A}w_i$.
\smartqedmark \end{proof}

\begin{remark}[On strong monotonicity conditions]\label{remark:strongmonotonicityconditions}
	Suppose that $D$ is $\delta$-strongly monotone. Then
	$R_{\gamma} \succeq (\gamma \lambda_{\min}(\hat{U}) + \delta)\Id$,
	and hence $R_{\gamma}$ is strongly monotone whenever
	$\gamma \lambda_{\min}(\hat{U})+\delta > 0$.
	Moreover, since
	$R_{\gamma} \preceq (\gamma \norm{\hat{U}} + \norm{D})\Id$,
	we obtain
	$S_{\gamma} \succeq \left(\frac{\gamma}{\gamma \norm{\hat{U}}+\norm{D}}+\lambda_{\min}(\hat{V})\right)\Id$.
	Thus, $S_{\gamma}$ is strongly monotone whenever
	$\frac{\gamma}{\gamma \norm{\hat{U}}+\norm{D}}+\lambda_{\min}(\hat{V})>0$.
	
	Suppose that the following holds:
	\begin{equation}\label{eq:C(U,V)_appendix}
		\begin{cases}
			1+\norm{\hat U}\lambda_{\min}(\hat V)>0 & \text{if }\hat U\succeq0 \text{ and } \hat V\not \succeq0,	\\
			-\lambda_{\min}(\hat V)\bigl(-\norm D\,\lambda_{\min}(\hat U)+\delta\norm{\hat U}\bigr)<\delta & \text{if } \hat U\not\succeq0 \text{ and } \hat V\not\succeq0.
		\end{cases}
	\end{equation}
	Then, for any
	$
		\gamma \in \left(
		\frac{\norm D\,(-\lambda_{\min}(\hat V))_+}
		{1+\norm{\hat U}\lambda_{\min}(\hat V)},
		\frac{\delta}{(-\lambda_{\min}(\hat U))_+}
		\right)$
	(this interval is a nonempty subset of $(0,+\infty)$ by \eqref{eq:C(U,V)_appendix}), both $R_{\gamma}$ and $S_{\gamma}$ are strongly monotone.  
\end{remark}

\begin{remark}\label{remark:scalar_UVD}
	Suppose that $D=\delta \Id$ with $\delta>0$, and that
	$
	\hat U=\hat\mu \Id
	$
	and
	$
	\hat V=\hat\nu \Id
	$
	for some $\hat\mu,\hat\nu\in\mathbb R$. Then
	$
	R_\gamma=(\delta+\gamma\hat\mu)\Id,$ and
	$S_\gamma=\left(\hat\nu+\frac{\gamma}{\delta+\gamma\hat\mu}\right)\Id.
	$
	Hence, $R_\gamma$ and $S_\gamma$ are strongly monotone if and only if $
	\delta+\gamma\hat\mu>0$ and $
	\delta\hat\nu+\gamma(1+\hat\mu\hat\nu)>0.$
	In particular, if $1+\hat\mu\hat\nu>0$, then this is equivalent to
	\begin{equation}\label{eq:chi_exact}
		\gamma\in \left(
		\frac{\delta(-\hat\nu)_+}{1+\hat\mu\hat\nu},
		\frac{\delta}{(-\hat\mu)_+}
		\right).
	\end{equation}
	The above interval is sharper than the sufficient range given in \cref{remark:strongmonotonicityconditions}, and in fact is exact in this scalar setting. 
\end{remark}

\begin{remark}\label{remark:scalarsemimonotone}
	If $A$ is maximally $(\mu \Id,\nu \Id)$-semimonotone with $\mu\nu<1/4$, then $\hat{A} = (A-\hat{U})^{-1} - \hat{V}$ is maximally monotone by \cref{prop:equivalence_semimonotone}, where $(\hat{U},\hat{V})=(\hat{\mu}\Id,\hat{\nu}\Id)$ is given in \eqref{eq:hatUV}. That is, $\hat\mu = \frac{1+2\varrho}{1+\varrho}\mu$ and $\hat{\nu}=(1+2\varrho)\nu$. In this case, $1+\hat{\mu}\hat{\nu}=1+\varrho>0$. Thus, from \eqref{eq:chi_exact}, the actual interval where $R_{\gamma}$ and $S_{\gamma}$ are strongly monotone is
	\begin{equation*}
		\gamma \in 
		\left(
		\frac{2\delta\max\{-\nu,0\} }{1+\sqrt{1-4\mu\nu}}
		\;,\;
		\frac{\delta(1+\sqrt{1-4\mu\nu})}{2\max\{ -\mu,0\}}
		\right).
	\end{equation*}
	When $\delta=1$, this recovers exactly \cite[Proposition 4.12]{EPLP25}, which reduces to the interval $(0,1/\max\{-\mu,0\})$ when $\nu=0$ (cf. \cite[Proposition 3.4]{DP18}) and to $(\max\{ -\nu,0\},+\infty)$ when $\mu=0$ (cf. \cite[Theorem 2.16]{Bauschke2021} and \cite[Proposition 3.7]{BDP22}).  
\end{remark}

\section{Omitted proofs}

\subsection{Proof of \cref{prop:assume_blanket_gensubspace_holds}}\label{app:AssumptionA_holds}
	
\begin{proof}
		Part (ii) follows from \cref{assume:blanket_multioperator}(ii) and the definition of $\bB$, and part (iii) follows from \eqref{eq:PR}. 
		We prove (i).  
		Note that $(\bU,\bV)$-semimonotonicity follows immediately from the definition of $\bA$ and the $(U_i,V_i)$-semimonotonicity of each $A_i$ in \cref{assume:blanket_multioperator}(i). We now prove maximality. Let $(\bx,\by)\in \bcalH\times \bcalH$ and suppose that for all $(\bx',\by')\in \gra(\bA)$,
		\begin{equation}\label{eq:relation}
			\phi(\bx,\by; \bx',\by')\coloneqq 
			\inner{\bx-\bx'}{\by-\by'} 
			- \inner{\bx-\bx'}{\bU(\bx-\bx')} 
			- \inner{\by-\by'}{\bV(\by-\by')} \geq 0 .
		\end{equation}
		Observe that $\phi(\bx,\by;\bx',\by')=\sum_{i=1}^n \phi_i(x_i,y_i;x_i',y_i')$, where
		\[
		\phi_i(x_i,y_i;x_i',y_i')\coloneqq 
		\inner{x_i-x_i'}{y_i-y_i'}
		- \inner{x_i-x_i'}{U_i(x_i-x_i')}
		- \inner{y_i-y_i'}{V_i(y_i-y_i')}.
		\]
		Define
		\[
		J \coloneqq \Bigl\{ j\in \{1,\dots,n\} : \exists (x_j',y_j')\in \gra(A_j)\ \text{s.t.}\ \phi_j(x_j,y_j;x_j',y_j')< 0\Bigr\}.
		\]
		It suffices to show that $J=\varnothing$. Indeed, in this case, $\phi_i(x_i,y_i;x_i',y_i')\geq 0$ for all $(x_i',y_i')\in \gra(A_i)$ and all $i$, and hence by maximal $(U_i,V_i)$-semimonotonicity of $A_i$, we obtain $(x_i,y_i)\in \gra(A_i)$ for every $i$. Therefore $(\bx,\by)\in \gra(\bA)$, as desired.
		
		For the sake of contradiction, suppose that $J\neq \varnothing$. For any $i\notin J$, we have $\phi_i(x_i,y_i;x_i',y_i')\geq 0$ for all $(x_i',y_i')\in \gra(A_i)$. By maximal $(U_i,V_i)$-semimonotonicity of $A_i$, it follows that $(x_i,y_i)\in \gra(A_i)$ for every $i\notin J$. In particular, $\phi_i(x_i,y_i;x_i,y_i)=0$ for all $i\notin J$.
		
		For each $j\in J$, fix a point $(\hat x_j,\hat y_j)\in \gra(A_j)$ such that $
		\phi_j(x_j,y_j;\hat x_j,\hat y_j)<0.$
		Define $(\bx^\star,\by^\star)\in \bcalH\times \bcalH$ by
		\[
		(x_i^\star,y_i^\star)\coloneqq
		\begin{cases}
			(x_i,y_i), & i\notin J,\\
			(\hat x_i,\hat y_i), & i\in J.
		\end{cases}
		\]
		Then $(\bx^\star,\by^\star)\in \gra(\bA)$ since $(x_i,y_i)\in \gra(A_i)$ for all $i\notin J$ and $(\hat x_i,\hat y_i)\in \gra(A_i)$ for all $i\in J$. Evaluating \eqref{eq:relation} at $(\bx',\by')=(\bx^\star,\by^\star)$ and using the decomposition of $\phi$, we obtain
		\[
		0 \leq\phi(\bx,\by;\bx^\star,\by^\star)
		= \sum_{i\notin J}\phi_i(x_i,y_i;x_i,y_i)
		+ \sum_{j\in J}\phi_j(x_j,y_j;\hat x_j,\hat y_j)
		= \sum_{j\in J}\phi_j(x_j,y_j;\hat x_j,\hat y_j) < 0,
		\]
		which is a contradiction. Hence,  $J=\varnothing$, and therefore $\bA$ is maximally $(\bU,\bV)$-semimonotone.
\smartqedmark \end{proof}

\subsection{Proof of \cref{prop:Lambda_nonempty}}\label{app:prop_Lambda_nonempty}

\begin{proof}
	By \cref{prop:equivalence_semimonotone}, $(A_i-\hat{U}_i)^{-1} - \hat{V}_i$ is maximally monotone. Hence, by \cref{prop:resolvent_singlevalued_fulldomain}, $J_{\gamma \bD^{-1} \bA}$ is single-valued, has full domain and is Lipschitz continuous if $(R_i)_{\gamma}\coloneqq \gamma \hat{U}_i + D_i$ and $(S_i)_{\gamma} \coloneqq \gamma (R_i)_{\gamma}^{-1} + \hat{V}_i$ are strongly monotone for all $i$. By \cref{remark:strongmonotonicityconditions}, these strong monotonicity conditions hold when $\gamma \in (\sigma,\varsigma)$. This proves the first claim. Noting the equivalence \eqref{eq:Phigamma_J_Dinverse_A}, it holds that $(\sigma, \varsigma)\subseteq \Lambda$, thus completing the proof. 
\smartqedmark \end{proof}

\subsection{Proof of \cref{prop:Lambda_nonempty_specialcase}}\label{app:prop_Lambda_nonempty_special}

\begin{proof}
The proof proceeds exactly as in \cref{app:prop_Lambda_nonempty}, except that one uses \cref{remark:scalar_UVD,remark:scalarsemimonotone} instead of \cref{remark:strongmonotonicityconditions}.
\smartqedmark \end{proof}

\section{Miscellaneous results}

\begin{proposition}\label{prop:DR_multioperator}
	The multioperator algorithm in \cite[Algorithm 1]{AT25} with equal weights $\lambda_i=\frac{1}{n-1}$ for $i=1,\dots,n-1$ coincides with \cref{algo:full} with the following choices: $\bL = L\otimes \Id$ and $\bM = M\otimes \Id$  with 
	\begin{equation*}
		M=
		\frac{1}{n-1}
		\begin{bmatrix}
			1 & 0 & \cdots & 0\\
			0 & 1 & \cdots & 0\\
			\vdots & \vdots & \ddots & \vdots\\
			0 & 0 & \cdots & 1\\
			-1 & -1 & \cdots & -1
		\end{bmatrix}
		\in \mathbb{R}^{\,n\times (n-1)} \quad \text{and} \quad L=
		\begin{bmatrix}
			\frac{1}{n-1} & 0 & \cdots & 0 & 0\\
			0 & \frac{1}{n-1} & \cdots & 0 & 0\\
			\vdots & \vdots & \ddots & \vdots & \vdots\\
			0 & 0 & \cdots & \frac{1}{n-1} & 0\\
			-\frac{2}{n-1} & -\frac{2}{n-1} & \cdots & -\frac{2}{n-1} & 1
		\end{bmatrix}
		\in \mathbb{R}^{\,n\times n}.
	\end{equation*}
\end{proposition}
\begin{proof}
	The proof is straightforward.
\smartqedmark \end{proof}

\begin{proposition}\label{prop:DR_twooperator}
	Let $L=
	\begin{bmatrix}
		1 & 0\\
		-2 & 1
	\end{bmatrix}$, $M =
	\begin{bmatrix}
		1 \\
		-1 
	\end{bmatrix}$, $U=\diag(\mu_1,\mu_2)$ and $V=\diag(\nu_1,\nu_2)$. Set\[
	a\coloneqq
	\begin{cases}
		0, & \mu_1=\mu_2=0,\\[1mm]
		\dfrac{\mu_1\mu_2}{\mu_1+\mu_2}, & \mu_1+\mu_2>0,
	\end{cases}
	\qquad\text{and} \qquad 
	b\coloneqq \begin{cases}
		0, & \nu_1=\nu_2= 0\\[1mm] 
		\dfrac{\nu_1\nu_2}{\nu_1+\nu_2}, & \nu_1+\nu_2>0.
	\end{cases}
	\]  Then 
	\begin{enumerate}
		\item $\Gamma{}$ given in \cref{tab:Gammastar} can be expressed as
		\[\Gamma{}
		=
		\left\{\gamma>0:\gamma+\min(0,a\gamma^2+b)>0\right\},\]
		and $\Gamma{} \cap  (\sigma,\varsigma)=\Gamma{}$, where $(\sigma,\varsigma)$ is given in \cref{tab:Lambdastar_scalarUV}. 
		\item $\beta_{\gamma}$ given in \cref{tab:notation} can be expressed as $\beta_{\gamma} = 1+\frac{1}{\gamma}b + \gamma a$. 
	\end{enumerate}
	
\end{proposition}
\begin{proof}
	We want to calculate the set 
	\[ \Gamma{} = \left\lbrace  \gamma>0 : \gamma + \lambda_{\min}\!\left((L_s^{1/2})^\dagger\bigl(\gamma^2 X+L^\top QL\bigr)(L_s^{1/2})^\dagger\right)>0\right\rbrace \]
	Note that $L_s =
	\begin{bmatrix}
		1 & -1\\
		-1 & 1
	\end{bmatrix}$, $L_s^{1/2}
	=
	\frac{1}{\sqrt{2}}
	L_s$ and $(L_s^{1/2})^\dagger
	=
	\frac{1}{2\sqrt{2}}
	L_s.$ Moreover, $\bX$ and $\bQ$ given in \cref{tab:notation} can be expressed as $\bX = X\otimes \Id$ and $\bQ = Q\otimes \Id$ where  $X=aL_s$ and $Q = b \begin{bmatrix}
		1 & 1\\
		1 & 1
	\end{bmatrix}$, respectively. Then, it can be shown that $L^\top QL = bL_s$ and 
	so $
	\gamma^2 X+L^\top QL =(a\gamma^2+b)
	L_s$.
	Thus,
	\[
	(L_s^{1/2})^\dagger\bigl(\gamma^2 X+L^\top QL\bigr)(L_s^{1/2})^\dagger
	=
	\frac{a\gamma^2+b}{2}
	L_s.
	\]
	The eigenvalues of this matrix are $0$ and $
	a\gamma^2+b$, and so  
	\[
	\lambda_{\min}\!\left((L_s^{1/2})^\dagger\bigl(\gamma^2 X+L^\top QL\bigr)(L_s^{1/2})^\dagger\right)
	=
	\min\{0,a\gamma^2+b\}.
	\]
	Finally, it can be directly verified that $\Gamma{} \subseteq (\sigma,\varsigma)$. This proves (i).
	
	To prove (ii), the above calculations give
	\[
	\hat{\Xi}_\gamma
	=
	L_s+\frac1\gamma L^\top QL+\gamma X
	=
	\left(1+\frac{b}{\gamma}+a\gamma\right)L_s.
	\]
	Hence, whenever $1+\frac{b}{\gamma}+a\gamma>0$,
	$
	(\hat{\Xi}_\gamma^{1/2})^\dagger
	=
	\left(1+\frac{b}{\gamma}+a\gamma\right)^{-1/2}(L_s^{1/2})^\dagger.
	$
	Since $
	\norm{(L_s^{1/2})^\dagger M}^2=1$,
	it follows that
	 $
	\norm{(\hat{\Xi}_\gamma^{1/2})^\dagger M}^2
	=
	\left(1+\frac{b}{\gamma}+a\gamma\right)^{-1}.
	$
	Hence, the claim of (ii) follows. 
\smartqedmark \end{proof}

\setlength{\bibsep}{1pt}


\begin{thebibliography}{99}
\providecommand{\natexlab}[1]{#1}
\providecommand{\url}[1]{\texttt{#1}}
\expandafter\ifx\csname urlstyle\endcsname\relax
  \providecommand{\doi}[1]{doi: #1}\else
  \providecommand{\doi}{doi: \begingroup \urlstyle{rm}\Url}\fi

\bibitem[{\AA}kerman et~al.(2025){\AA}kerman, Chenchene, Giselsson, and
  Naldi]{ACGN25}
A.~{\AA}kerman, E.~Chenchene, P.~Giselsson, and E.~Naldi.
\newblock Splitting the forward-backward algorithm: A full characterization,
  2025.
\newblock \href{https://arxiv.org/abs/2504.10999}{arXiv:2504.10999}.

\bibitem[Alcantara and Takeda(2025)]{AT25}
J.~H. Alcantara and A.~Takeda.
\newblock Douglas--{R}achford algorithm for nonmonotone multioperator inclusion
  problems, 2025.
\newblock \href{https://arxiv.org/abs/2501.02752}{arXiv:2501.02752}.

\bibitem[Alcantara et~al.(2025)Alcantara, Dao, and Takeda]{ADT25}
J.~H. Alcantara, M.~N. Dao, and A.~Takeda.
\newblock Douglas--{R}achford for multioperator comonotone inclusions with
  applications to multiblock optimization, 2025.
\newblock \href{https://arxiv.org/abs/2506.22928}{arXiv:2506.22928}.

\bibitem[Arag{\'o}n-Artacho et~al.(2023)Arag{\'o}n-Artacho, Bo{\c{t}}, and
  Torregrosa-Bel{\'e}n]{aragonartacho_bot_torregrosabelen_2023}
F.~J. Arag{\'o}n-Artacho, R.~I. Bo{\c{t}}, and D.~Torregrosa-Bel{\'e}n.
\newblock A primal-dual splitting algorithm for composite monotone inclusions
  with minimal lifting.
\newblock \emph{Numer. Algorithms}, 93:\penalty0 103--130, 2023.

\bibitem[Arag\'{o}n-Artacho et~al.(2023)Arag\'{o}n-Artacho, Malitsky, Tam, and
  Torregrosa-Bel\'{e}n]{AMTT23}
F.~J. Arag\'{o}n-Artacho, Y.~Malitsky, M.~K. Tam, and D.~Torregrosa-Bel\'{e}n.
\newblock Distributed forward-backward methods for ring networks.
\newblock \emph{Comput. Optim. Appl.}, 86:\penalty0 845--870, 2023.

\bibitem[Arag\'{o}n-Artacho et~al.(2025)Arag\'{o}n-Artacho, Campoy, and
  L\'{o}pez-Pastor]{ACL24}
F.~J. Arag\'{o}n-Artacho, R.~Campoy, and C.~L\'{o}pez-Pastor.
\newblock Forward-backward algorithms devised by graphs.
\newblock \emph{SIAM J. Optim.}, 35\penalty0 (4):\penalty0 2423--2451, 2025.

\bibitem[Bartz et~al.(2020)Bartz, Campoy, and Phan]{BCP20}
S.~Bartz, R.~Campoy, and H.~M. Phan.
\newblock Demiclosedness principles for generalized nonexpansive mappings.
\newblock \emph{J. Optim. Theory Appl.}, 186:\penalty0 759--778, 2020.

\bibitem[Bartz et~al.(2022)Bartz, Dao, and Phan]{BDP22}
S.~Bartz, M.~N. Dao, and H.~M. Phan.
\newblock Conical averagedness and convergence analysis of fixed point
  algorithms.
\newblock \emph{J. Glob. Optim.}, 82\penalty0 (2):\penalty0 351--373, 2022.

\bibitem[Bauschke and Combettes(2017)]{BC17}
H.~H. Bauschke and P.~L. Combettes.
\newblock \emph{Convex Analysis and Monotone Operator Theory in Hilbert
  Spaces}.
\newblock Springer, Cham, 2nd edition, 2017.

\bibitem[Bauschke et~al.(2021)Bauschke, Moursi, and Wang]{Bauschke2021}
H.~H. Bauschke, W.~M. Moursi, and X.~Wang.
\newblock Generalized monotone operators and their averaged resolvents.
\newblock \emph{Math. Program.}, 189\penalty0 (1):\penalty0 55--74, 2021.

\bibitem[Bredies et~al.(2024)Bredies, Chenchene, and Naldi]{BCN24}
K.~Bredies, E.~Chenchene, and E.~Naldi.
\newblock Graph and distributed extensions of the {D}ouglas--{R}achford method.
\newblock \emph{SIAM J. Optim.}, 34\penalty0 (2):\penalty0 1569--1594, 2024.

\bibitem[Campoy(2022)]{Cam22}
R.~Campoy.
\newblock A product space reformulation with reduced dimension for splitting
  algorithms.
\newblock \emph{Comput. Optim. Appl.}, 83\penalty0 (1):\penalty0 319--348,
  2022.

\bibitem[Chambolle and Pock(2011)]{Chambolle2011FirstOrderPrimalDual}
A.~Chambolle and T.~Pock.
\newblock A first-order primal-dual algorithm for convex problems with
  applications to imaging.
\newblock \emph{J. Math. Imaging Vision}, 40\penalty0 (1):\penalty0 120--145,
  2011.

\bibitem[Chen et~al.(2013)Chen, Huang, and Zhang]{ChenHuangZhang2013}
P.~Chen, J.~Huang, and X.~Zhang.
\newblock A primal--dual fixed point algorithm for convex separable
  minimization with applications to image restoration.
\newblock \emph{Inverse Problems}, 29\penalty0 (2):\penalty0 025011, 2013.

\bibitem[Combettes and Pesquet(2012)]{CombettesPesquet2012}
P.~L. Combettes and J.-C. Pesquet.
\newblock Primal--dual splitting algorithm for solving inclusions with mixtures
  of composite, {L}ipschitzian, and parallel-sum type monotone operators.
\newblock \emph{Set-Valued Var. Anal.}, 20\penalty0 (2):\penalty0 307--330,
  2012.

\bibitem[Condat(2013)]{Condat2013}
L.~Condat.
\newblock A primal--dual splitting method for convex optimization involving
  {L}ipschitzian, proximable and linear composite terms.
\newblock \emph{J. Optim. Theory Appl.}, 158:\penalty0 460--479, 2013.

\bibitem[Contino et~al.(2019)Contino, Maestripieri, and
  Marcantognini]{CONTINO2019214}
M.~Contino, A.~Maestripieri, and S.~Marcantognini.
\newblock Schur complements of selfadjoint {K}rein space operators.
\newblock \emph{Linear Algebra Appl.}, 581:\penalty0 214--246, 2019.
\newblock ISSN 0024-3795.

\bibitem[Conway(1990)]{Conway}
J.~B. Conway.
\newblock \emph{A Course in Functional Analysis}.
\newblock Springer-Verlag, 2nd edition, 1990.

\bibitem[Dao and Phan(2019)]{DP18}
M.~N. Dao and H.~M. Phan.
\newblock Adaptive {D}ouglas--{R}achford splitting algorithm for the sum of two
  operators.
\newblock \emph{SIAM J. Optim.}, 29\penalty0 (4):\penalty0 2697--2724, 2019.

\bibitem[Dao and Phan(2021)]{DP21}
M.~N. Dao and H.~M. Phan.
\newblock An adaptive splitting algorithm for the sum of two generalized
  monotone operators and one cocoercive operator.
\newblock \emph{Fixed Point Theory Algorithms Sci. Eng.}, 2021:\penalty0 Paper
  No. 16, 19 pp, 2021.

\bibitem[Dao et~al.(2025)Dao, Phan, Tam, and Truong]{DPTT25}
M.~N. Dao, H.~M. Phan, M.~K. Tam, and T.~D. Truong.
\newblock Primal-dual splitting for structured composite monotone inclusions
  with or without cocoercivity, 2025.
\newblock \href{https://arxiv.org/abs/2512.10366}{arXiv:2512.10366}.

\bibitem[Dao et~al.(2026)Dao, Tam, and Truong]{DTT26}
M.~N. Dao, M.~K. Tam, and T.~D. Truong.
\newblock A general approach to distributed operator splitting.
\newblock \emph{J. Math. Anal. Appl.}, 562\penalty0 (2):\penalty0 130692, 2026.

\bibitem[Davis and Yin(2017)]{DY17}
D.~Davis and W.~Yin.
\newblock A three-operator splitting scheme and its optimization applications.
\newblock \emph{Set-Valued Var. Anal.}, 25\penalty0 (4):\penalty0 829--858,
  2017.

\bibitem[Deng et~al.(2017)Deng, Lai, Peng, and Yin]{DengLaiPengYin2017}
Wei Deng, Ming-Jun Lai, Zhimin Peng, and Wotao Yin.
\newblock Parallel multi-block {ADMM} with {$o(1/k)$} convergence.
\newblock \emph{J. Sci. Comput.}, 71\penalty0 (2):\penalty0 712--736, 2017.

\bibitem[Douglas and Rachford(1956)]{douglas_rachford_1956}
J.~Douglas and H.~H. Rachford.
\newblock On the numerical solution of heat conduction problems in two and
  three space variables.
\newblock \emph{Trans. Amer. Math. Soc.}, 82:\penalty0 421--439, 1956.

\bibitem[Drori et~al.(2015)Drori, Sabach, and
  Teboulle]{DroriSabachTeboulle2015}
Y.~Drori, S.~Sabach, and M.~Teboulle.
\newblock A simple algorithm for a class of nonsmooth convex-concave
  saddle-point problems.
\newblock \emph{Oper. Res. Lett.}, 43\penalty0 (2):\penalty0 209--214, 2015.

\bibitem[Evens et~al.(2025{\natexlab{a}})Evens, Latafat, and Patrinos]{ELP25}
B.~Evens, P.~Latafat, and P.~Patrinos.
\newblock Convergence of the {C}hambolle--{P}ock algorithm in the absence of
  monotonicity.
\newblock \emph{J. Optim. Theory Appl.}, 206:\penalty0 7, 2025{\natexlab{a}}.

\bibitem[Evens et~al.(2025{\natexlab{b}})Evens, Pas, Latafat, and
  Patrinos]{EPLP25}
B.~Evens, P.~Pas, P.~Latafat, and P.~Patrinos.
\newblock Convergence of the preconditioned proximal point method and
  {D}ouglas--{R}achford splitting in the absence of monotonicity.
\newblock \emph{Math. Program.}, 214:\penalty0 247--301, 2025{\natexlab{b}}.

\bibitem[Kruger(1985)]{Kruger1985}
A.~Y. Kruger.
\newblock Generalized differentials of nonsmooth functions, and necessary
  conditions for an extremum.
\newblock \emph{Sib. Math. J.}, 26\penalty0 (3):\penalty0 370--379, 1985.

\bibitem[Latafat and Patrinos(2017)]{LatafatPatrinos2017}
P.~Latafat and P.~Patrinos.
\newblock Asymmetric forward-backward-adjoint splitting for solving monotone
  inclusions involving three operators.
\newblock \emph{Comput. Optim. Appl.}, 68\penalty0 (1):\penalty0 57--93, 2017.

\bibitem[Lions and Mercier(1979)]{LM79}
P.~L. Lions and B.~Mercier.
\newblock Splitting algorithms for the sum of two nonlinear operators.
\newblock \emph{SIAM J. Numer. Anal.}, 16\penalty0 (6):\penalty0 964--979,
  1979.

\bibitem[Malitsky and Tam(2023)]{MT23}
Y.~Malitsky and M.~K. Tam.
\newblock Resolvent splitting for sums of monotone operators with minimal
  lifting.
\newblock \emph{Math. Program.}, 201\penalty0 (1--2):\penalty0 231--262, 2023.

\bibitem[Moslehian et~al.(2019)Moslehian, Kian, and
  Xu]{MoslehianKianXu2019Positivity}
M.~S. Moslehian, M.~Kian, and Q.~Xu.
\newblock Positivity of {$2\times 2$} block matrices of operators.
\newblock \emph{Banach J. Math. Anal.}, 13\penalty0 (3):\penalty0 726--743,
  2019.

\bibitem[Raguet et~al.(2013)Raguet, Fadili, and Peyr\'{e}]{RFP13}
H.~Raguet, J.~Fadili, and G.~Peyr\'{e}.
\newblock A generalized forward-backward splitting.
\newblock \emph{SIAM J. Imaging Sci.}, 6\penalty0 (3):\penalty0 1199--1226,
  2013.

\bibitem[Rockafellar(1970)]{Rockafellar1970}
R.~T. Rockafellar.
\newblock On the maximality of sums of nonlinear monotone operators.
\newblock \emph{Trans. Amer. Math. Soc.}, 149:\penalty0 75--88, 1970.

\bibitem[Ryu(2020)]{ryu_2020}
E.~K. Ryu.
\newblock Uniqueness of {DRS} as the 2-operator resolvent-splitting and
  impossibility of 3-operator resolvent-splitting.
\newblock \emph{Math. Program.}, 182:\penalty0 233--273, 2020.

\bibitem[Tam(2024)]{Tam24}
M.~K. Tam.
\newblock Frugal and decentralised resolvent splittings defined by nonexpansive
  operators.
\newblock \emph{Optim. Lett.}, 18:\penalty0 1541--1559, 2024.

\bibitem[Tarcsay(2011)]{Tarcsay2011}
Z.~Tarcsay.
\newblock On operators with closed ranges.
\newblock \emph{Acta Sci. Math. (Szeged)}, 77\penalty0 (3--4):\penalty0
  579--587, 2011.

\end{thebibliography}
\end{document}